\documentclass[11pt]{amsart}
\usepackage[utf8]{inputenc}

\usepackage[margin=1in]{geometry}

\usepackage{amssymb, bm, hyperref, mathtools, enumitem, eucal, blkarray}

\usepackage[usenames,dvipsnames]{xcolor}
\usepackage{tikz}
\usetikzlibrary{knots} 

\numberwithin{equation}{section}

\newtheorem{thm}[equation]{Theorem}
\newtheorem{cor}[equation]{Corollary}
\newtheorem{lemma}[equation]{Lemma}
\newtheorem{prop}[equation]{Proposition}
\theoremstyle{definition}
\newtheorem{definition}[equation]{Definition}
\newtheorem{rem}[equation]{Remark}
\newtheorem{rems}[equation]{Remarks}
\newtheorem{examp}[equation]{Example} 

\newcommand{\CC}{\mathbb{C}}
\newcommand{\ZZ}{\mathbb{Z}}  

\newcommand{\A}{\mathbf{A}} 
\newcommand{\B}{\mathbf{B}}
\newcommand{\C}{\mathbf{C}}  
\newcommand{\J}{\mathbf{J}} 
\newcommand{\Motz}{\mathbf{M}}
\newcommand{\Pb}{\mathbf{P}\hspace{-1pt}}
\newcommand{\PR}{\mathbf{PR}}
\newcommand{\planarP}{\mathbf{PP}\hspace{-1pt}}
\newcommand{\R}{\mathbf{R}}
\newcommand{\RB}{\mathbf{RB}}
\newcommand{\Sb}{\mathbf{S}}  
\newcommand{\TL}{\mathbf{TL}}
\newcommand{\W}{\mathbf{W}}

\newcommand{\ds}{\mathsf{d}}
\newcommand{\e}{\mathsf{e}}  
\newcommand{\m}{\mathsf{m}}
\newcommand{\n}{\mathsf{n}}
\newcommand{\us}{\mathsf{u}}
\newcommand{\vs}{\mathsf{v}}

\newcommand{\BB}{\mathsf{B}}
\newcommand{\F}{\mathsf{F}}
\newcommand{\N}{\mathsf{N}}

\newcommand{\T}{\mathsf{T}} 
\newcommand{\V}{\mathsf{V}}

\newcommand{\calA}{\mathcal{A}}  
\newcommand{\calF}{\mathcal{F}}  
\newcommand{\calP}{\mathcal{P}}
\newcommand{\calW}{\mathcal{W}}

\DeclareRobustCommand{\stirling}{\genfrac\{\}{0pt}{}}
\newcommand{\Wm}{\calW^{m}_{\!\!\calA_k}}
\newcommand{\WWm}{\W^{m}_{\!\!\calA_k}}
\newcommand{\ot}{\otimes}
\newcommand{\End}{\mathsf{End}}
\newcommand{\YT}{\mathcal{YT}}
\newcommand{\SYT}{\mathcal{SYT}}
\newcommand{\SPT}{\mathcal{SPT}}
\newcommand{\SPTb}{\overline{\mathcal{SPT}}}
\newcommand{\SSPT}{\mathcal{SSPT}}
\newcommand{\tp}{\mathsf{top}}
\newcommand{\bt}{\mathsf{bot}} 
\newcommand{\pn}{\mathsf{rank}} 
\newcommand{\param}{{n}}

\title[Set-partition tableaux and representations of diagram algebras]{Set-partition tableaux and representations \\ of diagram algebras}

\author[T. Halverson]{Tom Halverson}
\address{Department of Mathematics\\ Statistics\\ and Computer Science \newline\indent
Macalester College \newline\indent
Saint Paul\\ MN 55105 (USA) }
\email{halverson@macalester.edu}

\author[T. N. Jacobson]{Theodore N. Jacobson}
\address{School of Physics and Astronomy \newline\indent
University of Minnesota \newline\indent
Minneapolis\\ MN 55455 (USA)}
\email{jaco2585@umn.edu}

\thanks{The authors gratefully acknowledge support from Simons Foundation grant 283311.}

\begin{document}

\date{\today}
\maketitle


\begin{abstract}
The partition algebra is an associative algebra with a basis of set-partition diagrams and multiplication given by diagram concatenation. It contains as subalgebras a large class of diagram algebras including the Brauer, planar partition, rook monoid, rook-Brauer, Temperley-Lieb,  Motzkin, planar rook monoid, and symmetric group algebras. 
We construct the irreducible modules of these algebras in three isomorphic ways: 
as the span of diagrams in a quotient of the left regular representation; 
as the span of  symmetric diagrams on which the algebra acts by conjugation twisted with an irreducible  symmetric group representation; 
and on a basis indexed by set-partition tableaux such that diagrams in the algebra act combinatorially on tableaux. The second representation is analogous to the Gelfand model and the third is a generalization of Young's natural representation of the symmetric group on standard tableaux. The methods of this paper work uniformly for the partition algebra and its diagram subalgebras. As an application, we express the characters of each of these algebras as nonnegative integer combinations of symmetric group characters whose coefficients count fixed points under conjugation. 
\end{abstract}

\vspace{0.5cm} 

\noindent
{\small \emph{Mathematics Subject Classification} (2010): MSC 05E10, MSC 05E18, MSC 20C15 }

\noindent
{\small \emph{Keywords}: Set partitions, tableaux, partition algebra, symmetric group, Brauer algebra, Temperley-Lieb algebra, Motzkin algebra, Rook monoid }


\section{Introduction}
\label{sec:intro}

The partition algebra $\Pb_k(\param)$ for $k \in \ZZ_{\geq 0}$ is a unital, associative algebra over $\CC$ (or any field of characteristic 0) and is semisimple for all $\param \in \CC \setminus \{0,1, \ldots, 2k-2\}$. It has a basis of set-partition diagrams and multiplication given by diagram concatenation.  This algebra arose in the work of P.P. Martin \cite{Ma0,Ma1} and V. Jones \cite{Jo} in the study of the Potts model, a $k$-site, $n$-state lattice model in statistical mechanics. For $k, n\in \ZZ_{\geq 1}$ the partition algebra $\Pb_k(n)$ and the symmetric group $\Sb_n$ are in Schur-Weyl duality on the $k$-fold tensor product $\V_n^{\otimes k}$ of the $n$-dimensional permutation module $\V_n$ of the symmetric group $\Sb_n$, and when $n \geq 2k$, $\Pb_k(n)$ is isomorphic to the centralizer algebra of $\Sb_n$ on $\V_n^{\otimes k}$. This allows information to flow back and forth between $\Pb_k(n)$ and $\Sb_n$.

The partition algebra $\Pb_k(n)$ contains as subalgebras a large class of diagram algebras including the Brauer, planar partition, rook monoid,  rook-Brauer,  Temperley-Lieb,  Motzkin,  planar rook monoid, and symmetric group algebras. Each of these subalgebras arises as the span of restricted types of set-partition diagrams (see Section~\ref{subsec:subalgebras}).  If $\A_k$ is the partition algebra or one of its diagram subalgebras, then the irreducible $\A_k$-modules can be indexed by a subset $\Lambda^{\A_k}_n \subseteq \{ \lambda \vdash n\}$ of the integer partitions of $n$. In this paper we give three explicit  constructions of the irreducible modules $\A_k^\lambda$ for $\lambda \in \Lambda^{\A_k}_n$. The first is as the span of diagrams inside a quotient of the left regular representation of $\A_k^\lambda$. The second is a combinatorial realization of the first. It is given by conjugation on a basis of symmetric $m$-diagrams (Definition~\ref{def:symmetricmdiagram}) that is twisted by a symmetric group representation. This method is analogous to the Gelfand models for diagram algebras found in \cite{HReeks} and \cite{KM-model}.  A nice feature of the construction here is that we isolate each irreducible module, rather than constructing a (multiplicity-free) sum of irreducible modules.

The third method of constructing $\A_k^\lambda$ is on a basis of set-partition tableaux. In \cite{BH-invariant1,BH-invariant2} and \cite{OZ}, it is shown that the dimension of the irreducible partition algebra module $\A^{\,\lambda}_k$ equals the number of standard set-partition tableaux of shape $\lambda$. Thus,  there should be a representation of these modules on a basis indexed by set partition tableaux, and such a construction is the main result of this paper. 
In Section \ref{sec:tableaux}, we give a combinatorial action of the diagrams in $\A_k(n)$ on these set partition tableaux and prove that it is isomorphic to the irreducible module $\A^{\,\lambda}_k$. This representation is a generalization of Young's natural representation of the symmetric group on a basis of standard Young tableaux. In fact, if $\lambda$ has $k$ boxes below the first row, when restricted to the symmetric group algebra $\CC \Sb_k \subseteq \Pb_k(n)$ we {exactly} recover Young's representation.

A surprising feature of the methods in this paper is that, by restriction, they work uniformly for the partition algebra and all of the diagram subalgebras listed above. Thus we obtain a complete set of analogs of Young's natural representation for these algebras. In the case of the non-planar algebras -- partition, Brauer, rook monoid, rook-Brauer -- we obtain new constructions of the irreducible modules on symmetric diagrams and on set-partition tableaux. In the case of the planar algebras -- planar partition, Temperley-Lieb, Motzkin, and planar rook monoid -- our methods specialize to known constructions. 

In Section~\ref{sec:characters}, we use our explicit construction of the irreducible modules on symmetric diagrams to write the irreducible characters of each $\A_k$ into a nonnegative integer sum of characters of the symmetric groups $\Sb_m$, for $0 \leq m \leq k$. We prove that if $\lambda  = [\lambda_1, \lambda_2, \ldots, \lambda_\ell] \vdash n$ with $\lambda^\ast = [\lambda_2, \lambda_3, \cdots, \lambda_\ell] \vdash m$, then the value of the irreducible $\A_k$ character on a diagram $\gamma_\kappa$ of cycle type $\kappa\vdash k$ (see~\eqref{eqn:cyclediagram}) is given by
\begin{equation}
\chi^\lambda_{\A_k}(\gamma_\kappa) = \sum_{\mu \vdash m} \F_{\!\A_k}^{\mu,\kappa}\,\chi^{\lambda^\ast}_{\Sb_m}(\gamma_\mu), \label{eqn:introchar}
\end{equation}
where $\F_{\!\A_k}^{\mu,\kappa} \in \ZZ_{\geq 0}$ and $\chi^{\lambda^\ast}_{\Sb_m}(\gamma_\mu)$ is the symmetric group character indexed by $\lambda^\ast$ on the conjugacy class of cycle type $\mu\vdash m$.  By counting fixed points under conjugation, we obtain a closed formula for the coefficients $\F_{\!\A_k}^{\mu,\kappa}$. For example, we prove in Proposition~\ref{prop:coeff} that for the partition algebra~$\Pb_k(n)$,
\begin{equation}\label{intro:coefficient}
\F_{\Pb_k(n)}^{\mu,\kappa} = \sum_{\nu | \kappa} \prod_{i} \sum_{t} \stirling{\m_i(\nu)}{t}\binom{t}{\m_i(\mu)}i^{\m_i(\nu)-t},
\end{equation}
where $\nu | \kappa$ means that $\nu$ is a divisor of $\kappa$ (see Definition~\ref{def:divisor}) and $\m_i(\nu)$ denotes the number of parts of $\nu$ equal to $i$.
In this formula $\stirling{a}{b}$ is the Stirling number of the second kind and $\binom{a}{b}$ is the binomial coefficient. The coefficient in~\eqref{intro:coefficient} specializes to the diagram subalgebras giving new character formulas for the partition, Brauer, and rook-Brauer algebras and known formulas for the rook monoid, Temperley-Lieb, Motzkin, and planar rook monoid algebras.

For further background on partition algebras see \cite{Jo}, \cite{Ma1,Ma2,Ma3}, \cite{ME}, \cite{MS,MS2}, \cite{DW}, \cite{HR}, \cite{BH-invariant2}.  Representing the irreducible modules on a basis of set-partition tableaux is new for all of these algebras. The construction on symmetric diagrams for the partition algebra is closely related to the work in \cite{MS} and \cite{DW} and for the Brauer algebra to the work in \cite{HW}. The construction of the irreducible modules of the planar algebras on symmetric diagrams is identical to the construction in the Gelfand models of \cite{HReeks} and \cite{KM-model} and is isomorphic to known representations of the Temperley-Lieb  \cite{Westbury}, Motzkin  \cite{BH-motzkin}, and planar partition \cite{FHH} algebras.
The representations constructed in this paper are different from the seminormal representations constructed for the partition \cite{Enyang}, Brauer \cite{Nazarov}, rook-Brauer \cite{dH}, rook monoid \cite{Ha-rook}, and Temperley-Lieb \cite{HMR} algebras.

\medskip\noindent
{\bf Acknowledgements} {We thank the anonymous referee who gave important suggestions that substantially improved this paper.}


\section{Partition Algebras}
\label{sec:intro}

\subsection{Set-partition diagrams}

We let $\Pi_{2k}$ denote the set of set partitions of~$\{1, \ldots, k,1',\ldots, k'\}$ and refer to the subsets of a set partition as \emph{blocks}. For example, 
\begin{equation} \label{eqn:setpartition}
\left\{1',2 \ \middle|\ 2',3' \ \middle|\ 4',1,3 \ \middle|\ 5',7' \ \middle|\ 6', 4, 7, 8 \ \middle|\ 8',6 \ \middle|\ 5 \right\}
\end{equation}
is a set partition in $\Pi_{16}$ with 7 blocks. The number of set partitions in $\Pi_{2k}$ with $t$ blocks is given by the Stirling number of the second kind $\stirling{2k}{t}$, and thus $\Pi_{2k}$ has order equal to the Bell number~$\BB(2k) = \sum_t  \stirling{2k}{t}$. 

A diagram $d$ of a set partition $\pi \in \Pi_{2k}$ consists of two rows of $k$ vertices labeled $1', \dots, k'$ on the bottom row and $1, \ldots, k$ on the top row. Edges are drawn such that the connected components of $d$ equal $\pi$. For example, the set partition in~\eqref{eqn:setpartition} is represented by
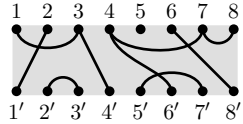
\begin{equation*}
\begin{array}{c}
\scalebox{0.75}{\begin{tikzpicture}[scale=.55,line width=1.35pt] 
\foreach \i in {1,...,8} 
{ \path (\i,2) coordinate (T\i); \path (\i,0) coordinate (B\i); } 
\filldraw[fill=gray!25,draw=gray!25,line width=4pt]  (T1) -- (T8) -- (B8) -- (B1) -- (T1);
\draw (T3) -- (B4);
\draw (T2) -- (B1);
\draw (T6) -- (B8);
\draw (T4)  .. controls +(.1,-1) and +(-.1,.7) .. (B6);
\draw (T1) .. controls +(.1,-.8) and +(-.1,-.8) .. (T3) ;
\draw (T4) .. controls +(.1,-.8) and +(-.1,-1) .. (T7) ;
\draw (T7) .. controls +(.1,-.6) and +(-.1,-.6) .. (T8) ;
\draw (B2) .. controls +(.1,.6) and +(-.1,.6) .. (B3) ;
{\draw (B5) .. controls +(.1,.8) and +(-.1,.8) .. (B7) ;}
\foreach \i in {1,...,8}
{\draw  (B\i)  node[below=0.05cm]{${\i'}$}; \draw  (T\i)  node[above=0.05cm]{${\i}$};}
\foreach \i in {1,...,8} 
{ \fill (T\i) circle (4.5pt); \fill (B\i) circle (4.5pt); } 
\end{tikzpicture}} \end{array}.
\end{equation*}
The way the edges are drawn is immaterial; what matters is that the connected components of the diagram correspond to the blocks of the set partition. Thus,  $d$  represents the equivalence class of all diagrams with connected components equal to the blocks of $\pi$. We define
\begin{equation}
\calP_{k} = \{ d \mid \text{$d$ is the diagram of a set partition in $\Pi_{2k}$} \}.
\end{equation}
Concatenation $d_1 \circ d_2$ of two diagrams $d_1$, $d_2$ is accomplished by placing $d_1$ above $d_2$,
identifying the vertices in the bottom row of $d_1$ with those in the top row of $d_2$,
concatenating  the edges,  and deleting all connected components that lie entirely in the middle row of the joined diagrams. For example, 
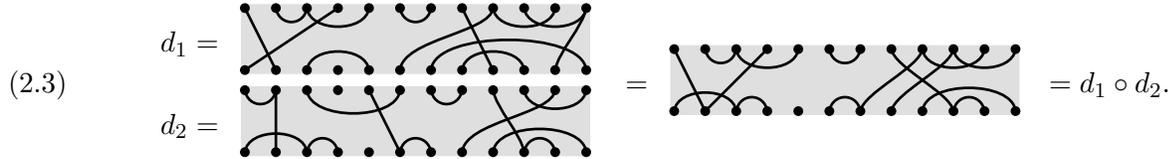
\begin{equation}\label{eq:diagrammultexample}
\begin{array}{r}
d_1 = 
\begin{array}{c}
\scalebox{0.75}{\begin{tikzpicture}[scale=.55,line width=1.35pt] 
\foreach \i in {1,...,12} 
{ \path (\i,2) coordinate (T\i); \path (\i,0) coordinate (B\i); } 
\filldraw[fill=gray!25,draw=gray!25,line width=4pt]  (T1) -- (T12) -- (B12) -- (B1) -- (T1);
\draw (T2) .. controls +(.1,-.6) and +(-.1,-.6) .. (T3) ;
\draw (T6) .. controls +(.1,-.6) and +(-.1,-.6) .. (T7) ;
\draw (T3) .. controls +(.1,-.8) and +(-.1,-.8) .. (T5) ;
\draw (T9) .. controls +(.1,-.8) and +(-.1,-.8) .. (T11) ;
\draw (T10) .. controls +(.1,-.8) and +(-.1,-.8) .. (T12) ;
\draw (B3) .. controls +(.1,.8) and +(-.1,.8) .. (B5) ;
\draw (B8) .. controls +(.1,.8) and +(-.1,.8) .. (B10) ;
\draw (B7) .. controls +(.1,1.3) and +(-.1,1.3) .. (B12) ;
\draw (B1) -- (T4);
\draw (B2) -- (T1);
\draw (B6) .. controls +(.1,1) and +(-.1,-.8) .. (T9) ;
\draw (B9) -- (T8);
\draw (B11) .. controls +(.1,.6) and +(-.1,-.6) .. (T12) ;
\foreach \i in {1,...,12}  { \fill (T\i) circle (4.5pt); \fill (B\i) circle (4.5pt); } 
\end{tikzpicture}}\end{array} \\
d_2 =
\begin{array}{c}
\scalebox{0.75}{\begin{tikzpicture}[scale=.55,line width=1.35pt] 
\foreach \i in {1,...,12} 
{ \path (\i,2) coordinate (T\i); \path (\i,0) coordinate (B\i); } 
\filldraw[fill=gray!25,draw=gray!25,line width=4pt]  (T1) -- (T12) -- (B12) -- (B1) -- (T1);
\draw (T1) .. controls +(.1,-.6) and +(-.1,-.6) .. (T2) ;
\draw (T7) .. controls +(.1,-.6) and +(-.1,-.6) .. (T8) ;
\draw (T10) .. controls +(.1,-.8) and +(-.1,-.8) .. (T12) ;
\draw (T3) .. controls +(.1,-1) and +(-.1,-1) .. (T6) ;
\draw (B3) .. controls +(.1,.6) and +(-.1,.6) .. (B4) ;
\draw (B6) .. controls +(.1,.6) and +(-.1,.6) .. (B7) ;
\draw (B10) .. controls +(.1,.6) and +(-.1,.6) .. (B11) ;
\draw (B1) .. controls +(.1,.8) and +(-.1,.8) .. (B3) ;
\draw (B9) .. controls +(.1,1) and +(-.1,1) .. (B12) ;
\draw (B2) -- (T2);
\draw (B6) -- (T5);
\draw (B8) .. controls +(.1,1) and +(-.1,-.8) .. (T11) ;
\draw (B10) .. controls +(-.1,.6) and +(.1,-.6) .. (T9) ;
\foreach \i in {1,...,12}  { \fill (T\i) circle (4.5pt); \fill (B\i) circle (4.5pt); } 
\end{tikzpicture}}\end{array} \end{array} 
= 
\begin{array}{c}  
\scalebox{0.75}{\begin{tikzpicture}[scale=.55,line width=1.35pt] 
\foreach \i in {1,...,12} 
{ \path (\i,2) coordinate (T\i); \path (\i,0) coordinate (B\i); } 
\filldraw[fill=gray!25,draw=gray!25,line width=4pt]  (T1) -- (T12) -- (B12) -- (B1) -- (T1);
\draw (T2) .. controls +(.1,-.6) and +(-.1,-.6) .. (T3) ;
\draw (T6) .. controls +(.1,-.6) and +(-.1,-.6) .. (T7) ;
\draw (T3) .. controls +(.1,-.8) and +(-.1,-.8) .. (T5) ;
\draw (T9) .. controls +(.1,-.8) and +(-.1,-.8) .. (T11) ;
\draw (T10) .. controls +(.1,-.8) and +(-.1,-.8) .. (T12) ;
\draw (B3) .. controls +(.1,.6) and +(-.1,.6) .. (B4) ;
\draw (B6) .. controls +(.1,.6) and +(-.1,.6) .. (B7) ;
\draw (B10) .. controls +(.1,.6) and +(-.1,.6) .. (B11) ;
\draw (B1) .. controls +(.1,.8) and +(-.1,.8) .. (B3) ;
\draw (B9) .. controls +(.1,1) and +(-.1,1) .. (B12) ;
\draw (T1) -- (B2) -- (T4);
\draw (B7) .. controls +(.1,.6) and +(-.1,-.6) .. (T9) ;
\draw (B8) .. controls +(.1,.6) and +(-.1,-.6) .. (T10) ;
\draw (B10) .. controls +(-.1,.6) and +(.1,-.6) .. (T8) ;
\foreach \i in {1,...,12} { \fill (T\i) circle (4.5pt); \fill (B\i) circle (4.5pt); } 
\end{tikzpicture} }
 \end{array} =  d_1 \circ d_2.  
 \end{equation} 
It is easy to confirm that concatenation depends only on the underlying set partitions and is independent of the diagrams chosen to represent them. Concatenation makes $\calP_k$ an associative monoid with identity element ${\bf 1}_k   =  
\begin{array}{c}
\scalebox{0.5}{
\begin{tikzpicture}[scale=.55,line width=1.35pt] 
\foreach \i in {1,...,6} 
{ \path (\i,1.5) coordinate (T\i); \path (\i,0) coordinate (B\i); } 
\filldraw[fill=gray!25,draw=gray!25,line width=4pt]  (T1) -- (T6) -- (B6) -- (B1) -- (T1);
\draw (T1) -- (B1);
\draw (T2) -- (B2);
\draw (T3) -- (B3);
\draw (T5) -- (B5);
\draw (T6) -- (B6);
\foreach \i in {1,2,3,5,6} { \fill (T\i) circle (4.5pt); \fill (B\i) circle (4.5pt); } 
\draw (T4) node  {$\qquad \cdots \qquad $ }; \draw (B4) node  { $\qquad \cdots\qquad $ }; 
\end{tikzpicture} } \end{array}
$
corresponding to the set partition $\left\{ 1,1' \ \middle|\ \cdots \ \middle|\ k,k'\right\}$.

Let $\Pb_0(n) = \CC$. For  $k \in \ZZ_{\geq 1}$ and $\param \in \CC$,  the \emph{partition algebra} $\Pb_k(\param)$ is the associative algebra over $\CC$ with basis~$\calP_k$,
\begin{equation}
\Pb_k(n) := \CC \calP_k = \CC\text{-span}\{d \mid \text{$d \in \calP_k$}\},
\end{equation}
such that multiplication in $\Pb_k(n)$ is defined on basis diagrams $d_1, d_2 \in \calP_k$ as
\begin{equation}
d_1 d_2 = \param^{\ell(d_1,d_2)}\, d_1 \circ d_2,
\end{equation}
where $\ell(d_1,d_2)$ is the number of connected components that were deleted from the middle row in the concatenation $d_1 \circ d_2$. For example, the product of the two diagrams in 
\eqref{eq:diagrammultexample} is $d_1 d_2 = \param^2 d_1 \circ d_2$. Since the basis of $\Pb_k(n)$ corresponds to set partitions in $\Pi_{2k}$ we have $\dim \Pb_{k}(n) = |\calP_k| = \BB(2k)$. 

The partition algebra is semisimple for all $\param \in \CC$ such that $\param \not \in \{0,1, \ldots, 2k-2\}$ (see \cite{MS2}, \cite[Thm.~3.27]{HR}), and the partition algebras $\Pb_k(n)$ are isomorphic to one another for all choices of the parameter $n$ such that $\Pb_k(n)$ is semisimple.   For this reason, we will assume 
that $n \in \ZZ$ such that $n \geq 2k$ so that we can take advantage of the Schur-Weyl duality between $\Pb_k(n)$ and $\Sb_n$ (see Section~\ref{subsec:schur-weyl}).

\subsection{Generators and relations}

For $k \in \ZZ_{\geq 1}$,  the partition algebra $\Pb_k(n)$  has a presentation by the generators
\begin{equation}
\label{s-gen}
\begin{array}{ccc}
\mathfrak{s}_i =  \!
\begin{array}{c}\scalebox{.75}{\begin{tikzpicture}[scale=.55,line width=1.35pt] 
\foreach \i in {1,...,8} 
{ \path (\i,2) coordinate (T\i); \path (\i,0) coordinate (B\i); } 
\filldraw[fill=gray!25,draw=gray!25,line width=4pt]  (T1) -- (T8) -- (B8) -- (B1) -- (T1);
\draw (T1) -- (B1);
\draw (T3) -- (B3);
\draw (T4) -- (B5);
\draw (T5) -- (B4);
\draw (T6) -- (B6);
\draw (T8) -- (B8);
\foreach \i in {1,3,4,5,6,8} { \fill (T\i) circle (4.5pt); \fill (B\i) circle (4.5pt); } 
\draw (T2) node  {$\cdots$}; \draw (B2) node  {$\cdots$}; \draw (T7) node  {$\cdots$}; \draw (B7) node  {$\cdots$}; 
\draw  (T4)  node[black,above=0.05cm]{$\scriptstyle{i}$};
\draw  (T5)  node[black,above=0.0cm]{$\scriptstyle{i+1}$};
\end{tikzpicture}}\end{array},
&
\ \mathfrak{p}_i = \!
\begin{array}{c}\scalebox{0.75}{\begin{tikzpicture}[scale=.55,line width=1.35pt] 
\foreach \i in {1,...,8} 
{ \path (\i,2) coordinate (T\i); \path (\i,0) coordinate (B\i); } 
\filldraw[fill=gray!25,draw=gray!25,line width=4pt]  (T1) -- (T8) -- (B8) -- (B1) -- (T1);
\draw (T1) -- (B1);
\draw (T3) -- (B3);
\draw (T5) -- (B5);
\draw (T6) -- (B6);
\draw (T8) -- (B8);
\foreach \i in {1,3,4,5,6,8} { \fill (T\i) circle (4.5pt); \fill (B\i) circle (4.5pt); } 
\draw (T2) node  {$\cdots$}; \draw (B2) node  {$\cdots$}; \draw (T7) node  {$\cdots$}; \draw (B7) node  {$\cdots$}; 
\draw  (T4)  node[black,above=0.05cm]{$\scriptstyle{i}$};
\end{tikzpicture}}\end{array},
&
\ \mathfrak{b}_i =  \!
\begin{array}{c}\scalebox{0.75}{\begin{tikzpicture}[scale=.55,line width=1.35pt] 
\foreach \i in {1,...,8} 
{ \path (\i,2) coordinate (T\i); \path (\i,0) coordinate (B\i); } 
\filldraw[fill=gray!25,draw=gray!25,line width=4pt]  (T1) -- (T8) -- (B8) -- (B1) -- (T1);
\draw (T1) -- (B1);
\draw (T3) -- (B3);
\draw (T4) .. controls +(.1,-.6) and +(-.1,-.6) .. (T5);
\draw (B4) .. controls +(.1,+.60) and +(-.1,+.6) .. (B5);
\draw (T4) -- (B4);
\draw (T5) -- (B5);
\draw (T6) -- (B6);
\draw (T8) -- (B8);
\foreach \i in {1,3,4,5,6,8} { \fill (T\i) circle (4.5pt); \fill (B\i) circle (4.5pt); } 
\draw (T2) node  {$\cdots$}; \draw (B2) node  {$\cdots$}; \draw (T7) node  {$\cdots$}; \draw (B7) node  {$\cdots$}; 
\draw  (T4)  node[black,above=0.05cm]{$\scriptstyle{i}$};
\draw  (T5)  node[black,above=0.0cm]{$\scriptstyle{i+1}$};
\end{tikzpicture}}\end{array} \\
\hskip.3in 1 \leq i \leq k-1 &\hskip.3in 1 \leq i \leq k &\hskip.3in 1 \leq i \leq k-1
\end{array}
\end{equation}
and the relations found in  \cite[Thm.~1.11]{HR}. 
It is useful in generating diagram subalgebras to define the elements 
$\mathfrak{e}_i =  \mathfrak{b}_i \mathfrak{p}_i \mathfrak{p}_{i+1} \mathfrak{b}_i$,  
$\mathfrak{l}_i =  \mathfrak{s}_i \mathfrak{p}_i $, and $\mathfrak{r}_i = \mathfrak{p}_i  \mathfrak{s}_i$ ,  so that
\begin{equation}
\label{extra-gen}
\begin{array}{ccc}
\mathfrak{e}_i =  \! 
\begin{array}{c}\scalebox{0.75}{\begin{tikzpicture}[scale=.55,line width=1.35pt] 
\foreach \i in {1,...,8} 
{ \path (\i,2) coordinate (T\i); \path (\i,0) coordinate (B\i); } 
\filldraw[fill=gray!25,draw=gray!25,line width=4pt]  (T1) -- (T8) -- (B8) -- (B1) -- (T1);
\draw (T1) -- (B1);
\draw (T3) -- (B3);
\draw (T4) .. controls +(.1,-.6) and +(-.1,-.6) .. (T5);
\draw (B4) .. controls +(.1,+.6) and +(-.1,+.6) .. (B5);
\draw (T6) -- (B6);
\draw (T8) -- (B8);
\foreach \i in {1,3,4,5,6,8} { \fill (T\i) circle (4.5pt); \fill (B\i) circle (4.5pt); } 
\draw (T2) node  {$\cdots$}; \draw (B2) node  {$\cdots$}; \draw (T7) node  {$\cdots$}; \draw (B7) node  {$\cdots$}; 
\draw  (T4)  node[black,above=0.05cm]{$\scriptstyle{i}$};
\draw  (T5)  node[black,above=0.0cm]{$\scriptstyle{i+1}$};
\end{tikzpicture}}\end{array},
&
\mathfrak{l}_i = \!
\begin{array}{c}\scalebox{.75}{\begin{tikzpicture}[scale=.55,line width=1.35pt] 
\foreach \i in {1,...,8} 
{ \path (\i,2) coordinate (T\i); \path (\i,0) coordinate (B\i); } 
\filldraw[fill=gray!25,draw=gray!25,line width=4pt]  (T1) -- (T8) -- (B8) -- (B1) -- (T1);
\draw (T1) -- (B1);
\draw (T3) -- (B3);
\draw (T4) -- (B5);
\draw (T6) -- (B6);
\draw (T8) -- (B8);
\foreach \i in {1,3,4,5,6,8} { \fill (T\i) circle (4.5pt); \fill (B\i) circle (4.5pt); } 
\draw (T2) node  {$\cdots$}; \draw (B2) node  {$\cdots$}; \draw (T7) node  {$\cdots$}; \draw (B7) node  {$\cdots$}; 
\draw  (T4)  node[black,above=0.05cm]{$\scriptstyle{i}$};
\draw  (T5)  node[black,above=0.0cm]{$\scriptstyle{i+1}$};
\end{tikzpicture}}\end{array},
&
\mathfrak{r}_i =  \!
\begin{array}{c}\scalebox{.75}{\begin{tikzpicture}[scale=.55,line width=1.35pt] 
\foreach \i in {1,...,8} 
{ \path (\i,2) coordinate (T\i); \path (\i,0) coordinate (B\i); } 
\filldraw[fill=gray!25,draw=gray!25,line width=4pt]  (T1) -- (T8) -- (B8) -- (B1) -- (T1);
\draw (T1) -- (B1);
\draw (T3) -- (B3);
\draw (T5) -- (B4);
\draw (T6) -- (B6);
\draw (T8) -- (B8);
\foreach \i in {1,3,4,5,6,8} { \fill (T\i) circle (4.5pt); \fill (B\i) circle (4.5pt); } 
\draw (T2) node  {$\cdots$}; \draw (B2) node  {$\cdots$}; \draw (T7) node  {$\cdots$}; \draw (B7) node  {$\cdots$}; 
\draw  (T4)  node[black,above=0.05cm]{$\scriptstyle{i}$};
\draw  (T5)  node[black,above=0.0cm]{$\scriptstyle{i+1}$};
\end{tikzpicture}}\end{array} \\
\hskip.3in 1 \leq i \leq k-1 &\hskip.3in 1 \leq i \leq k-1 &\hskip.3in 1 \leq i \leq k-1
\end{array}.
\end{equation}

\subsection{Subalgebras}
\label{subsec:subalgebras}

For $k,n \in  \ZZ_{\geq 1}$ with $n \geq 2k$ the following are semisimple subalgebras of the partition algebra $\Pb_k(n)$:
\begin{eqnarray*}
\CC\Sb_k & =&  \CC\hbox{-span}\left\{\, d\in \calP_k \ \bigg|\ 
\begin{array}{l}
\hbox{all blocks of $d$ have exactly one vertex in $\{1, \ldots k\}$} \\
\hbox{and exactly one vertex in $\{1', \ldots k'\}$} \\
\end{array}\right\}, \\
\R_k &=& \CC\hbox{-span}\left\{\, d\in \calP_k \ \bigg|\ 
\begin{array}{l}
\hbox{all blocks of $d$ have at most one vertex in $\{1, \ldots k\}$} \\
\hbox{and at most one vertex in $\{1', \ldots k'\}$} \\
\end{array}\right\}, \\
\B_k(n) & =&  \CC\hbox{-span}\{\, d\in \calP_k \ |\ \hbox{all blocks of $d$ have size 2}\}, \\
\RB_k(n) &=& \CC\hbox{-span}\{\, d\in \calP_k \ |\ \hbox{all blocks of $d$ have size 1 or 2}\}. 
\end{eqnarray*}
Here, $\CC\Sb_k$ is the group algebra of the symmetric group, $\B_k(n)$ is the Brauer algebra \cite{Br},  $\R_k$ is the rook monoid algebra \cite{So}, and $\RB_k(n)$ is the rook-Brauer algebra \cite{dH}, \cite{MM}.  

A set partition is {\it planar}  if it can be represented as a diagram without edge crossings inside of the rectangle formed
by its vertices.  The planar partition algebra \cite{Jo} is defined as
\begin{equation*}
\planarP_k(n) =  \CC\hbox{-span}\{\, d \in \calP_k \mid d \hbox{ is planar }\},
\end{equation*}
and following are the planar subalgebras of $\Pb_k(n)$, which are also semisimple:
\begin{equation*}
\begin{array}{rclcrcl}
\CC\mathbf{1}_k & =&  \CC \Sb_k \cap \planarP_k(n), & \hskip.4in &  \TL_k(n) & = &   \B_k(n) \cap \planarP_k(n), \\
\PR_k & =&   \R_k \cap \planarP_k(n), && \Motz_{k}(n) & = &  \RB_k(n) \cap \planarP_k(n).
\end{array}
\end{equation*}
Here, $\TL_k(n)$ is the Temperley-Lieb algebra \cite{TL}, $\PR_k$ is the planar rook monoid algebra \cite{FHH}, and $\Motz_{k}(n)$ is the Motzkin algebra \cite{BH-motzkin}. There is an algebra isomorphism $\planarP_k(n) \cong \TL_{2k}(n)$ (see \cite{Jo} or \cite{HR}) and we forgo discussion of the planar partition algebra in favor of the Temperley-Lieb algebra. The parameter $n$ does not arise when multiplying symmetric group diagrams (as there are never middle blocks to be removed). The following displays  examples from each of these subalgebras:
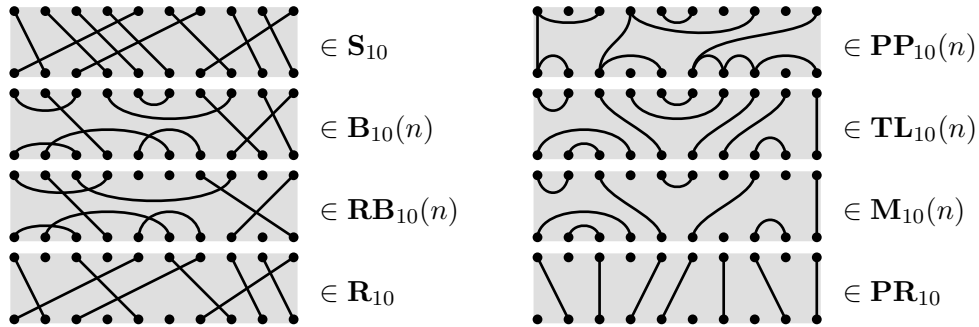
\begin{equation*}
\begin{array}{l c l}
 \begin{array}{c}
\scalebox{0.75}{
\begin{tikzpicture}[scale=.55,line width=1.35pt] 
\foreach \i in {1,...,10}  { \path (\i,1) coordinate (T\i); \path (\i,-1) coordinate (B\i); } 
\filldraw[fill=gray!25,draw=gray!25,line width=4pt]  (T1) -- (T10) -- (B10) -- (B1) -- (T1);
\draw[black] (T1) -- (B2);\
\draw[black] (T2) -- (B4);
\draw[black] (T3) -- (B5);
\draw[black] (T4) -- (B6);
\draw[black] (T5) -- (B1);
\draw[black] (T6) -- (B8);
\draw[black] (T7) -- (B3);
\draw[black] (T8) -- (B9);
\draw[black] (T9) -- (B10);
\draw[black] (T10) -- (B7);
\foreach \i in {1,...,10}  { \fill (T\i) circle (4.5pt); \fill (B\i) circle (4.5pt); } 
\end{tikzpicture}}\end{array}
 \in \Sb_{10} 
&  & 
\begin{array}{c}
\scalebox{0.75}{
\begin{tikzpicture}[scale=.55,line width=1.35pt] 
\foreach \i in {1,...,10}  { \path (\i,1) coordinate (T\i); \path (\i,-1) coordinate (B\i); } 
\filldraw[fill=gray!25,draw=gray!25,line width=4pt]  (T1) -- (T10) -- (B10) -- (B1) -- (T1);
\draw[black] (T1) .. controls +(.1,-.5) and +(-.1,-.5) .. (T3) ;
\draw[black] (T4) .. controls +(.1,-.75) and +(-.1,-1.1) .. (T8) ;
\draw[black] (T5) .. controls +(.1,-.5) and +(-.1,-.5) .. (T6) ;
\draw[black] (T1) -- (B1);
\draw[black] (T4) .. controls +(0,-1) and +(0,1) .. (B3);
\draw[black] (T10) .. controls +(0,-1) and +(0,1.5) .. (B6);
\draw[black] (B1) .. controls +(.1,.75) and +(-.1,.75) .. (B2) ;
\draw[black] (B3) .. controls +(.1,.75) and +(-.1,.75) .. (B5) ;
\draw[black] (B6) .. controls +(.1,.75) and +(-.1,.75) .. (B7) ;
\draw[black] (B7) .. controls +(.1,.75) and +(-.1,.75) .. (B8) ;
\draw[black] (B8) .. controls +(.1,.75) and +(-.1,.75) .. (B10) ;
\foreach \i in {1,...,10}  { \fill (T\i) circle (4.5pt); \fill (B\i) circle (4.5pt); } 
\end{tikzpicture}}\end{array}
 \in \planarP_{10}(n)  
 \\
\begin{array}{c}
\scalebox{0.75}{
\begin{tikzpicture}[scale=.55,line width=1.35pt] 
\foreach \i in {1,...,10}  { \path (\i,1) coordinate (T\i); \path (\i,-1) coordinate (B\i); } 
\filldraw[fill=gray!25,draw=gray!25,line width=4pt]  (T1) -- (T10) -- (B10) -- (B1) -- (T1);
\draw[black] (T7) -- (B9);
\draw[black] (T10) -- (B8);
\draw[black] (T9) -- (B10);
\draw[black] (T2) -- (B4);
\draw[black] (T1) .. controls +(.1,-.75) and +(-.1,-.75) .. (T3) ;
\draw[black] (T4) .. controls +(.1,-1.1) and +(-.1,-1.1) .. (T8) ;
\draw[black] (T5) .. controls +(.1,-.5) and +(-.1,-.5) .. (T6) ;
\draw[black] (B1) .. controls +(.1,.5) and +(-.1,.5) .. (B3) ;
\draw[black] (B5) .. controls +(.1,1.1) and +(-.1,1.1) .. (B7) ;
\draw[black] (B2) .. controls +(.1,1.1) and +(-.1,1.1) .. (B6) ;
\foreach \i in {1,...,10}  { \fill (T\i) circle (4.5pt); \fill (B\i) circle (4.5pt); } 
\end{tikzpicture}}\end{array}
 \in \B_{10}(n) &  & 
 \begin{array}{c}
\scalebox{0.75}{
\begin{tikzpicture}[scale=.55,line width=1.35pt] 
\foreach \i in {1,...,10}  { \path (\i,1) coordinate (T\i); \path (\i,-1) coordinate (B\i); } 
\filldraw[fill=gray!25,draw=gray!25,line width=4pt]  (T1) -- (T10) -- (B10) -- (B1) -- (T1);
\draw[black] (T8)  .. controls +(-.1,-.75) and +(.1,.75) .. (B6);
\draw[black] (T10) -- (B10);
\draw[black] (T9)  .. controls +(-.1,-.75) and +(.1,.75) .. (B7);
\draw[black] (T3)  .. controls +(.1,-.75) and +(-.1,.75) .. (B5);
\draw[black] (T1) .. controls +(.1,-.75) and +(-.1,-.75) .. (T2) ;
\draw[black] (T4) .. controls +(.1,-1.1) and +(-.1,-1.1) .. (T7) ;
\draw[black] (T5) .. controls +(.1,-.5) and +(-.1,-.5) .. (T6) ;
\draw[black] (B2) .. controls +(.1,.5) and +(-.1,.5) .. (B3) ;
\draw[black] (B8) .. controls +(.1,.75) and +(-.1,.75) .. (B9) ;
\draw[black] (B1) .. controls +(.1,1.1) and +(-.1,1.1) .. (B4) ;
\foreach \i in {1,...,10}  { \fill (T\i) circle (4.5pt); \fill (B\i) circle (4.5pt); } 
\end{tikzpicture}}\end{array}
 \in \TL_{10}(n)
 \\
 \begin{array}{c}
\scalebox{0.75}{
\begin{tikzpicture}[scale=.55,line width=1.35pt] 
\foreach \i in {1,...,10}  { \path (\i,1) coordinate (T\i); \path (\i,-1) coordinate (B\i); } 
\filldraw[fill=gray!25,draw=gray!25,line width=4pt]  (T1) -- (T10) -- (B10) -- (B1) -- (T1);
\draw[black] (T7) -- (B10);
\draw[black] (T10) -- (B8);
\draw[black] (T2) -- (B4);
\draw[black] (T1) .. controls +(.1,-.75) and +(-.1,-.75) .. (T4) ;
\draw[black] (T3) .. controls +(.1,-1.1) and +(-.1,-1.1) .. (T8) ;
\draw[black] (B1) .. controls +(.1,.5) and +(-.1,.5) .. (B3) ;
\draw[black] (B5) .. controls +(.1,1.1) and +(-.1,1.1) .. (B7) ;
\draw[black] (B2) .. controls +(.1,1.1) and +(-.1,1.1) .. (B6) ;
\foreach \i in {1,...,10}  { \fill (T\i) circle (4.5pt); \fill (B\i) circle (4.5pt); } 
\end{tikzpicture}}\end{array}
 \in \RB_{10}(n)   &  & 
 \begin{array}{c}
\scalebox{0.75}{
\begin{tikzpicture}[scale=.55,line width=1.35pt] 
\foreach \i in {1,...,10}  { \path (\i,1) coordinate (T\i); \path (\i,-1) coordinate (B\i); } 
\filldraw[fill=gray!25,draw=gray!25,line width=4pt]  (T1) -- (T10) -- (B10) -- (B1) -- (T1);
\draw[black] (T8)  .. controls +(-.1,-.75) and +(.1,.75) .. (B6);
\draw[black] (T10) -- (B10);
\draw[black] (T3)  .. controls +(.1,-.75) and +(-.1,.75) .. (B5);
\draw[black] (T1) .. controls +(.1,-.75) and +(-.1,-.75) .. (T2) ;
\draw[black] (T5) .. controls +(.1,-.5) and +(-.1,-.5) .. (T6) ;
\draw[black] (B2) .. controls +(.1,.5) and +(-.1,.5) .. (B3) ;
\draw[black] (B8) .. controls +(.1,.75) and +(-.1,.75) .. (B9) ;
\draw[black] (B1) .. controls +(.1,1.1) and +(-.1,1.1) .. (B4) ;
\foreach \i in {1,...,10}  { \fill (T\i) circle (4.5pt); \fill (B\i) circle (4.5pt); } 
\end{tikzpicture}}\end{array}
 \in \Motz_{10}(n)  \\
 \begin{array}{c}
\scalebox{0.75}{
\begin{tikzpicture}[scale=.55,line width=1.35pt] 
\foreach \i in {1,...,10}  { \path (\i,1) coordinate (T\i); \path (\i,-1) coordinate (B\i); } 
\filldraw[fill=gray!25,draw=gray!25,line width=4pt]  (T1) -- (T10) -- (B10) -- (B1) -- (T1);
\draw[black] (T1) -- (B2);
\draw[black] (T3) -- (B5);
\draw[black] (T5) -- (B1);
\draw[black] (T6) -- (B8);
\draw[black] (T7) -- (B3);
\draw[black] (T8) -- (B9);
\draw[black] (T9) -- (B10);
\draw[black] (T10) -- (B7);
\foreach \i in {1,...,10}  { \fill (T\i) circle (4.5pt); \fill (B\i) circle (4.5pt); } 
\end{tikzpicture}}\end{array}
 \in \R_{10} 
&  & 
 \begin{array}{c}
\scalebox{0.75}{
\begin{tikzpicture}[scale=.55,line width=1.35pt] 
\foreach \i in {1,...,10}  { \path (\i,1) coordinate (T\i); \path (\i,-1) coordinate (B\i); } 
\filldraw[fill=gray!25,draw=gray!25,line width=4pt]  (T1) -- (T10) -- (B10) -- (B1) -- (T1);
\draw[black] (T1) -- (B2);
\draw[black] (T3) -- (B3);
\draw[black] (T5) -- (B4);
\draw[black] (T6) -- (B5);
\draw[black] (T7) -- (B7);
\draw[black] (T8) -- (B9);
\draw[black] (T10) -- (B10);
\foreach \i in {1,...,10}  { \fill (T\i) circle (4.5pt); \fill (B\i) circle (4.5pt); } 
\end{tikzpicture}}\end{array}
 \in \PR_{10}
\end{array}
\end{equation*}

Each  diagram algebra $\A_k$ is generated as a unital subalgebra  $\A_k \subseteq \Pb_k(n)$ of the partition algebra using a subset of the generators 
$\mathfrak{s}_i,  \mathfrak{b}_i, \mathfrak{e}_i, \mathfrak{l}_i,\mathfrak{r}_i$ for $1 \leq i \leq k-1$ and $\mathfrak{p}_i$ for $1 \leq i \leq k$ as shown in the following table.
\begin{equation*}
\begin{array}{cl}
\text{Algebra} &  \text{Generators} \\
\hline
\Pb_k(n) &  \mathfrak{s}_i,  \mathfrak{b}_i,  \mathfrak{p}_i  \\
\CC \Sb_k & \mathfrak{s}_i  \\
\R_k & \mathfrak{s}_i,  \mathfrak{p}_i \\
\end{array}
\hskip.5in
\begin{array}{cl}
\text{Algebra} &  \text{Generators} \\
\hline
\B_k(n) & \mathfrak{s}_i,  \mathfrak{e}_i \\
\RB_k(n) & \mathfrak{s}_i,  \mathfrak{e}_i,  \mathfrak{p}_i \\
\planarP_k(n) & \mathfrak{p}_i,  \mathfrak{b}_i  \\
\end{array}
\hskip.5in
\begin{array}{cl}
\text{Algebra} &  \text{Generators} \\
\hline
\TL_k(n) & \mathfrak{e}_i \\
\Motz_k(n) & \mathfrak{e}_i, \mathfrak{l}_i,\mathfrak{r}_i  \\
\PR_k & \mathfrak{l}_i,\mathfrak{r}_i \\
\end{array}
\end{equation*}
Typically the rook monoid and planar rook monoid algebras do not have the parameter $\param$ \cite{So},\cite{Ha-rook}, and are recovered by replacing the generator $\mathfrak{p}_i$ with $\frac{1}{\param} \mathfrak{p}_i$. 

\subsection{Basic construction}
\label{sec:BasicConstruction}

Let $\A_k \subseteq \Pb_k(n)$ be the partition algebra or one of the subalgebras described in Section~\ref{subsec:subalgebras} and let $\calA_k\subseteq \calP_k$ be its diagram basis. 
There is a natural embedding of $\A_{r-1}$ as a subalgebra of $\A_r$  by placing an identity edge to the right of any diagram in $\A_{r-1}$ thus forming a tower of algebras: $\A_0 \subseteq \A_1 \subseteq \A_2 \subseteq \cdots \subseteq \A_{k-1} \subseteq \A_k.$

A block in a diagram $d\in \calA_{k}$ is a \emph{propagating block} if it contains vertices from both the top and bottom row, and the \emph{rank} (also called the propagating number) of $d$, denoted $\pn(d)$, is the number of propagating blocks of $d$. For $d_1, d_2 \in \calA_k$ we have
$
\pn(d_1\circ d_2) \leq \min(\pn(d_1),\pn(d_2)),
$
and thus the multiplication of diagrams can never increase the rank.  It follows that  
\begin{equation}\label{eq:ideal}
\J_m := \CC\text{-span}\{d \in \calA_{k}  \mid \text{$\pn(d) \leq m$}\}, \qquad 0 \leq m \leq k,
\end{equation}
is a two-sided ideal in $\A_k$ and we have the filtration 
\begin{equation}
\J_0 \subseteq \J_1 \subseteq \J_2 \subseteq \cdots \subseteq \J_{k-1} \subseteq \J_k = \A_k.
\end{equation}
In the case of the Brauer algebra $\B_k(n)$ and the Temperley-Lieb algebra $\TL_k(n)$ we have $\J_{k-1} = \J_k$,  $\J_{k-3} = \J_{k-2}$, and so on, since the rank of diagrams in these algebras have the same parity as $k$.

For each $m \ge 1$ we have
\begin{equation}\label{eq:BasicConstruction}
\A_m \cong \J_{m-1} \oplus \C_m,
\end{equation}
where $\C_m$ is the span of the diagrams of rank exactly equal to $m$. The isomorphism in ~\eqref{eq:BasicConstruction} is  the Jones basic construction for $\A_m$.  In our examples,
\begin{equation}\label{eq:planarJBC}
\begin{array}{ll}
\C_m \cong \CC \Sb_m &  \hbox{when $\A_k$ is one of the non-planar algebras  $\Pb_k(n), \B_k(n), \RB_k(n)$, or  $\R_k$}, \\
\C_m \cong \CC{\bf 1}_m & \hbox{when $\A_k$  is one of the planar algebras $\TL_k(n),  \Motz_k(n),$ or $\PR_k$.} \\
\end{array}
\end{equation}
We let ${\Gamma\!}_{\A_k}$ denote the set of possible diagram ranks in $\A_k$, so that
\begin{equation}\label{def:possiblerank}
{\Gamma\!}_{\A_k} = 
\begin{cases}
\left\{m \mid 0 \leq m \leq k\right\}, & \text{if $\A_k$ equals $\Pb_k(n), \RB_k(n), \R_k, \Motz_k(n),$ or  $\PR_k$},\\
\left\{k-2\ell  \mid  0 \leq \ell \leq \lfloor k/2 \rfloor \right\}, & \text{if $\A_k$ equals $\B_k(n)$ or $\TL_k(n)$.}
\end{cases}
\end{equation}
It follows from the basic construction that the irreducible modules of $\J_{m-1}$ are labelled by the same set as the irreducible modules for $\A_{m-1}$ (see \cite[Sec.~4.2]{HReeks}), so if
 $\Lambda^{\A_k}$ indexes the irreducible modules for $\A_k$, then  \eqref{eq:BasicConstruction} gives
\begin{equation}\label{eq:JBClabels}
\Lambda^{\A_k}   = \bigsqcup_{m  \in {\Gamma\!}_{\A_k} } \Lambda^{\C_m} =
\begin{cases} 
\displaystyle{\bigsqcup_{m  \in {\Gamma\!}_{\A_k} }} \{ \mu \vdash m\}, &  \text{if $\A_k$ is non-planar}, \\
\quad {\Gamma\!}_{\A_k}, &  \text{if $\A_k$ is planar}, \\
\end{cases}
\end{equation}
where the second equality comes from~\eqref{eq:planarJBC} and the fact that the irreducible modules for the group algebra $\CC \Sb_m$ of the symmetric group are indexed by the set $\{ \mu \vdash m\}$ of integer partitions of $m$.

\subsection{Schur-Weyl duality}
\label{subsec:schur-weyl}

For $k, n\in \ZZ_{\geq 1}$ the partition algebra $\Pb_k(n)$ and the symmetric group $\Sb_n$ are in Schur-Weyl duality on the $k$-fold tensor product $\V_n^{\otimes k}$ of the $n$-dimensional permutation module $\V_n$ of the symmetric group $\Sb_n$  (see \cite{Jo} or \cite{HR}). In particular, there is a surjective algebra homomorphism $\Pb_k(n) \to \End(\V_n^{\otimes k})$  such that the actions of $\Pb_k(n)$ and $\Sb_n$ on $\V^{\otimes k}$ commute. When $n \geq 2k$ the representation of $\Pb_k(n)$ on $\V_n^{\otimes k}$ is faithful and $\Pb_k(n) \cong \End_{\Sb_n}(\V_n^{\otimes k})$, the centralizer algebra of $\Sb_n$ on $\V_n^{\otimes k}$. 

For $n \geq 2k$, the decomposition of $\V_n^{\otimes k}$ as a bimodule for $(\Pb_k(n),\CC \Sb_n)$ is given by
\begin{equation}
\V_n^{\otimes k} \cong \bigoplus_{\lambda \in \Lambda_{k,n}} \Pb^{\,\lambda}_k \otimes \Sb_n^\lambda,
\end{equation}
where  $\Lambda_{k,n}$ indexes the irreducible $\Sb_n$ modules that appear as constituents of $\V_n^{\otimes k}$. Since irreducible $\Sb_n$ modules are indexed by partitions of $n$ we have $\Lambda_{k,n} \subseteq \{ \lambda \vdash n\},$ and it is easy to show by induction on $k$ (see, for example \cite{HR,BH-invariant2}), that
\begin{equation}
\Lambda_{k,n} = \left\{ \lambda \vdash n \mid 0 \leq |\lambda^\ast| \leq k \right\},
\end{equation}
where if $\lambda = [\lambda_1, \lambda_2, \ldots, \lambda_\ell]$ is an integer partition of $n$ then $\lambda^\ast = [\lambda_2, \ldots, \lambda_\ell]$ is the partition $\lambda$ with its first part removed as illustrated here 
\begin{equation}
\lambda = 
\begin{array}{c}
\begin{tikzpicture}[xscale=0.3, yscale=0.3]
\fill[gray!25] (0,3) rectangle (8,4);
\draw (0,0) -- (1,0) -- (1,1) -- (2,1) -- (2,2) -- (5,2) -- (5,3) -- (0,3) -- (0,0);
\draw (0,3) -- (8,3) -- (8,4) -- (0,4) -- (0,3);
\path (1,2) node {$\lambda^\ast$};
\end{tikzpicture}
\end{array}.
\end{equation}

We now have two ways to index the irreducible $\Pb_k(n)$-modules:  from the basic construction $\Lambda^{\Pb_k(n)} = \left\{ \mu \vdash m \mid 0 \leq |\mu| \leq k\right\}$ and  from Schur-Weyl duality $\Lambda_{k,n} = \left\{ \lambda \vdash n \mid 0 \leq |\lambda^\ast| \leq k\right\}$. When $n \geq 2k$, they are in bijection by identifying $\lambda \in \Lambda_{k,n}$ with $\lambda^\ast \in \Lambda^{\Pb_k(n)}$. The set-partition tableaux that we use in Section~\ref{sec:tableaux} require partitions of $n$, so we use $\Lambda_{k,n}$ for the remainder of this paper. To this end, for each $\A_k$ we add a first row of size $n-m$ to the partitions in $\Lambda^{\A_k}$ to get the partitions in $\Lambda^{\A_k}_n$ so that
\begin{equation}\label{index-set-n}
\Lambda^{\A_k}_n = \left\{ \lambda \vdash n \mid \lambda^\ast \in \Lambda^{\A_k}\right\}.
\end{equation}
These sets are given below for each of the diagram algebras. To unify our notation we view $\CC\mathbf{1}_m$ as the trivial subalgebra of $\CC\Sb_m$ and label its irreducible representation with the partition $[m]$, the index of the trivial module $\Sb_m^{[m]}$.
\begin{equation*}
\begin{array}{lll}
\A_k &  \Lambda^{\A_k}   &   \Lambda_n^{\A_k}  \\ \hline
\Pb_k(n), \RB_k(n), \R_k  &  \left\{ \mu \vdash m \mid 0 \leq m \leq k \right\} \phantom{\Big\vert} &  \left\{ \lambda \vdash n  \mid  |\lambda^\ast| = m, 0 \leq m \leq k \right\}  \\
\B_k(n) & \left\{ \mu \vdash k - 2 \ell \mid 0 \leq \ell \leq \lfloor k/2 \rfloor \right\}  &  \left\{\lambda \vdash n  \mid |\lambda^\ast| = k-2\ell, 0 \leq \ell \leq \lfloor k/2 \rfloor \right\}   \\
\Motz_k(n), \PR_k &  \left\{ m  \mid 0 \leq m \leq k \right\} \phantom{\Big\vert}  &  \left\{  [n-m,m]  \mid   0 \leq m \leq k \right\} \\
\TL_k(n)& \left\{ k - 2 \ell \mid 0 \leq \ell \leq \lfloor k/2 \rfloor \right\}  &  \left\{  [n-m,m]  \mid m = k-2\ell, 0 \leq \ell \leq \lfloor k/2 \rfloor \right\} 
\end{array}
\end{equation*}


\section{Irreducible Modules}
\label{sec:modules}
\noindent
In this section, for each $\lambda \in \Lambda^{\A_k}_n$ with $|\lambda^\ast| = m$, we identify  a copy of the irreducible $\A_k$ module indexed by $\lambda$ in the quotient $\A_k/\J_{m-1}$ of the left regular representation of $\A_k$ by the ideal $\J_{m-1}$ defined in \eqref{eq:ideal}.   We then give a combinatorial realization of this module, $\A_k^\lambda = \WWm \ot \Sb^{\lambda^\ast}_m$, where $\WWm$ is the span of symmetric $m$-diagrams in $\calA_k$ that $\A_k$ acts on by conjugation and $\Sb_m^{\lambda^\ast}$ is an irreducible symmetric group module. When a diagram $d \in \calA_k$ conjugates a symmetric $m$-diagram $w$ it permutes the $m$ fixed points of of $w$ by a permutation $\sigma_{d,w} \in \Sb_m$ which in turn acts on $\Sb_m^{\lambda^\ast}$. We view this as conjugation that is ``twisted" by the module $\Sb_m^{\lambda^\ast}$.  This construction is similar to the Gelfand model for diagram algebras in \cite{HReeks} and \cite{KM-model}.

\subsection{Symmetric group modules}
\label{subsec:symmetricgroupmodules}

For each partition $\mu \vdash m$, there is an irreducible module $\Sb_m^\mu$ for the symmetric group $\Sb_m$.  The dimension of $\Sb_m^\mu$ equals the number $f^\mu$ of standard Young tableaux of shape $\mu$, where a Young tableau of shape $\mu$ is a filling of the boxes of the diagram of $\mu$ with the numbers $1, 2, \ldots, m$ and a Young tableau is standard if the rows increase from left to right and the columns increase from top to bottom. We let $\SYT(\mu)$ denote the set of standard Young tableaux of shape $\mu$. For example, there are five standard Young tableaux of shape $\mu =[3,2]$:
\begin{equation*}\label{SYT32}
\SYT([3,2]) = \left\{
t_1 =\!\!\!\! \begin{array}{c}
\scalebox{0.95}{
\begin{tikzpicture}[scale=.45]
\draw (2,2) -- (2,0) -- (0,0) -- (0,2) -- (3,2) -- (3,1) -- (0,1);
\draw (1,0) -- (1,2);
\path (0.5,1.5) node {$1$}; \path (1.5,1.5) node {$3$}; \path (2.5,1.5) node {$5$};
\path (0.5,0.5) node {$2$}; \path (1.5,0.5) node {$4$};
\end{tikzpicture}\!} 
\end{array}, \
t_2 =\!\!\!\! \begin{array}{c}
\scalebox{0.95}{
\begin{tikzpicture}[scale=.45]
\draw (2,2) -- (2,0) -- (0,0) -- (0,2) -- (3,2) -- (3,1) -- (0,1);
\draw (1,0) -- (1,2);
\path (0.5,1.5) node {$1$}; \path (1.5,1.5) node {$3$}; \path (2.5,1.5) node {$4$};
\path (0.5,0.5) node {$2$}; \path (1.5,0.5) node {$5$};
\end{tikzpicture}\!} 
\end{array}, \
t_3 =\!\!\!\! \begin{array}{c}
\scalebox{0.95}{
\begin{tikzpicture}[scale=.45]
\draw (2,2) -- (2,0) -- (0,0) -- (0,2) -- (3,2) -- (3,1) -- (0,1);
\draw (1,0) -- (1,2);
\path (0.5,1.5) node {$1$}; \path (1.5,1.5) node {$2$}; \path (2.5,1.5) node {$5$};
\path (0.5,0.5) node {$3$}; \path (1.5,0.5) node {$4$};
\end{tikzpicture}\!} 
\end{array}, \
t_4 =\!\!\!\! \begin{array}{c}
\scalebox{0.95}{
\begin{tikzpicture}[scale=.45]
\draw (2,2) -- (2,0) -- (0,0) -- (0,2) -- (3,2) -- (3,1) -- (0,1);
\draw (1,0) -- (1,2);
\path (0.5,1.5) node {$1$}; \path (1.5,1.5) node {$2$}; \path (2.5,1.5) node {$4$};
\path (0.5,0.5) node {$3$}; \path (1.5,0.5) node {$5$};
\end{tikzpicture}\!} 
\end{array}, \
t_5 =\!\!\!\! \begin{array}{c}
\scalebox{0.95}{
\begin{tikzpicture}[scale=.45]
\draw (2,2) -- (2,0) -- (0,0) -- (0,2) -- (3,2) -- (3,1) -- (0,1);
\draw (1,0) -- (1,2);
\path (0.5,1.5) node {$1$}; \path (1.5,1.5) node {$2$}; \path (2.5,1.5) node {$3$};
\path (0.5,0.5) node {$4$}; \path (1.5,0.5) node {$5$};
\end{tikzpicture}\!} 
\end{array}
\right\}.
\end{equation*}
The \emph{column-reading tableau} $t_c$ (resp., row-reading tableau $t_r$) is the standard tableau obtained by entering the numbers $1,2,...,m$ consecutively down the columns (across the rows) of $\mu$. In the example above $t_1 = t_c$ and $t_5 = t_r$.

For $\mu \vdash m$, define the \emph{Young symmetrizer} (see, for example \cite[1.5.4]{JK}), 
\begin{equation}
p_\mu =  \sum_{\gamma \in C(t_c)}\sum_{\rho \in R(t_c)} \mathsf{sign}(\gamma)  \gamma \rho \in \CC \Sb_m, 
\end{equation}
where $C(t_c)$ and $R(t_c)$ are the row and column group of $t_c$, respectively, and $\mathsf{sign}(\gamma)$ is the sign of the permutation $\gamma$. 
That is $C(t_c) \subseteq \Sb_m$ is the subgroup of permutations that preserve the rows of $t_c$ and $R(t_c)$ is the subgroup that preserves the columns. Then $(\CC \Sb_m) p_\mu$ is a copy of the irreducible module $\Sb_m^\mu$ in the left regular representation $\CC \Sb_m$. 

For any Young tableau $t$ of shape $\mu$ let $\sigma_{t} \in \Sb_m$ be the the permutation defined by $\sigma_{t} (t_c) = t$. Then a basis of $\Sb_m^\mu$ is given by $\{\n_t :=  \sigma_t p_\mu \mid t \in \SYT(\mu) \}$  (see, for example, \cite{GM} for a proof of this classical result).  If $t \in \SYT(\mu)$ and $\pi \in \Sb_m$, then
\begin{equation}
\pi \n_t = \pi \sigma_t p_\mu = \sigma_{\pi(t)} p_\mu = \n_{\pi(t)}.
\end{equation}
If $\pi(t)$ is a standard tableau, then $\n_{\pi(t)}$ is another basis element of $\Sb_m^\mu$; otherwise,  $\n_{\pi(t)}$ can be expanded as an \emph{integer} linear combination of basis elements (i.e., indexed by standard tableaux) using a straightening algorithm such as tableaux intersection \cite{GM} or Garnir relations (see, for example, \cite{JK}, \cite{sagan2001symmetric}, \cite{Ra-skew}). The basis $\{\n_t =  \sigma_t p_\mu \mid t \in \SYT(\mu) \}$ is Young's natural basis of $\Sb_m^\mu$.

\subsection{Irreducible $\A_k$ modules in the regular representation}
\label{subsec:symmetricdiagrams}

For $0 \le m \le k$, recall the definition of the ideal $\J_m$ from \eqref{eq:ideal} and define the quotient map 
\begin{equation}
\begin{array}{cccc}
\Psi_{k,m}: & \A_k & \longrightarrow & \A_k/\J_m \\
& a & \mapsto & a + \J_m
\end{array},
\end{equation}
which is a surjective algebra homomorphism. 
For $0 \le m \le k$, define
\begin{equation}
\e_m = \begin{cases}
\frac{1}{n^{k-m}} \mathfrak{p}_{m+1} \cdots \mathfrak{p}_k, & \text{if $\A_k$ equals $\Pb_k(n), \RB_k(n), \R_k, \Motz_k(n),$ or  $\PR_k$}, \\
\frac{1}{n^{\ell}} \mathfrak{e}_{m+1}  \mathfrak{e}_{m+3} \cdots \mathfrak{e}_{k-1}, & 
 \text{if $\A_k$ equals $\B_k(n)$ or $\TL_k(n)$  and $m = k - 2 \ell$.}\\
\end{cases}
\end{equation}
so that, for example, 
\begin{align*}
 \e_4 &= \frac{1}{n^4}
\begin{array}{c}
\scalebox{1}{
\begin{tikzpicture}[xscale=.55,yscale=.55,line width=1.25pt] 
\foreach \i in {1,...,8} 
{ \path (\i,1.5) coordinate (T\i); \path (\i,0) coordinate (B\i); } 
\filldraw[fill= gray!30,draw=gray!30,line width=5pt]  (T1) -- (T8) -- (B8) -- (B1) -- (T1);
\draw[black] (T1) -- (B1) ;
\draw[black] (T2) -- (B2) ;
\draw[black] (T3) -- (B3) ;
\draw[black] (T4) -- (B4) ;
\foreach \i in {1,...,8} 
{ \fill (T\i) circle (4pt); \fill (B\i) circle (4pt); }  
\end{tikzpicture} }
\end{array} \in \Pb_{8}(n)\, \RB_8(n), \R_8, \Motz_8(n), \hbox{ or  } \PR_8, \\
 \e_4 & = \frac{1}{n^2}
\begin{array}{c}
\scalebox{1}{
\begin{tikzpicture}[xscale=.55,yscale=.55,line width=1.25pt] 
\foreach \i in {1,...,8} 
{ \path (\i,1.5) coordinate (T\i); \path (\i,0) coordinate (B\i); } 
\filldraw[fill= gray!30,draw=gray!30,line width=5pt]  (T1) -- (T8) -- (B8) -- (B1) -- (T1);
\draw[black] (T1) -- (B1) ;
\draw[black] (T2) -- (B2) ;
\draw[black] (T3) -- (B3) ;
\draw[black] (T4) -- (B4) ;
\draw[black] (T5) .. controls +(.1,-.5) and +(-.1,-.5) .. (T6);
\draw[black] (T7) .. controls +(.1,-.5) and +(-.1,-.5) .. (T8);
\draw[black] (B5) .. controls +(.1,.5) and +(-.1,.5) .. (B6);
\draw[black] (B7) .. controls +(.1,.5) and +(-.1,.5) .. (B8);
\foreach \i in {1,...,8} 
{ \fill (T\i) circle (4pt); \fill (B\i) circle (4pt); }  
\end{tikzpicture} }
\end{array} \in \B_{8}(n) \hbox{ or } \TL_8(n).
\end{align*}

For $\lambda \in \Lambda_{n}^{\A_k}$ with $|\lambda^\ast| = m$, define 
\begin{equation}
\e_\lambda = p_{\lambda^\ast} \e_m = \e_m p_{\lambda^\ast}.
\end{equation} 
If $n \in \CC$ is chosen such that $\A_k$ is semisimple, then the following theorem tells us that $\e_\lambda$ is the minimal idempotent corresponding to the irreducible $\A_k$-module indexed by $\lambda$. The proof in \cite{HR} is for $\A_k = \Pb_k(n)$ but it extends without alteration to the other diagram subalgebras using \eqref{eq:BasicConstruction} and \eqref{eq:planarJBC}.

\begin{thm}\label{thm:HRirreducible} \cite[Prop.~2.43]{HR} If $n \in \CC$ such that $\A_k$ is semisimple and  $\lambda \in \Lambda_{n}^{\A_k}$ with $|\lambda^\ast| = m$, then
\begin{equation*}
\Psi_{k,m}(\A_k \e_\lambda) = (\A_k \e_\lambda)/\J_{m-1}
\end{equation*}
is the irreducible $\A_k(n)$ module indexed by $\lambda$.
\end{thm}

We now construct an explicit diagram basis of $(\A_k \e_\lambda)/\J_{m-1}$.

\begin{definition} An \emph{$m$-factor} is a diagram $d \in \calA_k$ such that the following hold: (1) $\pn(d) = m$; (2)  the first $m$ vertices in the bottom row of $d$ propagate; (3a) the last $k-m$ vertices in the bottom row of $d$ are isolated if $\A_k$ equals $\Pb_k(n), \RB_k(n), \R_k, \Motz_k(n),$  $\PR_k$; and (3b) the last $k-m$ vertices in the bottom row of $d$ are are paired with their neighbor if $\A_k$ equals $\B_k(n)$ or $\TL_k(n)$ and $n-m$ is even.  Furthermore, we say that an $m$-factor is \emph{noncrossing} if the propagating edges of $d$ do not cross when $d$ is drawn in such a way that \emph{the propagating edges connect to the rightmost vertex of the block in the top row}.
\end{definition}

An $m$-factor $d$ has a unique decomposition $d = \omega \sigma$ such that $\omega$ is a noncrossing $m$-factor and $\sigma \in \Sb_m$, as illustrated in the following example, which holds in $\Pb_{10}(n)$,
\begin{align}\label{eg:mfactor}
d = 
\begin{array}{c}
\scalebox{1}{
\begin{tikzpicture}[xscale=.55,yscale=.55,line width=1.25pt] 
\foreach \i in {1,...,10} 
{ \path (\i,1.5) coordinate (T\i); \path (\i,0) coordinate (B\i); } 
\filldraw[fill= gray!30,draw=gray!30,line width=5pt]  (T1) -- (T10) -- (B10) -- (B1) -- (T1);
\draw[black] (T1) .. controls +(.1,-.4) and +(-.1,-.4) .. (T2) ;
\draw[black] (T2) -- (B2) ;
\draw[black] (T3) .. controls +(.1,-.4) and +(-.1,-.4) .. (T4) ;
\draw[black] (T4) .. controls +(.1,-.5) and +(-.1,-.5) .. (T6) ;
\draw[black] (T6) .. controls +(.1,-.6) and +(-.1,.4) .. (B4) ;
\draw[black] (T7) .. controls +(.1,-0.8) and +(-.1,1) .. (B3) ;
\draw[black] (T8) .. controls +(.1,-.5) and +(-.1,-.5) .. (T10) ;
\draw[black] (T9) .. controls +(.1,-1.5) and +(-.1,1) .. (B1) ;
\foreach \i in {1,...,10} 
{ \fill (T\i) circle (4pt); \fill (B\i) circle (4pt); } 
\end{tikzpicture} }
\end{array}
= 
\begin{array}{c}\begin{array}{c}
\scalebox{1}{
\begin{tikzpicture}[xscale=.55,yscale=.55,line width=1.25pt] 
\foreach \i in {1,...,10} 
{ \path (\i,1.5) coordinate (T\i); \path (\i,0) coordinate (B\i); } 
\filldraw[fill= gray!30,draw=gray!30,line width=5pt]  (T1) -- (T10) -- (B10) -- (B1) -- (T1);
\draw[black] (T1) .. controls +(.1,-.4) and +(-.1,-.4) .. (T2) ;
\draw[black] (T2) -- (B1) ;
\draw[black] (T3) .. controls +(.1,-.4) and +(-.1,-.4) .. (T4) ;
\draw[black] (T4) .. controls +(.1,-.5) and +(-.1,-.5) .. (T6) ;
\draw[black] (T6) .. controls +(.1,-1) and +(-.1,1) .. (B2) ;
\draw[black] (T7) .. controls +(.1,-1) and +(-.1,.7) .. (B3) ;
\draw[black] (T8) .. controls +(.1,-.5) and +(-.1,-.5) .. (T10) ;
\draw[black] (T9) .. controls +(.1,-1.5) and +(-.1,1) .. (B4) ;
\foreach \i in {1,...,10} 
{ \fill (T\i) circle (4pt); \fill (B\i) circle (4pt); } 
\end{tikzpicture} }\end{array} = \omega
\\
\begin{array}{c}\scalebox{1}{
\begin{tikzpicture}[xscale=.55,yscale=.55,line width=1.25pt] 
\foreach \i in {1,...,10} 
{ \path (\i,1.5) coordinate (T\i); \path (\i,0) coordinate (B\i); } 
\filldraw[fill= gray!30,draw=gray!30,line width=5pt]  (T1) -- (T10) -- (B10) -- (B1) -- (T1);
\draw[black] (T1)-- (B2) ;
\draw[black] (T2) -- (B4) ;
\draw[black] (T3).. controls +(.2,-1) and +(.2,1) .. (B3) ;
\draw[black] (T4) -- (B1) ;
\draw[black] (T5) -- (B5) ;
\draw[black] (T6) -- (B6) ;
\draw[black] (T7) -- (B7) ;
\draw[black] (T8) -- (B8) ;
\draw[black] (T9) -- (B9) ;
\draw[black] (T10) -- (B10) ;
\foreach \i in {1,...,10} 
{ \fill (T\i) circle (4pt); \fill (B\i) circle (4pt); } 
\end{tikzpicture} }\end{array} = \sigma
\end{array},
\end{align}
and in the following example, which holds in $\B_{10}(n)$,
\begin{align}
d = 
\begin{array}{c}
\scalebox{1}{
\begin{tikzpicture}[xscale=.55,yscale=.55,line width=1.25pt] 
\foreach \i in {1,...,10} 
{ \path (\i,1.5) coordinate (T\i); \path (\i,0) coordinate (B\i); } 
\filldraw[fill= gray!30,draw=gray!30,line width=5pt]  (T1) -- (T10) -- (B10) -- (B1) -- (T1);
\draw[black] (T1) .. controls +(.1,-.4) and +(-.1,-.4) .. (T3) ;
\draw[black] (T2) -- (B2) ;
\draw[black] (T4) .. controls +(.1,-.5) and +(-.1,-.5) .. (T5) ;
\draw[black] (T6) .. controls +(.1,-.6) and +(-.1,.4) .. (B4) ;
\draw[black] (T7) .. controls +(.1,-0.8) and +(-.1,1) .. (B3) ;
\draw[black] (T8) .. controls +(.1,-.5) and +(-.1,-.5) .. (T10) ;
\draw[black] (T9) .. controls +(.1,-1.5) and +(-.1,1) .. (B1) ;
\draw[black] (B5) .. controls +(.1,.5) and +(-.1,.5) .. (B6);
\draw[black] (B7) .. controls +(.1,.5) and +(-.1,.5) .. (B8);
\draw[black] (B9) .. controls +(.1,.5) and +(-.1,.5) .. (B10);
\foreach \i in {1,...,10} 
{ \fill (T\i) circle (4pt); \fill (B\i) circle (4pt); } 
\end{tikzpicture} }
\end{array}
= 
\begin{array}{c}\begin{array}{c}
\scalebox{1}{
\begin{tikzpicture}[xscale=.55,yscale=.55,line width=1.25pt] 
\foreach \i in {1,...,10} 
{ \path (\i,1.5) coordinate (T\i); \path (\i,0) coordinate (B\i); } 
\filldraw[fill= gray!30,draw=gray!30,line width=5pt]  (T1) -- (T10) -- (B10) -- (B1) -- (T1);
\draw[black] (T1) .. controls +(.1,-.4) and +(-.1,-.4) .. (T3) ;
\draw[black] (T2) -- (B1) ;
\draw[black] (T4) .. controls +(.1,-.5) and +(-.1,-.5) .. (T5) ;
\draw[black] (T6) .. controls +(.1,-1) and +(-.1,1) .. (B2) ;
\draw[black] (T7) .. controls +(.1,-1) and +(-.1,.7) .. (B3) ;
\draw[black] (T8) .. controls +(.1,-.5) and +(-.1,-.5) .. (T10) ;
\draw[black] (T9) .. controls +(.1,-1.5) and +(-.1,1) .. (B4) ;
\draw[black] (B5) .. controls +(.1,.5) and +(-.1,.5) .. (B6);
\draw[black] (B7) .. controls +(.1,.5) and +(-.1,.5) .. (B8);
\draw[black] (B9) .. controls +(.1,.5) and +(-.1,.5) .. (B10);
\foreach \i in {1,...,10} 
{ \fill (T\i) circle (4pt); \fill (B\i) circle (4pt); } 
\end{tikzpicture} }\end{array} = \omega
\\
\begin{array}{c}\scalebox{1}{
\begin{tikzpicture}[xscale=.55,yscale=.55,line width=1.25pt] 
\foreach \i in {1,...,10} 
{ \path (\i,1.5) coordinate (T\i); \path (\i,0) coordinate (B\i); } 
\filldraw[fill= gray!30,draw=gray!30,line width=5pt]  (T1) -- (T10) -- (B10) -- (B1) -- (T1);
\draw[black] (T1)-- (B2) ;
\draw[black] (T2) -- (B4) ;
\draw[black] (T3) -- (B3) ;
\draw[black] (T4) .. controls +(0,-1.5) and +(.2,.5) ..  (B1) ;
\draw[black] (T5) -- (B5) ;
\draw[black] (T6) -- (B6) ;
\draw[black] (T7) -- (B7) ;
\draw[black] (T8) -- (B8) ;
\draw[black] (T9) -- (B9) ;
\draw[black] (T10) -- (B10) ;
\foreach \i in {1,...,10} 
{ \fill (T\i) circle (4pt); \fill (B\i) circle (4pt); } 
\end{tikzpicture} }\end{array} = \sigma
\end{array}.
\end{align}
We let $\mathcal{N}_{\calA_k}^m$ denote the set of all noncrossing $m$-factors in $\calA_k$.  The following proposition is proved for the partition algebra in  \cite[Prop.~2.1]{DW} and for the Brauer algebra in \cite[Prop.~2.1]{DWH}. Here we prove it simultaneously for all of the diagram algebras of this paper.

\begin{prop}  If $n \in \CC$ such that $\A_k$ is semisimple, then $\lambda \in \Lambda_{n}^{\A_k}$ with $|\lambda^\ast| = m$, then the set $\{\omega \sigma_t p_{\lambda^\ast} + \J_{m-1} \mid \omega \in \mathcal{N}_{\calA_k}^m, t \in \SYT(\lambda^\ast)\}$ is a $\CC$-basis for the $\A_k$-module $(\A_k \e_\lambda)/\J_{m-1}$.
\end{prop} 

\begin{proof}
Let $d \in \calA_k$ and consider the element $d \e_\lambda + \J_{m-1}  = d \e_m p_{\lambda^\ast} + \J_{m-1}$ of the quotient space $(\A_k \e_\lambda)/\J_{m-1}$. Either $\pn(d \e_m) < m$ or $d \e_m = n^\ell d'$, where $d'$ is an $m$-factor and $\ell \in \ZZ_{\ge 0}$. In the latter case, $d \e_m = n^\ell \omega' \sigma'$ for  $\omega' \in \mathcal{N}_{\calA_k}^m$ and $\sigma' \in \Sb_m$. It follows that $d \e_m p_{\lambda^\ast}$ is 0 mod $\J_{m-1}$ or $d \e_m p_{\lambda^\ast} = n^\ell \omega' \sigma' p_{\lambda^\ast}$, and $\sigma' p_{\lambda^\ast}$ is expressible as a linear combination of $\sigma_t p_{\lambda^\ast}$ by the fact that $\{\sigma_t p_{\lambda^\ast} \mid t \in \SYT(\lambda^\ast)\}$ is a basis of the $\Sb_m$-module $\Sb_m^{\lambda^\ast}$. 
\end{proof}

\subsection{Symmetric diagrams}
\label{subsec:symmetricdiagrams}

In this section we reinterpret the basis of the previous section as a basis on symmetric diagrams which is simpler and is  better suited for the combinatorial computations in the remainder of this paper. 

For $d \in \calA_k$,  let $d^T \in \calA_k$ be the diagram obtained by reflecting $d$ over the horizontal axis. We say that a diagram is \emph{symmetric} if $d = d^T$. For example, the following are symmetric diagrams in $\calP_{10}$,
\begin{equation*}
d_1=\!\!\! \begin{array}{c}
\scalebox{0.75}{
\begin{tikzpicture}[scale=.55,line width=1.35pt] 
\foreach \i in {1,...,10}  { \path (\i,1) coordinate (T\i); \path (\i,-1) coordinate (B\i); } 
\filldraw[fill=gray!25,draw=gray!25,line width=4pt]  (T1) -- (T10) -- (B10) -- (B1) -- (T1);
\draw[black] (T1) .. controls +(.1,-.8) and +(-.1,-.8) .. (T3) ;
\draw[black] (T4) .. controls +(.1,-1.1) and +(-.1,-1.1) .. (T8) ;
\draw[black] (T5) .. controls +(.1,-.6) and +(-.1,-.6) .. (T6) ;
\draw[black] (B1) .. controls +(.1,.8) and +(-.1,.8) .. (B3) ;
\draw[black] (B4) .. controls +(.1,1.1) and +(-.1,1.1) .. (B8) ;
\draw[black] (B5) .. controls +(.1,.6) and +(-.1,.6) .. (B6) ;
\draw (T2) -- (B2);
\draw (T7) -- (B10); \draw (B7) -- (T10);
\draw (T9) -- (B9);
\foreach \i in {1,...,10}  { \fill (T\i) circle (4.5pt); \fill (B\i) circle (4.5pt); } 
\end{tikzpicture}}\end{array},\ 
d_2=\!\! \!\begin{array}{c}
\scalebox{0.75}{
\begin{tikzpicture}[scale=.55,line width=1.35pt]
\foreach \i in {1,...,10} 
{ \path (\i,2) coordinate (T\i);
 \path (\i,0) coordinate (B\i);}
\filldraw[fill=gray!25,draw=gray!25,line width=4pt]  (T1) -- (T10) -- (B10) -- (B1) -- cycle;
\draw (T3) .. controls +(.1,-.8) and +(-.1,-.8) .. (T5);
\draw (B3) .. controls +(.1,.8) and +(-.1,.8) .. (B5);
\draw (T5) .. controls +(.1,-.6) and +(-.1,-.6) .. (T6);
\draw (B5) .. controls +(.1,.6) and +(-.1,.6) .. (B6);
\draw (T1) .. controls +(.1,-.6) and +(-.1,-.6) .. (T2);
\draw (B1) .. controls +(.1,.6) and +(-.1,.6) .. (B2);
\draw (T8) .. controls +(.1,-.6) and +(-.1,-.6) .. (T9);
\draw (B8) .. controls +(.1,.6) and +(-.1,.6) .. (B9);
\draw (T9) .. controls +(.1,-.6) and +(-.1,-.6) .. (T10);
\draw (B9) .. controls +(.1,.6) and +(-.1,.6) .. (B10);
\draw (T2) -- (B2);
\draw (T4) -- (B4);
\draw (T6) -- (B8); \draw (T8) -- (B6);
\foreach \i in {1,...,10} { \fill (T\i) circle (4.5pt);   \fill (B\i) circle (4.5pt);} 
\end{tikzpicture}} 
\end{array},  \ 
d_3=\!\!\!\begin{array}{c}
\scalebox{0.75}{
\begin{tikzpicture}[scale=.55,line width=1.35pt]
\foreach \i in {1,...,10} 
{ \path (\i,2) coordinate (T\i);
 \path (\i,0) coordinate (B\i);}
\filldraw[fill=gray!25,draw=gray!25,line width=4pt]  (T1) -- (T10) -- (B10) -- (B1) -- cycle;
\draw (T3) .. controls +(.1,-.8) and +(-.1,-.8) .. (T5);
\draw (B3) .. controls +(.1,.8) and +(-.1,.8) .. (B5);
\draw (T5) .. controls +(.1,-.6) and +(-.1,-.6) .. (T6);
\draw (B5) .. controls +(.1,.6) and +(-.1,.6) .. (B6);
\draw (T1) .. controls +(.1,-.6) and +(-.1,-.6) .. (T2);
\draw (B1) .. controls +(.1,.6) and +(-.1,.6) .. (B2);
\draw (T8) .. controls +(.1,-.6) and +(-.1,-.6) .. (T9);
\draw (B8) .. controls +(.1,.6) and +(-.1,.6) .. (B9);
\draw (T9) .. controls +(.1,-.6) and +(-.1,-.6) .. (T10);
\draw (B9) .. controls +(.1,.6) and +(-.1,.6) .. (B10);
\draw (T2) -- (B2);
\draw (T4) -- (B4);
\draw (T6) -- (B6); 
\draw (T10) -- (B10);
\foreach \i in {1,...,10} { \fill (T\i) circle (4.5pt);   \fill (B\i) circle (4.5pt);} 
\end{tikzpicture}} 
\end{array}.
\end{equation*}
For a symmetric diagram $d = d^T$, let $\pi(d)$ and  $\pi'(d)$ denote the propagating blocks in the top and bottom rows of $d$, respectively.  In the examples above, $\pi(d_1) = \{ 2 \mid 7 \mid 9 \mid 10\}$ and $\pi(d_2) = \pi(d_3) = \{ 1,2 \mid 4 \mid 3,5,6 \mid 8,9,10 \}$, and observe that in a symmetric diagram $\pi'(d)$ is always equal to $\pi(d)$ with the vertices primed. 
\begin{definition}\label{def:symmetricmdiagram} 
A diagram $d \in \calA_k$ is a \emph{symmetric $m$-diagram} if (1) $d$ is symmetric; (2) $\pn(d) = m$; and (3) each of the $m$ propagating blocks in $\pi(d)$ is connected to its mirror image in $\pi'(d)$. 
\end{definition}
In the examples above, $d_3$ is a symmetric 4-diagram, but $d_1$ and $d_2$ are not since they each have  a propagating block not connected to its mirror image. For any of the diagram algebras $\A_k$, let
\begin{equation}
\Wm = \left\{ d \in \calA_k \ \middle| \ \text{$d$ is a symmetric $m$-diagram} \right\}.
\end{equation}
There is a simple bijection between the noncrossing $m$-factors of the previous section and the symmetric $m$-diagrams of this section. This is seen by the fact that both types of diagrams are completely determined by the set partition on their top row and by knowing which blocks propagate. We simply pair the diagrams with the same top row and propagating edges. For example,  the noncrossing $m$-factor of \eqref{eg:mfactor} is paired with a symmetric $m$-diagram as follows,
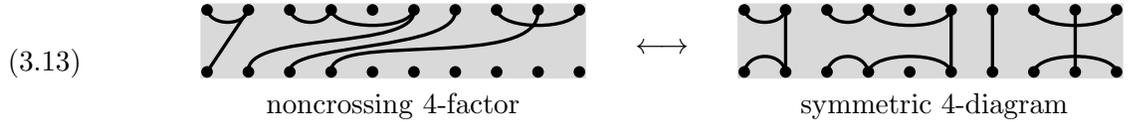
\begin{equation}\label{bijection:factors}
\begin{array}{ccc}
\begin{array}{c}
\scalebox{1}{
\begin{tikzpicture}[xscale=.55,yscale=.55,line width=1.25pt] 
\foreach \i in {1,...,10} 
{ \path (\i,1.5) coordinate (T\i); \path (\i,0) coordinate (B\i); } 
\filldraw[fill= gray!30,draw=gray!30,line width=5pt]  (T1) -- (T10) -- (B10) -- (B1) -- (T1);
\draw[black] (T1) .. controls +(.1,-.4) and +(-.1,-.4) .. (T2) ;
\draw[black] (T2) -- (B1) ;
\draw[black] (T3) .. controls +(.1,-.4) and +(-.1,-.4) .. (T4) ;
\draw[black] (T4) .. controls +(.1,-.5) and +(-.1,-.5) .. (T6) ;
\draw[black] (T6) .. controls +(.1,-1) and +(-.1,1) .. (B2) ;
\draw[black] (T7) .. controls +(.1,-1) and +(-.1,.7) .. (B3) ;
\draw[black] (T8) .. controls +(.1,-.5) and +(-.1,-.5) .. (T10) ;
\draw[black] (T9) .. controls +(.1,-1.5) and +(-.1,1) .. (B4) ;
\foreach \i in {1,...,10} 
{ \fill (T\i) circle (4pt); \fill (B\i) circle (4pt); } 
\end{tikzpicture} }\end{array}
&\longleftrightarrow&
\begin{array}{c}
\scalebox{1}{
\begin{tikzpicture}[xscale=.55,yscale=.55,line width=1.25pt] 
\foreach \i in {1,...,10} 
{ \path (\i,1.5) coordinate (T\i); \path (\i,0) coordinate (B\i); } 
\filldraw[fill= gray!30,draw=gray!30,line width=5pt]  (T1) -- (T10) -- (B10) -- (B1) -- (T1);
\draw[black] (T1) .. controls +(.1,-.4) and +(-.1,-.4) .. (T2) ;
\draw[black] (B1) .. controls +(.1,.5) and +(-.1,.5) .. (B2);
\draw[black] (T2) -- (B2) ;
\draw[black] (T3) .. controls +(.1,-.4) and +(-.1,-.4) .. (T4) ;
\draw[black] (B3) .. controls +(.1,.5) and +(-.1,.5) .. (B4);
\draw[black] (T4) .. controls +(.1,-.5) and +(-.1,-.5) .. (T6) ;
\draw[black] (B4) .. controls +(.1,.5) and +(-.1,.5) .. (B6);
\draw[black] (T6) -- (B6) ;
\draw[black] (T7) -- (B7) ;
\draw[black] (T8) .. controls +(.1,-.5) and +(-.1,-.5) .. (T10) ;
\draw[black] (B8) .. controls +(.1,.5) and +(-.1,.5) .. (B10);
\draw[black] (T9) -- (B9) ;
\foreach \i in {1,...,10} 
{ \fill (T\i) circle (4pt); \fill (B\i) circle (4pt); } 
\end{tikzpicture} }
\end{array}.\\
\text{noncrossing 4-factor} && \text{symmetric 4-diagram} \\
\end{array}
\end{equation}
A simple counting argument can be used to determine the number of symmetric $m$-diagrams $|\Wm|$, which equals the number of noncrossing $m$-factors   $|\mathcal{N}_{\calA_k}^m|,$ for each diagram algebra $\A_k$:
\begin{equation*}
\begin{array}{cl}
\A_k &  |\Wm| = |\mathcal{N}_{\calA_k}^m| \phantom{\Big{|}} \\ \hline
\Pb_k(n) & \sum_t  \stirling{k}{t}\binom{t}{m} \phantom{\Big{|}}  \\
\B_k(n) & \binom{k}{m}(k-m-1)!!  \phantom{\Big{|}} \\
\RB_k(n) & \sum_t \binom{k}{m}\binom{k-m}{2t}(2t-1)!! \phantom{\Big{|}} \\
\end{array}
\hskip.5in
\begin{array}{cl}
\A_k &  |\Wm|= |\mathcal{N}_{\calA_k}^m|  \phantom{\Big{|}} \\ \hline
\TL_k(n) & \binom{k}{\frac{k-m}{2}}-\binom{k}{\frac{k-m}{2}-1} \phantom{\Big{|}}  \\
\Motz_k(n) &  \sum_t \binom{k}{m+2t}\big(\binom{m+2t}{t}-\binom{m+2t}{t-1}\big) \phantom{\Big{|}} \\
\R_k, \PR_k & \binom{k}{m}  \phantom{\Big{|}} \\
\end{array}
\end{equation*}
The corresponding integer triangles can be found in \cite{OEIS} A049020, A008313, A096713, A064189, A111062, and A007318, respectively. In the case of the planar algebras, the symmetric $m$-diagrams are exactly equal to the rank-$m$ symmetric diagrams used in the Gelfand models in \cite{HReeks} and \cite{KM-model}, and in the case of the non-planar algebras, the symmetric $m$-diagrams are a subset of the rank-$m$ symmetric diagrams. Below are examples from these algebras.
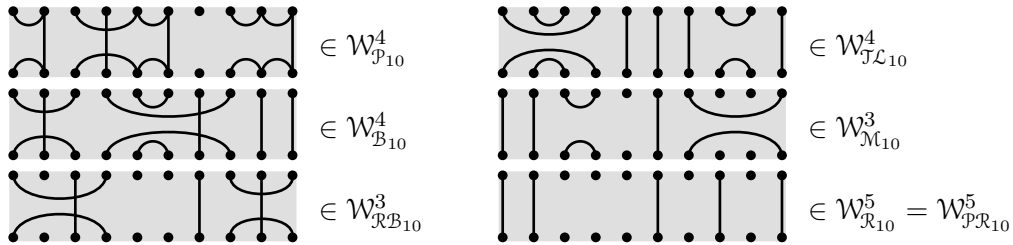
\begin{equation*}
\begin{array}{l c l}
\phantom{=  \calW^{\mathcal{PR}_{10}}_5} \begin{array}{c}
\scalebox{0.75}{
\begin{tikzpicture}[scale=.55,line width=1.35pt]
\foreach \i in {1,...,10} 
{ \path (\i,2) coordinate (T\i);
 \path (\i,0) coordinate (B\i);}
\filldraw[fill=gray!25,draw=gray!25,line width=4pt]  (T1) -- (T10) -- (B10) -- (B1) -- cycle;
\draw (T3) .. controls +(.1,-.8) and +(-.1,-.8) .. (T5);
\draw (B3) .. controls +(.1,.8) and +(-.1,.8) .. (B5);
\draw (T5) .. controls +(.1,-.6) and +(-.1,-.6) .. (T6);
\draw (B5) .. controls +(.1,.6) and +(-.1,.6) .. (B6);
\draw (T1) .. controls +(.1,-.6) and +(-.1,-.6) .. (T2);
\draw (B1) .. controls +(.1,.6) and +(-.1,.6) .. (B2);
\draw (T8) .. controls +(.1,-.6) and +(-.1,-.6) .. (T9);
\draw (B8) .. controls +(.1,.6) and +(-.1,.6) .. (B9);
\draw (T9) .. controls +(.1,-.6) and +(-.1,-.6) .. (T10);
\draw (B9) .. controls +(.1,.6) and +(-.1,.6) .. (B10);
\draw (T2) -- (B2);
\draw (T4) -- (B4);
\draw (T6) -- (B6); 
\draw (T10) -- (B10);
\foreach \i in {1,...,10} { \fill (T\i) circle (4.5pt);   \fill (B\i) circle (4.5pt);} 
\end{tikzpicture}} 
\end{array}
 \in \calW^4_{\!\calP_{10}}
&  &  
 \begin{array}{c}
\scalebox{0.75}{
\begin{tikzpicture}[scale=.55,line width=1.35pt] 
\foreach \i in {1,...,10}  { \path (\i,1) coordinate (T\i); \path (\i,-1) coordinate (B\i); } 
\filldraw[fill=gray!25,draw=gray!25,line width=4pt]  (T1) -- (T10) -- (B10) -- (B1) -- (T1);
\draw[black] (T2) .. controls +(.1,-.6) and +(-.1,-.6) .. (T3) ;
\draw[black] (B2) .. controls +(.1,.6) and +(-.1,.6) .. (B3) ;
\draw[black] (T1) .. controls +(.1,-1) and +(-.1,-1) .. (T4) ;
\draw[black] (B1) .. controls +(.1,1) and +(-.1,1) .. (B4) ;
\draw[black] (T8) .. controls +(.1,-.6) and +(-.1,-.6) .. (T9) ;
\draw[black] (B8) .. controls +(.1,.6) and +(-.1,.6) .. (B9) ;
\foreach \i in {5,6,7,10}{\draw (T\i) -- (B\i);}
\foreach \i in {1,...,10}  { \fill (T\i) circle (4.5pt); \fill (B\i) circle (4.5pt); } 
\end{tikzpicture}}\end{array}
 \in \calW^4_{\!\mathcal{TL}_{10}}
 \\
\phantom{=  \calW^{\mathcal{PR}_{10}}_5} \begin{array}{c}
\scalebox{0.75}{
\begin{tikzpicture}[scale=.55,line width=1.35pt] 
\foreach \i in {1,...,10}  { \path (\i,1) coordinate (T\i); \path (\i,-1) coordinate (B\i); } 
\filldraw[fill=gray!25,draw=gray!25,line width=4pt]  (T1) -- (T10) -- (B10) -- (B1) -- (T1);
\draw[black] (T1) .. controls +(.1,-.8) and +(-.1,-.8) .. (T3) ;
\draw[black] (B1) .. controls +(.1,.8) and +(-.1,.8) .. (B3) ;
\draw[black] (T4) .. controls +(.1,-1) and +(-.1,-1) .. (T8) ;
\draw[black] (B4) .. controls +(.1,1) and +(-.1,1) .. (B8) ;
\draw[black] (T5) .. controls +(.1,-.6) and +(-.1,-.6) .. (T6) ;
\draw[black] (B5) .. controls +(.1,.6) and +(-.1,.6) .. (B6) ;
\foreach \i in {2,7,9,10}{\draw (T\i) -- (B\i);}
\foreach \i in {1,...,10}  { \fill (T\i) circle (4.5pt); \fill (B\i) circle (4.5pt); } 
\end{tikzpicture}}\end{array}
 \in \calW^4_{\!\mathcal{B}_{10}} &  & 
\begin{array}{c}
\scalebox{0.75}{
\begin{tikzpicture}[scale=.55,line width=1.35pt] 
\foreach \i in {1,...,10}  { \path (\i,1) coordinate (T\i); \path (\i,-1) coordinate (B\i); } 
\filldraw[fill=gray!25,draw=gray!25,line width=4pt]  (T1) -- (T10) -- (B10) -- (B1) -- (T1);
\draw[black] (T3) .. controls +(.1,-.6) and +(-.1,-.6) .. (T4) ;
\draw[black] (B3) .. controls +(.1,.6) and +(-.1,.6) .. (B4) ;
\draw[black] (T7) .. controls +(.1,-1) and +(-.1,-1) .. (T10) ;
\draw[black] (B7) .. controls +(.1,1) and +(-.1,1) .. (B10) ;
\foreach \i in {1,2,6}{\draw (T\i) -- (B\i);}
\foreach \i in {1,...,10}  { \fill (T\i) circle (4.5pt); \fill (B\i) circle (4.5pt); } 
\end{tikzpicture}}\end{array}
 \in \calW^3_{\!\mathcal{M}_{10}}  \\
\phantom{=  \calW^{\mathcal{PR}_{10}}_5} \begin{array}{c}
\scalebox{0.75}{
\begin{tikzpicture}[scale=.55,line width=1.35pt] 
\foreach \i in {1,...,10}  { \path (\i,1) coordinate (T\i); \path (\i,-1) coordinate (B\i); } 
\filldraw[fill=gray!25,draw=gray!25,line width=4pt]  (T1) -- (T10) -- (B10) -- (B1) -- (T1);
\draw[black] (T1) .. controls +(.1,-1) and +(-.1,-1) .. (T4) ;
\draw[black] (B1) .. controls +(.1,1) and +(-.1,1) .. (B4) ;
\draw[black] (T10) .. controls +(.1,-.8) and +(-.1,-.8) .. (T8) ;
\draw[black] (B10) .. controls +(.1,.8) and +(-.1,.8) .. (B8) ;
\foreach \i in {3,7,9}{\draw (T\i) -- (B\i);}
\foreach \i in {1,...,10}  { \fill (T\i) circle (4.5pt); \fill (B\i) circle (4.5pt); } 
\end{tikzpicture}}\end{array}
 \in \calW^3_{\!\mathcal{RB}_{10}}   &  & 
 \begin{array}{c}
\scalebox{0.75}{
\begin{tikzpicture}[scale=.55,line width=1.35pt] 
\foreach \i in {1,...,10}  { \path (\i,1) coordinate (T\i); \path (\i,-1) coordinate (B\i); } 
\filldraw[fill=gray!25,draw=gray!25,line width=4pt]  (T1) -- (T10) -- (B10) -- (B1) -- (T1);
\foreach \i in {1,2,6,8,10}{\draw (T\i) -- (B\i);}
\foreach \i in {1,...,10}  { \fill (T\i) circle (4.5pt); \fill (B\i) circle (4.5pt); } 
\end{tikzpicture}}\end{array}
 \in \calW^5_{\!\mathcal{R}_{10}} =  \calW^5_{\!\mathcal{PR}_{10}}
\end{array}
\end{equation*}

For $d,w \in \calA_k$, we say that $d \circ w \circ d^T$ is the \emph{conjugate} of $w$ by $d$. For example, below is the  conjugation $d \circ w \circ d^T$ of diagrams $d \in \calP_{13}$ and $w \in \calW^5_{\!\calP_{13}}$,
\begin{equation} \label{eqn:conjugation}
 \begin{array}{cc}
d \!\!\!\! & = \!\!
\begin{array}{c}
\scalebox{0.9}{
\begin{tikzpicture}[scale=.55,line width=1.35pt]
\foreach \i in {1,...,13} 
{ \path (\i,2) coordinate (T\i);
 \path (\i,0) coordinate (B\i);}
\draw[rounded corners=.15mm, fill=gray!25,draw=gray!25,line width=4pt]  (T1) -- (T13) -- (B13) -- (B1) -- cycle;
\draw (T5) .. controls +(.1,-.5) and +(-.1,-.5) .. (T6);
\draw (T6) .. controls +(.1,-.5) and +(-.1,-.5) .. (T7);
\draw (T8) .. controls +(.1,-1.2) and +(-.1,-1.2) .. (T12);
\draw (T10) .. controls +(.1,-.5) and +(-.1,-.5) .. (T11);
\draw (T7) .. controls +(.1,-.5) and +(-.1,.5) .. (B8);
\draw (T1) .. controls +(.1,-.8) and +(-.1,.8) .. (B5);
\draw (T8) .. controls +(.1,-.8) and +(-.1,.8) .. (B4);
\draw (T9) .. controls +(.1,-.7) and +(-.1,.7) .. (B12);
\draw (T2) -- (B2);
\draw (T3) -- (B3);
\draw (T13) -- (B13);
\draw (B1) .. controls +(.1,.7) and +(-.1,.7) .. (B3);
\draw (B9) .. controls +(.1,.5) and +(-.1,.5) .. (B10);
\foreach \i in {1,...,13} { \fill (T\i) circle (4.5pt);  \fill (B\i) circle (4.5pt);} 
\end{tikzpicture} } 
\end{array} \\
w \!\!\!\! & = \!\!
\begin{array}{c}
\scalebox{0.9}{
\begin{tikzpicture}[scale=.55,line width=1.35pt]
\foreach \i in {1,...,13} 
{ \path (\i,2) coordinate (T\i);
 \path (\i,0) coordinate (B\i);}
\draw[rounded corners=.15mm, fill= gray!25!blue!10,draw=gray!25!blue!10,line width=4pt]  (T1) -- (T13) -- (B13) -- (B1) -- cycle;
\draw (T3) .. controls +(.1,-.7) and +(-.1,-.7) .. (T5);
\draw (B3) .. controls +(.1,.7) and +(-.1,.7) .. (B5);
\draw (T5) .. controls +(.1,-.5) and +(-.1,-.5) .. (T6);
\draw (B5) .. controls +(.1,.5) and +(-.1,.5) .. (B6);
\draw[blue!30!magenta] (T1) .. controls +(.1,-.5) and +(-.1,-.5) .. (T2);
\path (1.75,1) node {$\scriptstyle{1}$};
\draw[blue!30!magenta] (B1) .. controls +(.1,.5) and +(-.1,.5) .. (B2);
\draw[purple] (T8) .. controls +(.1,-.5) and +(-.1,-.5) .. (T9);
\draw[purple] (B8) .. controls +(.1,.5) and +(-.1,.5) .. (B9);
\draw[purple] (T9) .. controls +(.1,-.5) and +(-.1,-.5) .. (T10);
\draw[purple] (B9) .. controls +(.1,.5) and +(-.1,.5) .. (B10);
\draw[cyan] (T7) .. controls +(.1,-1) and +(-.1,-.8) .. (T13);
\draw[cyan] (B7) .. controls +(.1,1) and +(-.1,.8) .. (B13);
\draw[blue!30!magenta] (T2) -- (B2);
\draw[blue!30!green] (T4) -- (B4);
\path (3.75,1) node {$\scriptstyle{2}$};
\draw[purple] (T10) -- (B10);
\path (9.75,1) node {$\scriptstyle{3}$};
\draw[orange] (T12) -- (B12);
\path (11.75,1) node {$\scriptstyle{4}$};
\draw[cyan] (T13) -- (B13);
\path (12.75,1) node {$\scriptstyle{5}$};
\foreach \i in {1,...,13} { \fill (T\i) circle (4.5pt);   \fill (B\i) circle (4.5pt);} 
\end{tikzpicture} } 
\end{array} \\
d^T \!\!\!\! & = \!\!
\begin{array}{c}
\scalebox{0.9}{
\begin{tikzpicture}[scale=.55,line width=1.35pt]
\foreach \i in {1,...,13} 
{ \path (\i,2) coordinate (B\i);
 \path (\i,0) coordinate (T\i);}
\draw[rounded corners=.15mm, fill=gray!25,draw=gray!25,line width=4pt]  (T1) -- (T13) -- (B13) -- (B1) -- cycle;
\draw (T5) .. controls +(.1,.5) and +(-.1,.5) .. (T6);
\draw (T6) .. controls +(.1,.5) and +(-.1,.5) .. (T7);
\draw (T8) .. controls +(.1,1.2) and +(-.1,1.2) .. (T12);
\draw (T10) .. controls +(.1,.5) and +(-.1,.5) .. (T11);
\draw (T7) .. controls +(.1,.5) and +(-.1,-.5) .. (B8);
\draw (T1) .. controls +(.1,.8) and +(-.1,-.8) .. (B5);
\draw (T8) .. controls +(.1,.8) and +(-.1,-.8) .. (B4);
\draw (T9) .. controls +(.1,.7) and +(-.1,-.7) .. (B12);
\draw (T2) -- (B2);
\draw (T3) -- (B3);
\draw (T13) -- (B13);
\draw (B1) .. controls +(.1,-.7) and +(-.1,-.7) .. (B3);
\draw (B9) .. controls +(.1,-.5) and +(-.1,-.5) .. (B10);
\foreach \i in {1,...,13} { \fill (T\i) circle (4.5pt);  \fill (B\i) circle (4.5pt);} 
\end{tikzpicture} } 
\end{array}
\end{array}
\!\!\!\!\!\! =  \!\! \begin{array}{c}
\scalebox{0.9}{
\begin{tikzpicture}[scale=.55,line width=1.35pt]
\foreach \i in {1,...,13} 
{ \path (\i,2) coordinate (T\i);
 \path (\i,0) coordinate (B\i);}
\draw[rounded corners=.15mm, fill= gray!25!blue!10,draw=gray!25!blue!10,line width=4pt]  (T1) -- (T13) -- (B13) -- (B1) -- cycle;
\draw (T10) .. controls +(.1,-.5) and +(-.1,-.5) .. (T11);
\draw (B10) .. controls +(.1,.5) and +(-.1,.5) .. (B11);
\draw[blue!30!magenta] (T1) .. controls +(.1,-.5) and +(-.1,-.5) .. (T2);
\draw[blue!30!magenta] (B1) .. controls +(.1,.5) and +(-.1,.5) .. (B2);
\draw[blue!30!magenta] (T2) .. controls +(.1,-.5) and +(-.1,-.5) .. (T3);
\draw[blue!30!magenta] (B2) .. controls +(.1,.5) and +(-.1,.5) .. (B3);
\draw[purple] (T5) .. controls +(.1,-.5) and +(-.1,-.5) .. (T6);
\draw[purple] (B5) .. controls +(.1,.5) and +(-.1,.5) .. (B6);
\draw[purple] (T6) .. controls +(.1,-.5) and +(-.1,-.5) .. (T7);
\draw[purple] (B6) .. controls +(.1,.5) and +(-.1,.5) .. (B7);
\draw[blue!30!green] (T8) .. controls +(.1,-.8) and +(-.1,-1) .. (T12);
\draw[blue!30!green] (B8) .. controls +(.1,.8) and +(-.1,1) .. (B12);
\draw[blue!30!magenta] (T3) -- (B3);
\path (2.75,1) node {$\scriptstyle{1}$};
\draw[purple] (T7) -- (B7);
\path (6.75,1) node {$\scriptstyle{2}$};
\draw[orange] (T9) -- (B9);
\path (8.75,1) node {$\scriptstyle{3}$};
\draw[blue!30!green] (T12) -- (B12);
\path (11.75,1) node {$\scriptstyle{4}$};
\draw[cyan] (T13) -- (B13);
\path (12.75,1) node {$\scriptstyle{5}$};
\foreach \i in {1,...,13} { \fill (T\i) circle (4.5pt); \fill (B\i) circle (4.5pt);} 
\end{tikzpicture} }
\end{array}\!\!\!\!  =  d\circ w \circ d^T.
\end{equation}
We order the $m$ propagating blocks of a symmetric $m$-diagram according to their maximum entry. So, for example, we order the blocks in $\pi(w) = \{ 1,2 \mid 4 \mid 8,9,10 \mid 12 \mid 7,13 \}$ as follows:  $\{ 1,2\} < \{ 4 \} < \{8,9,10\} < \{12\} < \{7,13\}$. We refer to this as \emph{max-entry order}. Furthermore, by convention, \emph{we always draw the propagating edges in a symmetric $m$-diagram as identity edges connecting the maximum entries in the blocks}. Upon conjugating a symmetric $m$-diagram $w$ by $d \in \calP_k$,  if $\pn(d\circ w\circ d^T) =m$, then the propagating blocks of $w$ have been permuted, and we let $\sigma_{d,w}\in \Sb_m$ be the permutation of the fixed blocks, so that (in max-entry order),
\begin{equation}\label{def:twist}
\text{\it the $i$th propagating block in $w$ gets sent to the $\sigma_{d,w}(i)$th propagating block in $d\circ w \circ d^T$.}
\end{equation}
 We refer to $\sigma_{d,w}$ as the \emph{twist} of the conjugation of $w$ by $d$. For example in~\eqref{eqn:conjugation} $\sigma_{d,w}$ is the three-cycle $(4,3,2)$.
 
 \begin{rem} \label{rem:dddT} The following properties can be verified through simple diagram calculus for $d\in \calP_k$.
 \begin{enumerate}
\item[(1)] If $w$ is a symmetric $m$-diagram, then $d \circ w \circ d^T$ is a symmetric $m'$-diagram with $m' = \pn(d \circ w \circ d^T) \leq \pn(w) = m$.
\item[(2)] If $d = d^T$ then $d\circ d \circ d^T= d$.
\end{enumerate}
\end{rem}

\subsection{Irreducible modules $\A^\lambda_k$}

For any of the diagram algebras $\A_k$, let
\begin{equation}
\WWm := \CC\Wm = \CC\text{-span}\left\{ d \in \calA_k \ \middle| \ \text{$d$ is a symmetric $m$-diagram} \right\}.
\end{equation}
For $\lambda \in \Lambda_{n}^{\A_k}$ with $|\lambda^\ast| = m$, let $\A^\lambda_k$  be the vector space
\begin{equation}
\A^\lambda_{k} := \WWm \ot \Sb^{\lambda^\ast}_m =  \CC\text{-span}\left\{ w \ot \n_t \ \middle| \ w \in \Wm, \ t \in \SYT(\lambda^\ast) \right\}, 
\end{equation}
where $\n_t$ is a natural basis element of $\Sb^{\lambda^\ast}_m$ (see Section \ref{subsec:symmetricgroupmodules}).
If $w \in \Wm$ and $t \in \SYT(\lambda^\ast)$, we define the action of $d \in \calA_k$ on the basis element $w \ot \n_t$ to be
\begin{equation} 
d\cdot (w \ot \n_t) = \begin{cases}
n^{\ell(d,w)} (d\circ w \circ d^T) \ot \sigma_{d,w} \cdot \n_t  & \text{if $\pn(d\circ w\circ d^T) = m$, } \\
0 & \text{if $\pn(d\circ w\circ d^T) < m$, }
\end{cases} \label{eqn:twistedaction}
\end{equation}
where $\ell(d,w)$ is the number of connected components removed from the middle row during the diagram concatenation $d\circ w$ and  $\sigma_{d,w} \in \Sb_m$ is the twist of the conjugation of $w$ by $d$ defined in~\eqref{def:twist}; that is, $\sigma_{d,w}$ is the permutation on the propagating blocks of $w$ induced by $d$.

The bijection \eqref{bijection:factors} between symmetric $m$-diagrams and noncrossing $m$-factors gives rise to the following isomorphism.

\begin{prop} \label{prop:twistedaction} If $n \in \CC$ such that $\A_k$ is semisimple and $\lambda \in \Lambda_n^{\A_k}$, then $\A^\lambda_k$ and $(\A_k \e_\lambda)/\J_{m-1}$ are isomorphic as $\A_k$-modules. 
\end{prop}

\begin{proof} Let $\omega \in \mathcal{N}^m_{\calA_k}$ be a noncrossing $m$-factor that is in bijection with the symmetric $m$-diagram $w \in \mathcal{W}^m_{\calA_k}$ via \eqref{bijection:factors}. For any $t \in \SYT(\lambda^\ast)$  identify the basis element $\omega \sigma_t p_\lambda + \J_{m-1}$ of  $(\A_k \e_\lambda)/\J_{m-1}$ with the basis element $w \otimes \n_t$ of $\A_k^\lambda$, and extend this identification linearly to a vector space isomorphism.  That it is also an algebra homomorphism comes from the fact that the action in \eqref{eqn:twistedaction} is simply a combinatorial realization of diagram multiplication in the quotient space $(\A_k \e_\lambda)/\J_{m-1}$.  To see this, consider the action of 
$d \in {\calA_k}$ on each basis vector. If $\pn(d \omega) = m$, then $d \omega = n^{\ell(d,\omega)} \omega' \sigma$ for $\omega' \in \mathcal{N}^m_{\calA_k}$ such that $\omega'$ is in bijection with $d \circ w \circ d^T \in \mathcal{W}^m_{\calA_k}$. Moreover, the permutation   $\sigma \in \Sb_m$ which uncrosses $d \omega$ is, by definition of $\sigma_{d,w}$, the same as the permutation $\sigma_{d,w} \in \Sb_m$ of the fixed blocks of $w$. Thus,   $d \omega \sigma_t p_\lambda + \J_{m-1} = n^{\ell(d,\omega)} \omega' \sigma \sigma_t p_\lambda + \J_{m-1}$ if and only if $d \cdot (w \otimes \n_t) = n^{\ell(d,w)} (d\circ w \circ d^T) \ot \sigma_{d,w} \cdot \n_t$.

Finally,  if $\pn(d \omega) < m$, then $d \omega \sigma_t p_\lambda + \J_{m-1}$ is zero in the quotient space and $d \cdot w \otimes \n_t = 0$ by definition in \eqref{eqn:twistedaction}.  Thus, the actions of diagrams on basis elements are identical and the bijection extends to an $\A_k$-module isomorphism.
\end{proof}

The next theorem is then a consequence of Proposition \ref{prop:twistedaction}  and Theorem \ref{thm:HRirreducible}.

\begin{thm} If $n \in \CC$ such that $\A_k$ is semisimple, then $\{ \A^\lambda_{k} \mid \lambda \in \Lambda_n^{\A_k}\}$ is a complete set of pairwise nonisomorphic irreducible $\A_k$ modules.
\end{thm}

\begin{examp} \label{ex:choiceofbasis1}
Let $k=13$ and $\lambda = [n-5,3,2]$. There are five standard Young tableaux of shape $\lambda^\ast =[3,2]$ as shown in~\eqref{SYT32}.
Let $d$ and $w$ be the diagrams given in~\eqref{eqn:conjugation} and consider the action of $d$ on the basis element $w \ot \n_{t_4}$ of $\Pb^{\,\lambda}_k$. There is one block removed during the diagram concatenation $d\circ w$, and the five fixed blocks of $w$ are twisted by the permutation $\sigma_{d,w} = (4,3,2)$ in cycle notation. Hence $d\cdot (w\ot \n_{t_4}) = n(w'\ot\sigma_{d,w}\cdot \n_{t_4})$, where $w ' = d\circ w \circ d^T$. 
 Then $\sigma_{d,w}\cdot \n_{t_4} = \n_{\sigma_{d,w}(t_4)}$, where
\begin{equation*}
\sigma_{d,w}(t_4) = \begin{array}{c}
\scalebox{0.95}{
\begin{tikzpicture}[scale=.45]
\draw (2,2) -- (2,0) -- (0,0) -- (0,2) -- (3,2) -- (3,1) -- (0,1);
\draw (1,0) -- (1,2);
\path (0.5,1.5) node {$1$}; \path (1.5,1.5) node {$4$}; \path (2.5,1.5) node {$3$};
\path (0.5,0.5) node {$2$}; \path (1.5,0.5) node {$5$};
\end{tikzpicture}} 
\end{array}
= 
\begin{array}{c}
\scalebox{0.95}{
\begin{tikzpicture}[scale=.45]
\draw (2,2) -- (2,0) -- (0,0) -- (0,2) -- (3,2) -- (3,1) -- (0,1);
\draw (1,0) -- (1,2);
\path (0.5,1.5) node {$1$}; \path (1.5,1.5) node {$3$}; \path (2.5,1.5) node {$4$};
\path (0.5,0.5) node {$2$}; \path (1.5,0.5) node {$5$};
\end{tikzpicture}} 
\end{array} -
\begin{array}{c}
\scalebox{0.95}{
\begin{tikzpicture}[scale=.45]
\draw (2,2) -- (2,0) -- (0,0) -- (0,2) -- (3,2) -- (3,1) -- (0,1);
\draw (1,0) -- (1,2);
\path (0.5,1.5) node {$1$}; \path (1.5,1.5) node {$3$}; \path (2.5,1.5) node {$5$};
\path (0.5,0.5) node {$2$}; \path (1.5,0.5) node {$4$};
\end{tikzpicture}} 
\end{array}.
\end{equation*}
The second equality above comes from the Garnir relations (see \cite{sagan2001symmetric}), and it follows that
\begin{equation*}
d\cdot(w\ot \n_{t_4}) = n w' \ot \n_{\sigma_{d,w}(t_4)} = n w' \ot \n_{t_2} -n w' \ot \n_{t_1}.
\end{equation*}
\end{examp}

\begin{rem} A  counting argument confirms that the sum of the squares of the dimensions of the simple modules equals the dimension of $\A_k$ (Wedderburn's theorem).
For $\lambda \in \Lambda^{\A_k}_n$, the dimension of $\A^\lambda_k$ is given by
$
\dim \A^\lambda_{k} = \left| \calW^{|\lambda^\ast|}_{\!\!\calA_k} \right| f^{\lambda^\ast}.
$
For the non-planar algebras, the sum of squares of these dimensions  is given by
\begin{equation*}
\sum_{\lambda \in \Lambda^{\A_k}_n} \left(\left| \calW^{|\lambda^\ast|}_{\!\!\calA_k} \right| f^{\lambda^\ast}\right)^2 = \sum_{m \in \Gamma_{\!{\A_k}}} \sum_{\mu \vdash m} |\Wm|^2 (f^\mu)^2 = \sum_{m \in \Gamma_{\!{\A_k}}} |\Wm|^2 m! = \dim \A_k,
\end{equation*}
where we have used the corresponding symmetric group identity $m! = \sum_{\mu\vdash m}(f^\mu)^2$. The first equality uses the bijection between~\eqref{index-set-n} and~\eqref{eq:JBClabels}. The last equality is justified as follows: $|\Wm|$ counts the number of possible top (resp., bottom) rows of diagrams in $\calA_k$ with $m$ blocks distinguished as propagating blocks, so that $|\Wm|^2$ counts the number of top and bottom rows with $m$ blocks chosen from each to be propagating blocks. The distinguished blocks can be matched up in $m!$ ways, and summing over the possible ranks gives the number of basis diagrams for $\A_k$. The planar case is similar, except we have only the trivial partition $[m]$ for each $m \in \Gamma_{\A_k}$, and there is no $m!$ because propagating edges cannot cross. 
\end{rem}

\begin{rem} \label{rem:Skmodule}
When $|\lambda^\ast| = k$, the only diagrams that do not act as zero on $\A^\lambda_k$ are the permutation diagrams in $\Sb_k \subseteq \calP_k$.  Then the action~\eqref{eqn:twistedaction} is exactly the action of $\Sb_k$ on the irreducible module $\Sb^{\lambda^\ast}_k$, and there is an isomorphism $\A^\lambda_k \cong \Sb^{\lambda^*}_k$ as $\Sb_k$ modules.
\end{rem}


\newpage

\section{Set-Partition Tableaux}
\label{sec:tableaux}

In this section we describe the irreducible modules of the algebras $\A_k$ on a basis indexed by set-partition tableaux. These tableaux first appear for the partition algebra implicitly in \cite[Sec.~5.3]{BHH} and explicitly in \cite[Def.~3.14]{BH-invariant2}. They also appear (independently) as multiset tableaux in \cite{OZ}. In Section \ref{sec:setpartsubalg} we restrict the definition of set-partition tableaux to work for each of  the diagram subalgebras $\A_k$. Throughout the following we let $\Pb^{\,\lambda}_k$ denote the module $\A^\lambda_{k}$ when $\A_k$ is the partition algebra. 

\begin{definition}
Let $\lambda= [\lambda_1,\lambda_2, \dots, \lambda_\ell]$ be an integer partition of $n$ into $\ell$ parts, with $\lambda^\ast =[\lambda_2,\dots,\lambda_\ell]$, and let $\pi$ be a set partition of $\{1, \ldots, k\}$ into $t$ blocks with $|\lambda^\ast| \leq t \leq n$. A \textit{set-partition tableau $\T$ of shape $\lambda$ and content $\pi$} is a filling of the boxes of the skew shape $\lambda/[n-t]$ with the blocks of $\pi$ so that  each box of $\lambda/[n-t]$ contains a unique block of $\pi$. 
The blocks below the first row of $\T$ are called \emph{propagating} blocks, while the blocks in the first row are called \emph{non-propagating}. 
A set-partition tableau is  \textit{standard} if all of the entries of $\T$  increase across the rows from left to right and down the columns using max-entry order on the blocks of $\pi$. 
 \end{definition}

\begin{examp}
Below is a set-partition tableau $\T$ of shape $\lambda = [8,4,3,1] \vdash 16$ and content 
\begin{equation*}
\pi = \{3\mid 5 \mid 6 \mid 8 \mid 2,9 \mid  12 \mid 4,7,10,14\mid 13,15  \mid 1,16 \mid 11,17 \} 
\end{equation*}
which has $t=10$ blocks. The blocks are increasing across the rows and down the columns, so $\T$ is standard.  We have emphasized max-entry order by underlining the maximal elements in each block of $\pi$. 
\begin{align*}
\begin{array}{c}
\scalebox{0.95}{
\begin{tikzpicture}[scale=.55]
\filldraw[fill=gray!20, draw=black] (0,3) rectangle (12.5,4);
\draw (0,3) rectangle (12.5,4);
\draw[line width=1] (1.5,3) -- (1.5,0) -- (0,0) -- (0,3) -- (10,3) -- (10,2) -- (0,2);
\draw[line width=1] (7.7,3) -- (7.7,1) -- (0,1);
\draw[line width=1] (5.4,3) -- (5.4,1);
\draw[line width=1] (12.5,3) -- (12.5,4) -- (15.5,4) -- (15.5,3) -- (12.5,3); \draw[line width=1] (13.75,4) -- (13.75,3);
\draw (1.5,4) -- (1.5,3); \draw (5.4,4) -- (5.4,3); \draw (7.7,4) -- (7.7,3); \draw (10,4) -- (10,3); \draw (11.25,4) -- (11.25,3);
\path (13.1,3.5) node {$\underline{12}$}; \path (14.6,3.5) node {$1,\underline{16}$}; 
\path (.75,2.5) node {$\underline{3}$}; \path (3.4,2.5) node {$\underline{6}$}; \path (6.5,2.5) node {$\underline{8}$}; \path (8.8,2.5) node {$11,\underline{17}$};
\path (.75,1.5) node {$\underline{5}$}; \path (3.4,1.5) node {$4,7,10,\underline{14}$}; \path (6.5,1.5) node {$13,\underline{15}$};
\path (.75,0.5) node {$2,\underline{9}$};
\end{tikzpicture} }
\end{array}
\end{align*}
\end{examp}

\begin{rem}
Let $\lambda\vdash n$ and $\pi$ be a set partition of $\{1,\ldots,k\}$, and let $\T$ be a set-partition tableau of shape $\lambda$ and content $\pi$. When $n \geq 2k$ (which we  assume for the semisimplicity of $\Pb_k(n)$) there is no column of $\T$ with both propagating and non-propagating blocks. To simplify our figures we omit unnecessary boxes from the first row, and let a single box with $``\cdots"$ denote the correct number of boxes. For instance, consider the same tableau as in the example above, but where $\lambda \in \Lambda_{17,n}$ is the partition  $[n-8,4,3,1]$ for some $n \geq 34$,
\begin{align*}
\begin{array}{c}
\scalebox{0.95}{
\begin{tikzpicture}[scale=.55]
\filldraw[fill=gray!20, draw=black] (0,3) rectangle (11.5,4);
\draw[line width=1] (1.5,3) -- (1.5,0) -- (0,0) -- (0,3) -- (10,3) --(10,2) -- (0,2);
\draw[line width=1] (11.5,3) -- (11.5,4) -- (12.7,4) -- (12.7,3) -- (11.5,3);
\draw[line width=1] (12.7,3) -- (12.7,4) -- (14.5,4) -- (14.5,3) -- (12.7,3);
\draw[line width=1] (7.7,3) -- (7.7,1) -- (0,1);
\draw[line width=1] (5.4,3) -- (5.4,1);
\draw (1.5,4) -- (1.5,3); \draw (5.4,4) -- (5.4,3); \draw (7.7,4) -- (7.7,3); \draw (10,3) -- (10,4);
\path (5.75,4.8) node {$\overbrace{\phantom{\qquad\qquad\qquad\qquad\qquad\qquad\qquad\quad\quad}}^{n-10 \text{ boxes }}$};
\path (10.8,3.5) node {$\cdots$}; \path (12.1,3.5) node {$\underline{12}$};  \path (13.6,3.5) node {$1,\underline{16}$}; 
\path (.75,2.5) node {$\underline{3}$}; \path (3.4,2.5) node {$\underline{6}$}; \path (6.5,2.5) node {$\underline{8}$}; \path (8.8,2.5) node {$11,\underline{17}$};
\path (.75,1.5) node {$\underline{5}$}; \path (3.4,1.5) node {$4,7,10,\underline{14}$}; \path (6.5,1.5) node {$13,\underline{15}$};
\path (.75,0.5) node {$2,\underline{9}$};
\end{tikzpicture} }
\end{array}\! .
\end{align*}
\end{rem}

For $k \in \ZZ_{\geq 0}$ and $\lambda \in \Lambda_{k,n}$, define $\SPT(\lambda, k)$ to be the set of set-partition tableaux of shape $\lambda$ whose content is a set partition of $\{1,\ldots,k\}$, and define $\SPTb(\lambda,k)$ to be the subset of these tableaux whose first row is increasing from left to right. Finally, define $\SSPT(\lambda,k)$ to be the subset of standard set-partition tableaux. For a fixed $\lambda$ and $k$, $\SSPT(\lambda,k) \subseteq \SPTb(\lambda,k) \subseteq \SPT(\lambda,k)$, and the sizes of these sets (when $n \geq 2k$) are given by 
\begin{subequations}
\begin{align}
|\SPT(\lambda,k)| &= \sum_{t} \stirling{k}{t}\binom{t}{m}t!, \label{eqn:sptsize} \\
|\SPTb(\lambda,k)| &= \sum_{t} \stirling{k}{t}\binom{t}{m}m! = |\calW^m_{\!\calP_k}| m!,   \label{eqn:sptbsize} \\
|\SSPT(\lambda,k)| &= \sum_{t} \stirling{k}{t}\binom{t}{m} f^{\lambda/[n-t]}  = \dim \Pb^{\,\lambda}_k = \sum_{t} \stirling{k}{t}\binom{t}{m} f^{\lambda^\ast}, \label{eqn:ssptsize}
\end{align}
\end{subequations}
which are justified as follows: first partition the set $\{1,\ldots,k\}$ into at $t \geq m$ blocks, and choose $m$ of these blocks to propagate. For~\eqref{eqn:sptsize}, we can arrange the blocks of the tableau in $t!$ ways, for~\eqref{eqn:sptbsize} the first row is fixed and we are free to arrange the propagating blocks in $m!$ ways, and for~\eqref{eqn:ssptsize} the number of standard arrangements of the blocks is equal to $f^{\lambda/[n-t]}$. The second equality in~\eqref{eqn:ssptsize} also holds when $n < 2k$ \cite{BH-invariant1,BH-invariant2,BHH}. 

Recall that $\calW^m_{\!\calP_k}$ is the set of symmetric $m$-diagrams in $\calP_k$, and let $\YT(\mu)$ be the set of Young tableaux of shape $\mu$. 
For each $\lambda \in \Lambda_{k,n}$, there is an easy-to-verify  bijection,
\begin{equation}
\SPTb(\lambda,k) \xleftrightarrow{\hspace*{1cm}} \calW^{|\lambda^\ast|}_{\!\calP_k} \times \YT(\lambda^\ast), \label{eqn:sptbijection}
\end{equation}
which is given by associating $\T \in \SPTb(\lambda,k)$ with the pair $(w,t) \in \calW^{|\lambda^\ast|}_{\!\calP_k} \times \YT(\lambda^\ast)$ where $w$ is the
symmetric $|\lambda^\ast|$-diagram whose propagating and non-propagating blocks are those of $\T$ and where $t$ is the standard tableau with entries $\{1, \ldots, |\lambda^\ast|\}$ such that $i$ is placed in the same position the $i$th propagating block of $w$ occupies in $\T$.   See Example~\ref{ex:bijection}.

\begin{examp}\label{ex:bijection}
If $w \ot v_{t_4}$ is the basis element from~\eqref{eqn:conjugation}, then the bijection in \eqref{eqn:sptbijection} gives the pairing,
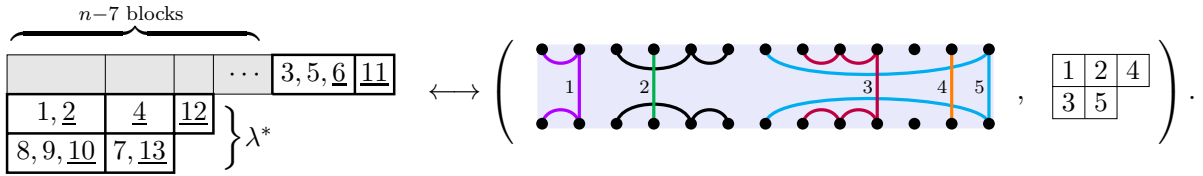
\begin{equation*}\begin{array}{c}
\scalebox{0.95}{
\begin{tikzpicture}[scale=0.55] 
\filldraw[fill=gray!20, draw=black] (0,2) rectangle (6.75,3);
\draw (2.5,2) -- (2.5,3); \draw (4.25,2) -- (4.25,3); \draw (5.25,2) -- (5.25,3);
\draw[line width = 1] (4.25,2) -- (4.25,0) -- (0,0) -- (0,2) -- (5.25,2) -- (5.25,1) -- (0,1);
\draw[line width = 1] (2.5,0) -- (2.5,2);
\draw[line width = 1] (6.75,2) -- (9.85,2) -- (9.85,3) -- (6.75,3) -- (6.75,2); \draw[line width = 1] (8.85,2) -- (8.85,3);
\path (6.1,2.5) node {$\cdots$};
\path (7.8,2.5) node {$3,5,\underline{6}$};  \path (9.35,2.5) node {$\underline{11}$}; 
\path (1.25,1.5) node {$1,\underline{2}$}; \path (3.375,1.5) node {$\underline{4}$}; \path (4.75,1.5) node {$\underline{12}$};
 \path (1.25,0.5) node {$8,9,\underline{10}$}; \path (3.375,0.5) node {$7,\underline{13}$}; 
\path (3.25,3.8) node {$\overbrace{\phantom{\qquad\qquad\qquad\qquad\quad}}^{n-7 \text{ blocks }}$};
\path (6.1,0.9) node {$\scalebox{1.3}{\Big\}} \lambda^\ast$};
\end{tikzpicture} } 
\end{array}
\longleftrightarrow
\left(\begin{array}{c}
\scalebox{0.9}{
\begin{tikzpicture}[scale=.55,line width=1.35pt]
\foreach \i in {1,...,13} 
{ \path (\i,2) coordinate (T\i);
 \path (\i,0) coordinate (B\i);}
\draw[rounded corners=.15mm, fill= gray!25!blue!10,draw=gray!25!blue!10,line width=4pt]  (T1) -- (T13) -- (B13) -- (B1) -- cycle;
\draw (T3) .. controls +(.1,-.7) and +(-.1,-.7) .. (T5);
\draw (B3) .. controls +(.1,.7) and +(-.1,.7) .. (B5);
\draw (T5) .. controls +(.1,-.5) and +(-.1,-.5) .. (T6);
\draw (B5) .. controls +(.1,.5) and +(-.1,.5) .. (B6);
\draw[blue!30!magenta] (T1) .. controls +(.1,-.5) and +(-.1,-.5) .. (T2);
\path (1.75,1) node {$\scriptstyle{1}$};
\draw[blue!30!magenta] (B1) .. controls +(.1,.5) and +(-.1,.5) .. (B2);
\draw[purple] (T8) .. controls +(.1,-.5) and +(-.1,-.5) .. (T9);
\draw[purple] (B8) .. controls +(.1,.5) and +(-.1,.5) .. (B9);
\draw[purple] (T9) .. controls +(.1,-.5) and +(-.1,-.5) .. (T10);
\draw[purple] (B9) .. controls +(.1,.5) and +(-.1,.5) .. (B10);
\draw[cyan] (T7) .. controls +(.1,-1) and +(-.1,-.8) .. (T13);
\draw[cyan] (B7) .. controls +(.1,1) and +(-.1,.8) .. (B13);
\draw[blue!30!magenta] (T2) -- (B2);
\draw[blue!30!green] (T4) -- (B4);
\path (3.75,1) node {$\scriptstyle{2}$};
\draw[purple] (T10) -- (B10);
\path (9.75,1) node {$\scriptstyle{3}$};
\draw[orange] (T12) -- (B12);
\path (11.75,1) node {$\scriptstyle{4}$};
\draw[cyan] (T13) -- (B13);
\path (12.75,1) node {$\scriptstyle{5}$};
\foreach \i in {1,...,13} { \fill (T\i) circle (4.5pt);   \fill (B\i) circle (4.5pt);} 
\end{tikzpicture} } 
\end{array}, \begin{array}{c}
\scalebox{0.95}{
\begin{tikzpicture}[scale=.45]
\draw (2,2) -- (2,0) -- (0,0) -- (0,2) -- (3,2) -- (3,1) -- (0,1);
\draw (1,0) -- (1,2);
\path (0.5,1.5) node {$1$}; \path (1.5,1.5) node {$2$}; \path (2.5,1.5) node {$4$};
\path (0.5,0.5) node {$3$}; \path (1.5,0.5) node {$5$};
\end{tikzpicture}\!} 
\end{array} \right).
\end{equation*}
\end{examp}

\subsection{Action of diagrams on set-partition tableaux}
\label{subsec:action}

Now we define an action of diagrams in $\calP_k$ on set-partition tableaux that generalizes the permutation action of the symmetric group on Young tableaux. For $d \in \calP_k$  let $\tp(d)$ be the partition of $\{1,\ldots,k\}$ induced on the top row of $d$. 

\begin{definition} \label{def:dT}
For a diagram $d \in \calP_k$ and a set partition $\pi$ of $\{1,\ldots,k\}$, let $d\circ \pi$ denote the diagram concatenation of $d$ with $\pi$, where $\pi$ is viewed as a one-line set-partition diagram. Given a  set-partition tableau $\T$ of shape $\lambda\vdash n$ and content $\pi$,  define the action of $d$ on $\T$, denoted $d(\T)$, to be the set-partition tableau of shape $\lambda$ where:
\begin{enumerate}[label=(\alph*)] 
\item the propagating blocks in $d(\T)$ are obtained by replacing each propagating block of $\T$ with the block it is connected to in $\tp(d \circ \pi)$, 
\item the non-propagating blocks in $d(\T)$ are 
\begin{enumerate}[label=(\roman*)]
\item the non-propagating blocks of $\tp(d\circ \pi)$, 
\item blocks of $\tp(d\circ \pi)$ which are connected only to non-propagating blocks of $\T$,
\end{enumerate}
\item the non-propagating blocks increase from left to right in the first row of $d(\T)$,
\item if the results of the above steps do not produce a set-partition tableau, then $d(\T) = 0$.
\end{enumerate}
The action of a diagram $d$ on a tableau $\T$ is easily obtained by placing $d$ above $\T$, drawing edges from the blocks of $\T$ to the corresponding blocks on the bottom row of $d$, and performing diagram multiplication. For instance, see Examples~\ref{ex:dTpartition} and~\ref{ex:dTsubalgebras}.
\end{definition}

\begin{examp} \label{ex:dTpartition}
Let $\T$ be the set-partition tableau from Example~\ref{ex:bijection}. Acting with the diagram $d$ from~\eqref{eqn:conjugation}, we find
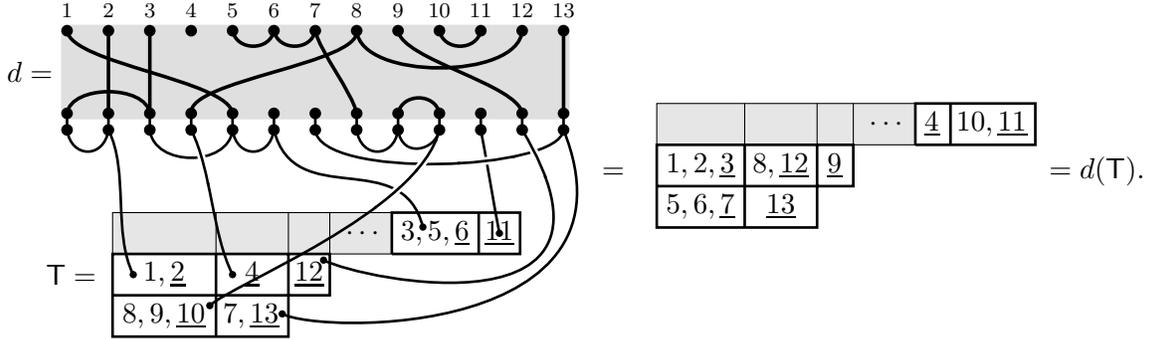
\begin{equation*}
\hspace{-0.75cm}
\begin{array}{c}
\scalebox{1}{
\begin{tikzpicture}[scale=0.55] 
\foreach \i in {1,...,13} 
{ \path (\i-2.1,5.4) coordinate (T\i); \path (\i-2.1,7.4) coordinate (t\i); \path (\i-2.1,5) coordinate (b\i);}
\path (0.5,1.5) coordinate (B1);
\path (2.9,1.5) coordinate (B2);
\path (7.5,2.65) coordinate (B3);
\path (2.35,0.75) coordinate (B4);
\path (9.35,2.5) coordinate (B5);
\path (5.1,1.85) coordinate (B6);
\path (4.1,0.55) coordinate (B7);
\path (-2,6.4) node {$d=$};
\draw[rounded corners=.15mm, fill=gray!25,draw=gray!25,line width=4.5pt]  (t1) -- (t13) -- (T13) -- (T1) -- cycle;
\draw[line width = 1.35] (t5) .. controls +(.1,-.5) and +(-.1,-.5) .. (t6);
\draw[line width = 1.35] (t6) .. controls +(.1,-.5) and +(-.1,-.5) .. (t7);
\draw[line width = 1.35] (t8) .. controls +(.1,-1.2) and +(-.1,-1.2) .. (t12);
\draw[line width = 1.35] (t10) .. controls +(.1,-.5) and +(-.1,-.5) .. (t11);
\draw[line width = 1.35] (t7) .. controls +(.1,-.5) and +(-.1,.5) .. (T8);
\draw[line width = 1.35] (t1) .. controls +(.1,-.8) and +(-.1,.8) .. (T5);
\draw[line width = 1.35] (t8) .. controls +(.1,-.8) and +(-.1,.8) .. (T4);
\draw[line width = 1.35] (t9) .. controls +(.1,-.7) and +(-.1,.7) .. (T12);
\draw[line width = 1.35] (t2) -- (T2);
\draw[line width = 1.35] (t3) -- (T3);
\draw[line width = 1.35] (t13) -- (T13);
\draw[line width = 1.35] (T1) .. controls +(.1,.7) and +(-.1,.7) .. (T3);
\draw[line width = 1.35] (T9) .. controls +(.1,.5) and +(-.1,.5) .. (T10);
\path (-1,1.5) node {$\T=$};
\filldraw[fill=gray!20, draw=black] (0,2) rectangle (6.75,3);
\draw (2.5,2) -- (2.5,3); \draw (4.25,2) -- (4.25,3); \draw (5.25,2) -- (5.25,3);
\draw[line width = 1] (4.25,2) -- (4.25,0) -- (0,0) -- (0,2) -- (5.25,2) -- (5.25,1) -- (0,1);
\draw[line width = 1] (2.5,0) -- (2.5,2);
\draw[line width = 1] (6.75,2) -- (9.85,2) -- (9.85,3) -- (6.75,3) -- (6.75,2); \draw[line width = 1] (8.85,2) -- (8.85,3);
\path (6.1,2.5) node {$\cdots$};
\path (7.8,2.5) node {$3,5,\underline{6}$};  \path (9.35,2.5) node {$\underline{11}$}; 
\path (1.25,1.5) node {$1,\underline{2}$}; \path (3.375,1.5) node {$\underline{4}$}; \path (4.75,1.5) node {$\underline{12}$};
 \path (1.25,0.5) node {$8,9,\underline{10}$}; \path (3.35,0.5) node {$7,\underline{13}$}; 
 \begin{knot}[clip width=3,end tolerance=2pt]
 \strand[line width = 1] (b4) .. controls +(.5,-1) and +(-.5,1) .. (B2);
\strand[line width = 1] (b3) .. controls +(.1,-.9) and +(-.1,-.9) .. (b5);
\strand[line width = 1] (b5) .. controls +(.1,-.7) and +(-.1,-.7) .. (b6);
\strand[line width = 1] (b1) .. controls +(.1,-.7) and +(-.1,-.7) .. (b2);
\strand[line width = 1] (b2) .. controls +(.5,-1) and +(-.5,1) .. (B1);
\strand[line width = 1] (b12)  .. controls +(3,-5) and +(1,-0.5) .. (B6);
 \strand[line width = 1] (b10) .. controls +(-.5,-2) and +(.5,0.5) .. (B4);
\strand[line width = 1] (b7) .. controls +(.1,-1.2) and +(-.1,-1) .. (b13);
\strand[line width = 1] (b13) .. controls +(2,-5) and +(1,-0.5) .. (B7);
\strand[line width = 1] (b11) .. controls +(.1,-.5) and +(-.1,.5) .. (B5);
\strand[line width = 1] (b6) .. controls +(.5,-2) and +(-.5,2) .. (B3);
\end{knot}
 \draw[line width = 1] (b8) .. controls +(.1,-.7) and +(-.1,-.7) .. (b9);
\draw[line width = 1] (b9) .. controls +(.1,-.7) and +(-.1,-.5) .. (b10);
\foreach \i in {1,...,7} { \fill (B\i) circle (2.5pt); } 
\foreach \i in {1,...,13} {\draw[line width=1] (b\i) -- (T\i); \fill (T\i) circle (4pt); \fill (t\i) circle (4pt); \fill (b\i) circle (4pt); 
\draw  (t\i)  node[above=0.05cm]{${\scriptstyle \i}$}; } 
\end{tikzpicture} } 
\end{array} \hspace{-1cm} = \begin{array}{c}
\scalebox{1}{
\begin{tikzpicture}[scale=0.55] 
\filldraw[fill=gray!20, draw=black] (0,2) rectangle (6.25,3);
\draw (2.125,2) -- (2.125,3); \draw (3.875,2) -- (3.875,3); \draw (4.75,2) -- (4.75,3);
\draw[line width = 1] (3.875,2) -- (3.875,0) -- (0,0) -- (0,2) -- (4.75,2) -- (4.75,1) -- (0,1);
\draw[line width = 1] (2.125,0) -- (2.125,2); \draw[line width = 1] (7.1,2) -- (7.1,3);
\draw[line width = 1] (6.25,2) -- (9.15,2) -- (9.15,3) -- (6.25,3) -- (6.25,2);
\path (5.6,2.5) node {$\cdots$};
 \path (6.65,2.5) node {$\underline{4}$}; \path (8.1, 2.5) node {$10,\underline{11}$};
\path (1.0625,1.5) node {$1,2,\underline{3}$}; \path (3,1.5) node {$8,\underline{12}$}; \path (4.3,1.5) node {$\underline{9}$};
 \path (1.0625,0.5) node {$5,6,\underline{7}$}; \path (3,0.5) node {$\underline{13}$}; 
\end{tikzpicture} } 
\end{array}\hspace{-0.25cm} = d(\T).
\end{equation*}
The following diagram acts as zero on $\T$, since the result is not a set-partition tableau.
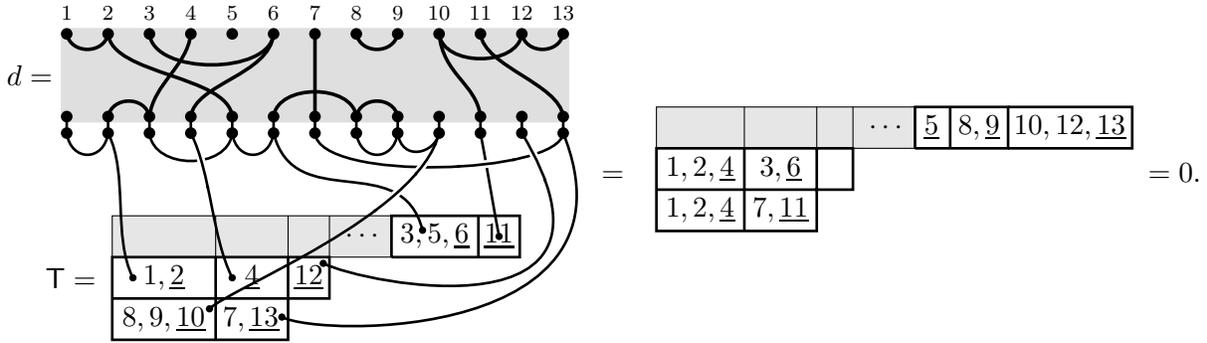
\begin{equation*}
\hspace{-0.75cm}
\begin{array}{c}
\scalebox{1}{
\begin{tikzpicture}[scale=0.55] 
\foreach \i in {1,...,13} 
{ \path (\i-2.1,5.4) coordinate (T\i); \path (\i-2.1,7.4) coordinate (t\i); \path (\i-2.1,5) coordinate (b\i);}
\path (0.5,1.5) coordinate (B1);
\path (2.9,1.5) coordinate (B2);
\path (7.5,2.65) coordinate (B3);
\path (2.35,0.75) coordinate (B4);
\path (9.35,2.5) coordinate (B5);
\path (5.1,1.85) coordinate (B6);
\path (4.1,0.55) coordinate (B7);
\path (-2,6.4) node {$d=$};
\draw[rounded corners=.15mm, fill=gray!25,draw=gray!25,line width=4.5pt]  (t1) -- (t13) -- (T13) -- (T1) -- cycle;
\draw[line width = 1.35] (t1) .. controls +(.1,-.5) and +(-.1,-.5) .. (t2);
\draw[line width = 1.35] (t3) .. controls +(.1,-1) and +(-.1,-1) .. (t6);
\draw[line width = 1.35] (t8) .. controls +(.1,-.5) and +(-.1,-.5) .. (t9);
\draw[line width = 1.35] (t12) .. controls +(.1,-.5) and +(-.1,-.5) .. (t13);
\draw[line width = 1.35] (t10) .. controls +(.1,-.8) and +(-.1,-.8) .. (t12);
\draw[line width = 1.35] (T2) .. controls +(.1,.5) and +(-.1,.5) .. (T3);
\draw[line width = 1.35] (T8) .. controls +(.1,.5) and +(-.1,.5) .. (T9);
\draw[line width = 1.35] (T6) .. controls +(.1,.8) and +(-.1,.8) .. (T8);
\draw[line width = 1.35] (t2) .. controls +(.1,-1) and +(-.1,1) .. (T5);
\draw[line width = 1.35] (t4)  .. controls +(-.1,-.8) and +(.1,.8) ..  (T3);
\draw[line width = 1.35] (t6) .. controls +(-.1,-1) and +(.1,.8) .. (T4);
\draw[line width = 1.35] (t7) -- (T7);
\draw[line width = 1.35] (t10) .. controls +(.1,-1) and +(-.1,.8) .. (T11);
\draw[line width = 1.35] (t11) .. controls +(.1,-.8) and +(-.1,.8) .. (T13);
\path (-1,1.5) node {$\T=$};
\filldraw[fill=gray!20, draw=black] (0,2) rectangle (6.75,3);
\draw (2.5,2) -- (2.5,3); \draw (4.25,2) -- (4.25,3); \draw (5.25,2) -- (5.25,3);
\draw[line width = 1] (4.25,2) -- (4.25,0) -- (0,0) -- (0,2) -- (5.25,2) -- (5.25,1) -- (0,1);
\draw[line width = 1] (2.5,0) -- (2.5,2);
\draw[line width = 1] (6.75,2) -- (9.85,2) -- (9.85,3) -- (6.75,3) -- (6.75,2); \draw[line width = 1] (8.85,2) -- (8.85,3);
\path (6.1,2.5) node {$\cdots$};
\path (7.8,2.5) node {$3,5,\underline{6}$};  \path (9.35,2.5) node {$\underline{11}$}; 
\path (1.25,1.5) node {$1,\underline{2}$}; \path (3.375,1.5) node {$\underline{4}$}; \path (4.75,1.5) node {$\underline{12}$};
 \path (1.25,0.5) node {$8,9,\underline{10}$}; \path (3.35,0.5) node {$7,\underline{13}$}; 
 \begin{knot}[clip width=3,end tolerance=2pt]
\strand[line width = 1] (b4) .. controls +(.5,-1) and +(-.5,1) .. (B2);
\strand[line width = 1] (b3) .. controls +(.1,-.9) and +(-.1,-.9) .. (b5);
\strand[line width = 1] (b5) .. controls +(.1,-.7) and +(-.1,-.7) .. (b6);
\strand[line width = 1] (b1) .. controls +(.1,-.7) and +(-.1,-.7) .. (b2);
\strand[line width = 1] (b2) .. controls +(.5,-1) and +(-.5,1) .. (B1);
\strand[line width = 1] (b12)  .. controls +(3,-5) and +(1,-0.5) .. (B6);
\strand[line width = 1] (b10) .. controls +(-.5,-2) and +(.5,0.5) .. (B4);
\strand[line width = 1] (b7) .. controls +(.1,-1.2) and +(-.1,-1) .. (b13);
\strand[line width = 1] (b13) .. controls +(2,-5) and +(1,-0.5) .. (B7);
\strand[line width = 1] (b11) .. controls +(.1,-.5) and +(-.1,.5) .. (B5);
\strand[line width = 1] (b6) .. controls +(.5,-2) and +(-.5,2) .. (B3);
\end{knot}
 \draw[line width = 1] (b8) .. controls +(.1,-.7) and +(-.1,-.7) .. (b9);
\draw[line width = 1] (b9) .. controls +(.1,-.7) and +(-.1,-.5) .. (b10);
\foreach \i in {1,...,7} { \fill (B\i) circle (2.5pt); } 
\foreach \i in {1,...,13} {\draw[line width=1] (b\i) -- (T\i); \fill (T\i) circle (4pt); \fill (t\i) circle (4pt); \fill (b\i) circle (4pt); 
\draw  (t\i)  node[above=0.05cm]{${\scriptstyle \i}$}; } 
\end{tikzpicture} } 
\end{array} \hspace{-1cm} = \begin{array}{c}
\scalebox{1}{
\begin{tikzpicture}[scale=0.55] 
\filldraw[fill=gray!20, draw=black] (0,2) rectangle (6.25,3);
\draw (2.125,2) -- (2.125,3); \draw (3.875,2) -- (3.875,3); \draw (4.75,2) -- (4.75,3);
\draw[line width = 1] (3.875,2) -- (3.875,0) -- (0,0) -- (0,2) -- (4.75,2) -- (4.75,1) -- (0,1);
\draw[line width = 1] (2.125,0) -- (2.125,2); \draw[line width = 1] (7.1,2) -- (7.1,3);
\draw[line width = 1] (6.25,2) -- (8.5,2) -- (8.5,3) -- (6.25,3) -- (6.25,2);
\draw[line width = 1] (8.5,3) -- (11.5,3) -- (11.5,2) -- (8.5,2);
\path (5.6,2.5) node {$\cdots$};
 \path (6.65,2.5) node {$\underline{5}$}; \path (7.8, 2.5) node {$8,\underline{9}$}; \path (10, 2.5) node {$10,12,\underline{13}$};
\path (1.0625,1.5) node {$1,2,\underline{4}$}; \path (3,1.5) node {$3,\underline{6}$}; \path (4.3,1.5) node {$$};
 \path (1.0625,0.5) node {$1,2,\underline{4}$}; \path (3,0.5) node {$7,\underline{11}$}; 
\end{tikzpicture} } 
\end{array}\hspace{-0.25cm} =0.
\end{equation*}
\end{examp}

\begin{rem} \label{rem:combinatorialaction} 
A diagram $d$ acts as zero on a set-partition tableau $\T$ if
\begin{enumerate}[label=(\alph*)]
\item two propagating blocks of $\T$ become connected when constructing $d(\T)$, or
\item a propagating block of $\T$ does not propagate to the top of $d$ when constructing $d(\T)$.
\end{enumerate}
\end{rem}

\subsection{Natural basis} \label{subsec:natural}

For $\lambda \in \Lambda_{k,n}$, let $\left\{ \N_\T \ \middle| \ \T \in \SSPT(\lambda,k)  \right\}$ be a set of vectors indexed by the standard set-partition tableaux of shape $\lambda$. Define
\begin{equation} \label{eqn:naturalbasis} 
\mathsf{P}_k^\lambda = \CC\text{-span}\left\{ \N_\T \ \middle| \ \T \in \SSPT(\lambda,k)  \right\},
\end{equation}
and for $d \in \calP_k$ and $\T \in \SSPT(\lambda,k)$ define
\begin{equation}  \label{eqn:naturalaction} 
d\cdot \N_\T = \begin{cases} 
n^{\ell(d,\T)} \N_{d(\T)} & \text{if $d(\T)$ is a set-partition tableau, } \\
0 & \text{if $d(\T)$ is not a set-partition tableau, }
\end{cases} 
\end{equation}
where $d(\T)$ is defined in Definition~\ref{def:dT} and $\ell(d,\T)$ is the number of connected components removed in the construction of $d(\T)$. If $d(\T)$ is not standard, then $\N_{d(\T)}$ can be expressed as an integer linear combination of basis elements using Garnir relations (see Section~\ref{subsec:symmetricgroupmodules}).

\begin{examp}
Let $d$ and $\T$ be defined as in the first example from Example~\ref{ex:dTpartition}. In the construction of $d(\T)$ there is one connected component removed, so that
\begin{equation*}
d\cdot \N_\T =n \N_{d(\T)}, \text{ where } d(\T) = 
\begin{array}{c}
\scalebox{0.9}{
\begin{tikzpicture}[scale=0.55] 
\filldraw[fill=gray!20, draw=black] (0,2) rectangle (6.25,3);
\draw (2.125,2) -- (2.125,3); \draw (3.875,2) -- (3.875,3); \draw (4.75,2) -- (4.75,3);
\draw[line width = 1] (3.875,2) -- (3.875,0) -- (0,0) -- (0,2) -- (4.75,2) -- (4.75,1) -- (0,1);
\draw[line width = 1] (2.125,0) -- (2.125,2); \draw[line width = 1] (7.1,2) -- (7.1,3);
\draw[line width = 1] (6.25,2) -- (9.15,2) -- (9.15,3) -- (6.25,3) -- (6.25,2);
\path (5.6,2.5) node {$\cdots$};
 \path (6.65,2.5) node {$\underline{4}$}; \path (8.1, 2.5) node {$10,\underline{11}$};
\path (1.0625,1.5) node {$1,2,\underline{3}$}; \path (3,1.5) node {$8,\underline{12}$}; \path (4.3,1.5) node {$\underline{9}$};
 \path (1.0625,0.5) node {$5,6,\underline{7}$}; \path (3,0.5) node {$\underline{13}$}; 
\end{tikzpicture} } 
\end{array}.
\end{equation*}
The result is nonstandard, with a descent in the first row. The Garnir relation for straightening $\N_{d(\T)}$ is:
\begin{equation*}
\begin{array}{c}
\scalebox{0.9}{
\begin{tikzpicture}[scale=0.55] 
\filldraw[fill=gray!20, draw=black] (0,2) rectangle (6.25,3);
\draw (2.125,2) -- (2.125,3); \draw (3.875,2) -- (3.875,3); \draw (4.75,2) -- (4.75,3);
\draw[line width = 1] (3.875,2) -- (3.875,0) -- (0,0) -- (0,2) -- (4.75,2) -- (4.75,1) -- (0,1);
\draw[line width = 1] (2.125,0) -- (2.125,2); \draw[line width = 1] (7.1,2) -- (7.1,3);
\draw[line width = 1] (6.25,2) -- (9.15,2) -- (9.15,3) -- (6.25,3) -- (6.25,2);
\path (5.6,2.5) node {$\cdots$};
 \path (6.65,2.5) node {$\underline{4}$}; \path (8.1, 2.5) node {$10,\underline{11}$};
\path (1.0625,1.5) node {$1,2,\underline{3}$}; \path (3,1.5) node {$8,\underline{11}$}; \path (4.3,1.5) node {$\underline{9}$};
 \path (1.0625,0.5) node {$5,6,\underline{7}$}; \path (3,0.5) node {$\underline{12}$}; 
\end{tikzpicture} } 
\end{array} =
\begin{array}{c}
\scalebox{0.9}{
\begin{tikzpicture}[scale=0.55] 
\filldraw[fill=gray!20, draw=black] (0,2) rectangle (6.35,3);
\draw (2.125,2) -- (2.125,3); \draw (3.125,2) -- (3.125,3); \draw (4.85,2) -- (4.85,3);
\draw[line width = 1] (3.125,2) -- (3.125,0) -- (0,0) -- (0,2) -- (4.85,2) -- (4.85,1) -- (0,1);
\draw[line width = 1] (2.125,0) -- (2.125,2); \draw[line width = 1] (7.2,2) -- (7.2,3);
\draw[line width = 1] (6.35,2) -- (9.25,2) -- (9.25,3) -- (6.35,3) -- (6.35,2);
\path (5.7,2.5) node {$\cdots$};
 \path (6.75,2.5) node {$\underline{4}$}; \path (8.2, 2.5) node {$10,\underline{11}$};
\path (1.0625,1.5) node {$1,2,\underline{3}$}; \path (2.625,1.5) node {$\underline{9}$}; \path (4,1.5) node {$8,\underline{11}$};
 \path (1.0625,0.5) node {$5,6,\underline{7}$}; \path (2.625,0.5) node {$\underline{12}$}; 
\end{tikzpicture} } 
\end{array} - 
\begin{array}{c}
\scalebox{0.9}{
\begin{tikzpicture}[scale=0.55] 
\filldraw[fill=gray!20, draw=black] (0,2) rectangle (6.45,3);
\draw (2.125,2) -- (2.125,3); \draw (3.875,2) -- (3.875,3); \draw (4.95,2) -- (4.95,3);
\draw[line width = 1] (3.875,2) -- (3.875,0) -- (0,0) -- (0,2) -- (4.95,2) -- (4.95,1) -- (0,1);
\draw[line width = 1] (2.125,0) -- (2.125,2); \draw[line width = 1] (7.3,2) -- (7.3,3);
\draw[line width = 1] (6.45,2) -- (9.35,2) -- (9.35,3) -- (6.45,3) -- (6.45,2);
\path (5.8,2.5) node {$\cdots$};
 \path (6.85,2.5) node {$\underline{4}$}; \path (8.3, 2.5) node {$10,\underline{11}$};
\path (1.0625,1.5) node {$1,2,\underline{3}$}; \path (3,1.5) node {$\underline{9}$}; \path (4.4,1.5) node {$\underline{12}$};
 \path (1.0625,0.5) node {$5,6,\underline{7}$}; \path (3,0.5) node {$8,\underline{11}$}; 
\end{tikzpicture} } 
\end{array}, 
\end{equation*}
and hence 
\begin{equation*}
d\cdot \N_\T =n \N_{\T_1} - n \N_{\T_2},
\end{equation*}
where $\T_1$ and $\T_2$ are the two standard set-partition tableaux appearing above. This can be compared to Example~\ref{ex:choiceofbasis1} (a), which gives the analogous action on the diagram basis. 
\end{examp}

\begin{thm} \label{thm:naturalrep}
The action defined in~\eqref{eqn:naturalaction} makes $\mathsf{P}^\lambda_k$ into a $\Pb_k(n)$-module, and $\mathsf{P}^\lambda_k \cong \Pb^{\,\lambda}_k$. 
\end{thm}

\begin{proof}
We show that the action defined on set-partition tableaux is simply the result of applying the bijection~\eqref{eqn:sptbijection} to the action defined in~\eqref{eqn:twistedaction} when the basis for $\Sb^{\lambda^\ast}_m$ is Young's natural basis $v_t = \n_t$. Let $\T$ be a standard set-partition tableau of shape $\lambda$ and content $\pi$, and let $w \ot \n_t$, be the basis element associated to $\T$ via the bijection~\eqref{eqn:sptbijection}. Assuming $\pn(d\circ w \circ d^T) = m$, we have $d \cdot(w\ot \n_t) =n^{\ell(d,w)} d\circ w \circ d^T \ot \n_{\sigma_{d,w}(t)}$, where the $i$th propagating block of $w$ gets sent to the $\sigma_{d,w}(i)$th propagating block of $d\circ w \circ d^T$. To obtain $\sigma_{d,w}(t)$, we replace $i$ with $\sigma_{d,w}(i)$.  
If $\T'$ is the set-partition tableau associated to $d\circ w \circ d^T \ot \n_{\sigma_{d,w}(t)}$  by~\eqref{eqn:sptbijection}, then the propagating blocks of $\T'$ are obtained by replacing the $i$th propagating block of $\T$ with the $\sigma_{d,w}(i)$th propagating block of $d \circ w \circ d^T$ for each $i$, and the non-propagating blocks of $\T'$ are the non-propagating blocks of $d\circ w \circ d^T \ot \n_{\sigma_{d,w}(t)}$, which  are either the non-propagating blocks of $d$ or the blocks of $d$ connected only to non-propagating blocks of $w$. 
Hence $\T' = d(\T)$. One can easily confirm that the connected components removed in the construction of $d(\T)$ are connected only to non-propagating blocks of $\T$, otherwise the action gives zero. Hence $\ell(d,\T)  = \ell(d,w)$. Finally, by Remark~\ref{rem:combinatorialaction} the criteria for $d(\T)= 0$ are equivalent to the criteria for $\pn(d\circ w \circ d^T) < m$. 
\end{proof}

\begin{rem}
The construction defined in~\eqref{eqn:naturalbasis} and~\eqref{eqn:naturalaction} is a partition algebra analogue of Young's natural basis for the irreducible modules of the symmetric group, and we refer to $\left\{ \N_\T \ \middle| \ \T \in \SSPT(\lambda,k)  \right\}$ as the \emph{natural basis} for $\Pb^{\,\lambda}_k$. Analogous modules can be constructed when the basis for $\Sb^{\lambda^\ast}_m$ is seminormal $\vs_t$ or orthogonal $\us_t$. However, the action on these modules, though isomorphic to the one defined above, lacks the ``naturalness" evident in~\eqref{eqn:naturalaction}.
\end{rem}

The generators $\mathfrak{s}_i, \mathfrak{p}_i, \mathfrak{b}_i,$  have particularly nice actions on set-partition tableaux, which we describe in the following theorem. The actions of the generators $\mathfrak{e}_i, \mathfrak{l}_i$, and $\mathfrak{r}_i$ of the subalgebras are omitted for brevity but can easily be obtained from $\mathfrak{s}_i, \mathfrak{b}_i$, and $\mathfrak{p}_i$. 

\begin{thm} \label{thm:generatorsactionSPT}
Let $\lambda \in \Lambda_{k,n}$ and $\T \in \SSPT(\lambda, k)$, so that $\N_\T$ is an element of the natural basis for $\mathsf{P}^{\lambda}_k$. Then the action of $\mathfrak{s}_i$, $\mathfrak{p}_i$, and $\mathfrak{b}_i$ on $\N_\T$ are given by: 
\begin{enumerate}[label=\rm{(\alph*)}]

\item $\displaystyle \mathfrak s_i \cdot \N_\T = \N_{\mathfrak s_i(\T)},$ where $\mathfrak s_i(\T)$ is the set-partition tableau in $\SPTb(\lambda,k)$ obtained from $\T$ by swapping $i$ and $i+1$, and standardizing the first row.

\item $\displaystyle \mathfrak p_i \cdot \N_\T = \begin{dcases}
n \N_{\T} & \text{if $\{i\}$ is a non-propagating singleton block in $\T$, }\\
0 & \text{if $\{i\}$ is a propagating singleton block in $\T$, }\\
\N_{\mathfrak p_i(\T)} & \text{otherwise, } 
\end{dcases}$

\noindent
where $\mathfrak p_i(\T)$ is the set-partition tableau in $\SPTb(\lambda,k)$ obtained from $\T$ by removing $i$ from its block, 
placing the singleton block $\{i\}$ into the first row,
and standardizing the first row.

\item $\displaystyle \mathfrak b_i \cdot \N_\T = \begin{dcases}
\N_{\T} & \text{if $i$ and $i+1$ are in the same block in $\T$, }\\
0 & \text{if $i$ and $i+1$ are in different propagating blocks in $\T$,} \\
\N_{\mathfrak b_i(\T)} & \text{otherwise, } 
\end{dcases}$

\noindent
where $\mathfrak b_i(\T)$ is the set-partition tableau in $\SPTb(\lambda,k)$ obtained from $\T$ by joining the block containing $i$ with the block containing $i+1$, and standardizing the first row.  
The resulting block becomes propagating if one of the original blocks was propagating, and otherwise stays non-propagating.

\end{enumerate}
If ${\mathfrak s_i(\T)}$, ${\mathfrak p_i(\T)}$, ${\mathfrak b_i(\T)}$
is a nonstandard set-partition tableau then $\N_{\mathfrak s_i(\T)}$, $\N_{\mathfrak p_i(\T)}$, $\N_{\mathfrak b_i(\T)}$ 
can be expressed as an integer linear combination of basis elements using Garnir relations (see Section~\ref{subsec:symmetricgroupmodules}).
\end{thm}

\begin{proof}
The action is easily obtained from~\eqref{eqn:naturalaction} through diagram calculus as in Example~\ref{ex:dTpartition}. 
\end{proof}

\begin{examp}

We give examples of the explicit action of $\mathfrak p_i$ and $\mathfrak b_i$ described above. 

\begin{enumerate}[label=\rm{(\alph*)}]

\item \emph{Action of $\mathfrak p_i$.} Consider the following standard set-partition tableau $\T$ of shape $[n-4,3,1]$,
\begin{align*}
\mathfrak p_5 \left(
\begin{array}{c}
\scalebox{0.95}{
\begin{tikzpicture}[scale=.55] 
\fill[gray!25, draw=black] (0,2) rectangle (5.5,3);
\draw[line width=1] (1,2) -- (1,0) -- (0,0) -- (0,2) -- (4,2) -- (4,1) -- (0,1);
\draw[line width=1] (3,1) -- (3,2); \draw[line width=1] (6.25,2) -- (6.25,3);
\draw[line width=1] (5.5,2) -- (7.75,2) -- (7.75,3) -- (5.5,3) -- (5.5,2);
\draw (1,2) -- (1,3); \draw (3,2) -- (3,3); \draw (4,2) -- (4,3);
\path (4.8,2.5) node {$\cdots$}; \path (5.9,2.5) node {$\underline{1}$}; \path (7,2.5) node {$5,\underline{6}$};
\path (0.5,1.5) node {$\underline{4}$}; \path (0.5,0.5) node {$\underline{7}$}; 
\path (2,1.5) node {$2,3,\underline{8}$}; \path (3.5,1.5) node {$\underline{9}$}; 
\end{tikzpicture} } 
\end{array} \right) &= 
\begin{array}{c}
\scalebox{0.95}{
\begin{tikzpicture}[scale=.55] 
\fill[gray!25, draw=black] (0,2) rectangle (5.5,3);
\draw[line width=1] (1,2) -- (1,0) -- (0,0) -- (0,2) -- (4,2) -- (4,1) -- (0,1);
\draw[line width=1] (3,1) -- (3,2); \draw[line width=1] (6.25,2) -- (6.25,3);  \draw[line width=1] (7,2) -- (7,3);
\draw[line width=1] (5.5,2) -- (7.75,2) -- (7.75,3) -- (5.5,3) -- (5.5,2);
\draw (1,2) -- (1,3); \draw (3,2) -- (3,3); \draw (4,2) -- (4,3);
\path (4.8,2.5) node {$\cdots$}; \path (5.9,2.5) node {$\underline{1}$}; \path (6.65,2.5) node {$\underline{5}$}; \path (7.35,2.5) node {$\underline{6}$};
\path (0.5,1.5) node {$\underline{4}$}; \path (0.5,0.5) node {$\underline{7}$}; 
\path (2,1.5) node {$2,3,\underline{8}$}; \path (3.5,1.5) node {$\underline{9}$}; 
\end{tikzpicture} } 
\end{array}, \\
\mathfrak p_8 \left(
\begin{array}{c}
\scalebox{0.95}{
\begin{tikzpicture}[scale=.55] 
\fill[gray!25, draw=black] (0,2) rectangle (5.5,3);
\draw[line width=1] (1,2) -- (1,0) -- (0,0) -- (0,2) -- (4,2) -- (4,1) -- (0,1);
\draw[line width=1] (3,1) -- (3,2); \draw[line width=1] (6.25,2) -- (6.25,3);
\draw[line width=1] (5.5,2) -- (7.75,2) -- (7.75,3) -- (5.5,3) -- (5.5,2);
\draw (1,2) -- (1,3); \draw (3,2) -- (3,3); \draw (4,2) -- (4,3);
\path (4.8,2.5) node {$\cdots$}; \path (5.9,2.5) node {$\underline{1}$}; \path (7,2.5) node {$5,\underline{6}$};
\path (0.5,1.5) node {$\underline{4}$}; \path (0.5,0.5) node {$\underline{7}$}; 
\path (2,1.5) node {$2,3,\underline{8}$}; \path (3.5,1.5) node {$\underline{9}$}; 
\end{tikzpicture} } 
\end{array} \right) &= 
\begin{array}{c}
\scalebox{0.95}{
\begin{tikzpicture}[scale=.55] 
\fill[gray!25, draw=black] (0,2) rectangle (5,3);
\draw[line width=1] (1,2) -- (1,0) -- (0,0) -- (0,2) -- (3.5,2) -- (3.5,1) -- (0,1);
\draw[line width=1] (2.5,1) -- (2.5,2); \draw[line width=1] (5.75,2) -- (5.75,3); \draw[line width=1] (7.1,2) -- (7.1,3);
\draw[line width=1] (5,2) -- (7.75,2) -- (7.75,3) -- (5,3) -- (5,2);
\draw (1,2) -- (1,3); \draw (2.5,2) -- (2.5,3); \draw (3.5,2) -- (3.5,3);
\path (4.3,2.5) node {$\cdots$}; \path (5.4,2.5) node {$\underline{1}$}; \path (6.4,2.5) node {$5,\underline{6}$}; \path (7.45,2.5) node {$\underline{8}$};
\path (0.5,1.5) node {$\underline{4}$}; \path (0.5,0.5) node {$\underline{7}$}; 
\path (1.75,1.5) node {$2,\underline{3}$}; \path (3,1.5) node {$\underline{9}$}; 
\end{tikzpicture} } 
\end{array}.
\end{align*}
Since $1$ is a non-propagating singleton, $\mathfrak p_1\cdot \N_\T = n \N_\T$. Since $4$ is a propagating singleton, $\mathfrak p_4$ acts as zero on $\T$. When $\mathfrak p_5$ acts on $\T$, it separates $5$ and $6$. When $\mathfrak p_8$ acts on $\T$, it moves $8$ to its own block on the first row, and the result is nonstandard. We then have $\mathfrak p_5 \cdot \N_\T = \N_{\mathfrak p_5(\T)}$ and $\mathfrak p_8 \cdot \N_\T = \N_{\mathfrak p_8(\T)}$.

\item \emph{Action of $\mathfrak b_i$.} Consider the following set-partition tableau $\T$ of shape $[n-4,3,1]$,
\begin{align*}
\mathfrak b_2 \left(
\begin{array}{c}
\scalebox{0.95}{
\begin{tikzpicture}[scale=0.55] 
\fill[gray!25, draw=black] (0,2) rectangle (5,3);
\draw[line width=1] (1.5,2) -- (1.5,0) -- (0,0) -- (0,2) -- (3.5,2) -- (3.5,1) -- (0,1);
\draw[line width=1] (2.4,2) -- (2.4,1); \draw[line width=1] (6.4,2) -- (6.4,3);
\draw[line width=1] (5,2) -- (8.5,2) -- (8.5,3) -- (5,3) -- (5,2);
\draw (1.5,2) -- (1.5,3); \draw (2.4,2) -- (2.4,3); \draw (3.5,2) -- (3.5,3);
\path (4.35,2.5) node {$\cdots$}; \path (5.7,2.5) node {$6,\underline{8}$}; \path (7.4,2.5) node {$1,2,\underline{9}$};
\path (0.75,1.5) node {$\underline{3}$}; \path (1.9,1.5) node {$\underline{7}$}; \path (2.95,1.5) node {$\underline{10}$};
\path (0.75,0.5) node {$4,\underline{5}$};
\end{tikzpicture} } 
\end{array} \right) &= 
\begin{array}{c}
\scalebox{0.95}{
\begin{tikzpicture}[scale=0.55] 
\fill[gray!25, draw=black] (0,2) rectangle (6.5,3);
\draw[line width=1] (3,2) -- (3,0) -- (0,0) -- (0,2) -- (5,2) -- (5,1) -- (0,1);
\draw[line width=1] (3.9,2) -- (3.9,1);
\draw[line width=1] (6.5,2) -- (6.5,3) -- (8,3) --(8,2) -- (6.5,2);
\draw (3,2) -- (3,3); \draw (3.9,2) -- (3.9,3); \draw (5,2) -- (5,3);
\path (5.85,2.5) node {$\cdots$}; \path (7.25,2.5) node {$6,\underline{8}$};
\path (1.5,1.5) node {$1,2,3,\underline{9}$}; \path (1.5,0.5) node {$4,\underline{5}$}; 
\path (3.45,1.5) node {$\underline{7}$}; \path (4.45,1.5) node {$\underline{10}$}; 
\end{tikzpicture} } 
\end{array}, \\
\mathfrak b_8 \left(
\begin{array}{c}
\scalebox{0.95}{
\begin{tikzpicture}[scale=0.55] 
\fill[gray!25, draw=black] (0,2) rectangle (5,3);
\draw[line width=1] (1.5,2) -- (1.5,0) -- (0,0) -- (0,2) -- (3.5,2) -- (3.5,1) -- (0,1);
\draw[line width=1] (2.4,2) -- (2.4,1); \draw[line width=1] (6.4,2) -- (6.4,3);
\draw[line width=1] (5,2) -- (8.5,2) -- (8.5,3) -- (5,3) -- (5,2);
\draw (1.5,2) -- (1.5,3); \draw (2.4,2) -- (2.4,3); \draw (3.5,2) -- (3.5,3);
\path (4.35,2.5) node {$\cdots$}; \path (5.7,2.5) node {$6,\underline{8}$}; \path (7.4,2.5) node {$1,2,\underline{9}$};
\path (0.75,1.5) node {$\underline{3}$}; \path (1.9,1.5) node {$\underline{7}$}; \path (2.95,1.5) node {$\underline{10}$};
\path (0.75,0.5) node {$4,\underline{5}$};
\end{tikzpicture} } 
\end{array} \right) &= 
\begin{array}{c}
\scalebox{0.95}{
\begin{tikzpicture}[scale=0.55] 
\fill[gray!25, draw=black] (0,2) rectangle (5,3);
\draw[line width=1] (1.5,2) -- (1.5,0) -- (0,0) -- (0,2) -- (3.5,2) -- (3.5,1) -- (0,1);
\draw[line width=1] (2.4,2) -- (2.4,1);
\draw[line width=1] (5,2) -- (8.5,2) -- (8.5,3) -- (5,3) -- (5,2);
\draw (1.5,2) -- (1.5,3); \draw (2.4,2) -- (2.4,3); \draw (3.5,2) -- (3.5,3);
\path (4.35,2.5) node {$\cdots$}; \path (6.75,2.5) node {$1,2,6,8,\underline{9}$};
\path (0.75,1.5) node {$\underline{3}$}; \path (1.95,1.5) node {$\underline{7}$}; \path (2.95,1.5) node {$\underline{10}$};
\path (0.75,0.5) node {$4,\underline{5}$};
\end{tikzpicture} } 
\end{array}.
\end{align*}
Since $1$ and $2$, and $4$ and $5$, are in the same block, both $\mathfrak b_1$ and $\mathfrak b_4$ fix $\T$. Since $3$ and $4$ are in different propagating blocks, $\mathfrak b_3$ acts as zero on $\T$. When $\mathfrak b_2$ acts on $\T$, the contents of the block containing $2$ are appended to the block containing $3$, and the result is nonstandard. Finally, $\mathfrak b_8$ acts by joining the blocks containing $8$ and $9$. Thus $\mathfrak b_2 \cdot \N_\T = \N_{\mathfrak b_2(\T)}$ and $\mathfrak b_8 \cdot \N_\T = \N_{\mathfrak b_8(\T)}$.

\end{enumerate}
\end{examp}

\subsection{Subalgebras}\label{sec:setpartsubalg}

When $\A_k$ is a subalgebra of the partition algebra, applying the bijection~\eqref{eqn:sptbijection} to basis elements of $\A^\lambda_k$ yields restricted types of standard set-partition tableaux of shape $\lambda \in \Lambda^{\A_k}_n$. In particular, for all of the proper subalgebras the propagating blocks are singletons. For the Brauer and Temperley-Lieb algebras the non-propagating blocks are pairs, for the rook-Brauer and Motzkin algebras the non-propagating blocks are  pairs or singletons, and for the rook monoid and planar rook monoid algebras the non-propagating blocks are singletons. Below are example set-partition tableaux for these subalgebras.
\begin{equation*}
\begin{array}{l l}
\begin{array}{c}
\scalebox{0.8}{
\begin{tikzpicture}[scale=0.55] 
\fill[gray!25, draw=black] (0,2) rectangle (4.5,3);
\draw[line width=1] (2,2) -- (2,1);
\foreach \i in {1,...,3} {\draw (\i,3) -- (\i,2);}
\draw[line width=1] (1,2) -- (1,0) -- (0,0) -- (0,2) -- (3,2) -- (3,1) -- (0,1);
\draw[line width=1] (4.5,2) -- (4.5,3) -- (8.4,3) -- (8.4,2) -- (4.5,2);
\draw[line width=1] (5.8,3) -- (5.8,2); \draw[line width=1] (7.1,3) -- (7.1,2); 
\path (3.85,2.5) node {$\cdots$}; \path (5.15,2.5) node {$1,\underline{3}$};  \path (6.45,2.5) node {$5,\underline{6}$};  \path (7.75,2.5) node {$4,\underline{8}$}; 
\path (0.5,1.5) node {$\underline{2}$}; \path (1.5,1.5) node {$\underline{7}$}; \path (2.5,1.5) node {$\underline{10}$}; 
\path (0.5,0.5) node {$\underline{9}$}; 
\end{tikzpicture} } 
\end{array} \!\! \in \SSPT(\lambda, \calW^4_{\!\mathcal{B}_{10}}) 
&
\begin{array}{c}
\scalebox{0.8}{
\begin{tikzpicture}[scale=0.55] 
\fill[gray!25, draw=black] (0,2) rectangle (5.5,3);
\draw[line width=1] (2,2) -- (2,1);\draw[line width=1] (3,2) -- (3,1);
\foreach \i in {1,...,4} {\draw (\i,3) -- (\i,2);}
\draw[line width=1] (1,2) -- (1,1) -- (0,1) -- (0,2) -- (4,2) -- (4,1) -- (0,1);
\draw[line width=1] (5.5,2) -- (5.5,3) -- (9.4,3) -- (9.4,2) -- (5.5,2);
\draw[line width=1] (6.8,3) -- (6.8,2); \draw[line width=1] (8.1,3) -- (8.1,2); 
\path (4.85,2.5) node {$\cdots$}; \path (6.15,2.5) node {$2,\underline{3}$};  \path (7.45,2.5) node {$1,\underline{4}$};  \path (8.75,2.5) node {$8,\underline{9}$}; 
\path (0.5,1.5) node {$\underline{5}$}; \path (1.5,1.5) node {$\underline{6}$}; \path (2.5,1.5) node {$\underline{7}$}; \path (3.5,1.5) node {$\underline{10}$}; 
\end{tikzpicture} } 
\end{array} \!\! \in \SSPT(\lambda,\calW^4_{\!\mathcal{TL}_{10}}) 
\\
\begin{array}{c}
\scalebox{0.8}{
\begin{tikzpicture}[scale=0.55] 
\fill[gray!25, draw=black] (0,2) rectangle (3.5,3);
\foreach \i in {1,...,2} {\draw (\i,3) -- (\i,2);}
\draw[line width=1] (1,2) -- (1,0) -- (0,0) -- (0,2) -- (2,2) -- (2,1) -- (0,1);
\draw[line width=1] (3.5,2) -- (3.5,3) -- (9.1,3) -- (9.1,2) -- (3.5,2);
\foreach \i in {4.3,5.7,6.6,7.5} {\draw[line width=1] (\i,3) -- (\i,2);}
\path (2.85,2.5) node {$\cdots$}; \path (3.9,2.5) node {$\underline{2}$};  \path (4.95,2.5) node {$1,\underline{4}$};  \path (6.15,2.5) node {$\underline{5}$}; \path (7.05,2.5) node {$\underline{6}$}; \path (8.3,2.5) node {$8,\underline{10}$}; 
\path (0.5,1.5) node {$\underline{3}$}; \path (1.5,1.5) node {$\underline{9}$}; 
\path (0.5,0.5) node {$\underline{7}$}; 
\end{tikzpicture} } 
\end{array} \!\! \in \SSPT(\lambda,\calW^3_{\!\mathcal{RB}_{10}}) 
&
\begin{array}{c}
\scalebox{0.8}{
\begin{tikzpicture}[scale=0.55] 
\fill[gray!25, draw=black] (0,2) rectangle (4.5,3);
\foreach \i in {1,...,3} {\draw (\i,3) -- (\i,2);}
\draw[line width=1] (1,2) -- (1,1) -- (0,1) -- (0,2) -- (3,2) -- (3,1) -- (0,1);
\draw[line width=1] (2,2) -- (2,1);
\draw[line width=1] (4.5,2) -- (4.5,3) -- (10.1,3) -- (10.1,2) -- (4.5,2);
\foreach \i in {5.85,6.7,7.6,8.5} {\draw[line width=1] (\i,3) -- (\i,2);}
\path (3.85,2.5) node {$\cdots$}; \path (5.2,2.5) node {$3,\underline{4}$};  \path (6.3,2.5) node {$\underline{5}$};  \path (7.15,2.5) node {$\underline{8}$}; \path (8.05,2.5) node {$\underline{9}$}; \path (9.3,2.5) node {$7,\underline{10}$}; 
\path (0.5,1.5) node {$\underline{1}$}; \path (1.5,1.5) node {$\underline{2}$}; \path (2.5,1.5) node {$\underline{6}$}; 
\end{tikzpicture} } 
\end{array} \!\! \in \SSPT(\lambda,\calW^3_{\!\mathcal{M}_{10}}) 
\\
\begin{array}{c}
\scalebox{0.8}{
\begin{tikzpicture}[scale=0.55] 
\fill[gray!25, draw=black] (0,2) rectangle (4.5,3);
\foreach \i in {1,...,3} {\draw (\i,3) -- (\i,2);}
\draw[line width=1] (1,2) -- (1,0) -- (0,0) -- (0,2) -- (3,2) -- (3,1) -- (0,1);
\draw[line width=1] (2,2) -- (2,0) -- (1,0);
\draw[line width=1] (4.5,2) -- (4.5,3) -- (9,3) -- (9,2) -- (4.5,2);
\foreach \i in {5.4,6.3,7.2,8.1} {\draw[line width=1] (\i,3) -- (\i,2);}
\path (3.85,2.5) node {$\cdots$};  \path (4.95,2.5) node {$\underline{3}$};  \path (5.85,2.5) node {$\underline{4}$}; \path (6.75,2.5) node {$\underline{5}$}; \path (7.65,2.5) node {$\underline{7}$}; \path (8.55,2.5) node {$\underline{9}$}; 
\path (0.5,1.5) node {$\underline{1}$}; \path (1.5,1.5) node {$\underline{2}$};  \path (2.5,1.5) node {$\underline{8}$}; 
\path (0.5,0.5) node {$\underline{6}$};  \path (1.5,0.5) node {$\underline{10}$}; 
\end{tikzpicture} } 
\end{array} \!\! \in \SSPT(\lambda,\calW^5_{\!\mathcal{R}_{10}}) 
&
\begin{array}{c}
\scalebox{0.8}{
\begin{tikzpicture}[scale=0.55] 
\fill[gray!25, draw=black] (0,2) rectangle (4.5,3);
\foreach \i in {1,...,3} {\draw (\i,3) -- (\i,2);}
\draw[line width=1] (1,2) -- (1,1) -- (0,1) -- (0,2) -- (3,2) -- (3,1) -- (0,1);
\draw[line width=1] (2,2) -- (2,1);
\draw[line width=1] (4.5,2) -- (4.5,3) -- (10.8,3) -- (10.8,2) -- (4.5,2);
\foreach \i in {5.4,6.3,7.2,8.1,9,9.85} {\draw[line width=1] (\i,3) -- (\i,2);}
\path (3.85,2.5) node {$\cdots$};  \path (4.95,2.5) node {$\underline{1}$};  \path (5.85,2.5) node {$\underline{2}$}; \path (6.75,2.5) node {$\underline{3}$}; \path (7.65,2.5) node {$\underline{5}$}; \path (8.55,2.5) node {$\underline{6}$}; \path (9.45,2.5) node {$\underline{9}$}; \path (10.35,2.5) node {$\underline{10}$}; 
\path (0.5,1.5) node {$\underline{4}$}; \path (1.5,1.5) node {$\underline{7}$}; \path (2.5,1.5) node {$\underline{8}$}; 
\end{tikzpicture} } 
\end{array} \!\! \in \SSPT(\lambda,\calW^3_{\!\mathcal{PR}_{10}}) 
\end{array}
\end{equation*}
When restricted to the subalgebra $\A_k$, Definition~\ref{def:dT} defines an action of $\A_k$ on set-partition tableaux $\T \in \SSPT(\lambda,\Wm)$. This leads to the following theorem, whose proof is identical to that of Theorem~\ref{thm:naturalrep}.

\begin{thm}
When restricted to any of the subalgebras $\A_k$, the action~\eqref{eqn:naturalaction} defines an analogue of Young's natural representation for $\A_k$. 
\end{thm}

\begin{rem}
When $|\lambda^\ast| = k$, the standard set-partition tableaux of shape $\lambda^\ast$ have $k$ propagating singletons and no non-propagating blocks, and thus are standard Young tableaux. Furthermore, the only diagrams which are nonzero on $\A^\lambda_k$ are permutation diagrams. Upon restriction to the subalgebra $\CC\Sb_k$, the module $\A^\lambda_k$ corresponds exactly to Young's natural representation. 
\end{rem}

\begin{examp} \label{ex:dTsubalgebras}
We give examples in the Brauer, Temperley-Lieb, and Rook monoid algebras. 
\begin{enumerate}[label=(\alph*)]
\item \emph{Brauer algebra.} In the example below $d\cdot \N_\T = n \N_{d(\T)},$ where $\N_{d(\T)}$ can be re-expressed in the basis of standard tableaux using Garnir relations as in Section~\ref{subsec:symmetricgroupmodules} (in this particular case, the Garnir relation is simple: $\N_{d(\T)} = \N_{\T'}$, where $\T'$ has $7$ and $8$ switched).
\begin{equation*}
\hspace{-0.75cm}
\begin{array}{c}
\scalebox{0.95}{
\begin{tikzpicture}[scale=0.55] 
\foreach \i in {1,...,10} 
{ \path (\i-1.35,5.4) coordinate (T\i); \path (\i-1.35,7.4) coordinate (t\i); \path (\i-1.35,5) coordinate (b\i);}
\path (0.25,1.5) coordinate (B1);
\path (5.15,2.75) coordinate (B2);
\path (6.5,2.65) coordinate (B3);
\path (1.75,1.55) coordinate (B4);
\path (7.75,2.55) coordinate (B5);
\path (0.8,0.6) coordinate (B6);
\path (2.85,1.45) coordinate (B7);
\path (-1.2,6.4) node {$d=$};
\draw[rounded corners=.15mm, fill=gray!25,draw=gray!25,line width=4.5pt]  (t1) -- (t10) -- (T10) -- (T1) -- cycle;
\draw[line width = 1.35] (t5) .. controls +(.1,-.5) and +(-.1,-.5) .. (t6);
\draw[line width = 1.35] (t1) .. controls +(.1,-1) and +(-.1,-1) .. (t4);
\draw[line width = 1.35] (t9) .. controls +(.1,-.5) and +(-.1,-.5) .. (t10);
\draw[line width = 1.35] (t7) .. controls +(.1,-.5) and +(-.1,.5) .. (T10);
\draw[line width = 1.35] (t3) .. controls +(.1,-1) and +(-.5,1.5) .. (T9);
\draw[line width = 1.35] (t2) .. controls +(.1,-.5) and +(-.1,.5) .. (T1);
\draw[line width = 1.35] (t8) .. controls +(.1,-.5) and +(-.1,.5) .. (T7);
\draw[line width = 1.35] (T4) .. controls +(.1,.5) and +(-.1,.5) .. (T5);
\draw[line width = 1.35] (T2) .. controls +(.1,.5) and +(-.1,.5) .. (T3);
\draw[line width = 1.35] (T6) .. controls +(.1,.7) and +(-.1,.7) .. (T8);
\path (-1,1.5) node {$\T=$};
\fill[gray!25, draw=black] (0,2) rectangle (4.5,3);
\draw[line width=1] (2,2) -- (2,1);
\foreach \i in {1,...,3} {\draw (\i,3) -- (\i,2);}
\draw[line width=1] (1,2) -- (1,0) -- (0,0) -- (0,2) -- (3,2) -- (3,1) -- (0,1);
\draw[line width=1] (4.5,2) -- (4.5,3) -- (8.4,3) -- (8.4,2) -- (4.5,2);
\draw[line width=1] (5.8,3) -- (5.8,2); \draw[line width=1] (7.1,3) -- (7.1,2); 
\path (3.85,2.5) node {$\cdots$}; \path (5.15,2.5) node {$1,\underline{3}$};  \path (6.45,2.5) node {$5,\underline{6}$};  \path (7.75,2.5) node {$4,\underline{8}$}; 
\path (0.5,1.5) node {$\underline{2}$}; \path (1.5,1.5) node {$\underline{7}$}; \path (2.5,1.5) node {$\underline{10}$}; 
\path (0.5,0.5) node {$\underline{9}$}; 
\begin{knot}[clip width=3,end tolerance=1pt]
\strand[line width = 1] (b7) .. controls +(0,-1) and +(0,1.5) .. (B4);
\strand[line width = 1] (b5) .. controls +(.1,-.7) and +(-.1,-.7) .. (b6);
\strand[line width = 1] (b3) .. controls +(.5,-1) and +(-.5,1) .. (B2);
\strand[line width = 1] (b8) .. controls +(.5,-1) and +(-.1,1) .. (B5);
\strand[line width = 1] (b6) .. controls +(.5,-1) and +(-.5,1) .. (B3);
\strand[line width = 1] (b4) .. controls +(.1,-1.2) and +(-.1,-1.5) .. (b8);
\strand[line width = 1] (b2) .. controls +(.1,-1) and +(-.1,1) .. (B1);
\strand[line width = 1] (b1) .. controls +(.1,-.9) and +(-.1,-.9) .. (b3);
\end{knot}
\begin{knot}[clip width=3]
\strand[line width = 1] (b10)  .. controls +(3,-5) and +(1,-0.5) .. (B7);
\strand[line width = 1] (b9)  .. controls +(4,-6) and +(1,-0.5) .. (B6);
\end{knot}
\foreach \i in {1,...,7} { \fill (B\i) circle (2.5pt); } 
\foreach \i in {1,...,10} {\draw[line width=1] (b\i) -- (T\i); \fill (b\i) circle (4pt); \fill (T\i) circle (4pt); \fill (t\i) circle (4pt); 
\draw  (t\i)  node[above=0.05cm]{${\scriptstyle \i}$}; } 
\end{tikzpicture} } 
\end{array} \hspace{-1.25cm} = \begin{array}{c}
\scalebox{0.95}{
\begin{tikzpicture}[scale=0.55] 
\fill[gray!25, draw=black] (0,2) rectangle (4.5,3);
\draw[line width=1] (2,2) -- (2,1);
\foreach \i in {1,...,3} {\draw (\i,3) -- (\i,2);}
\draw[line width=1] (1,2) -- (1,0) -- (0,0) -- (0,2) -- (3,2) -- (3,1) -- (0,1);
\draw[line width=1] (4.5,2) -- (4.5,3) -- (8.7,3) -- (8.7,2) -- (4.5,2);
\draw[line width=1] (5.8,3) -- (5.8,2); \draw[line width=1] (7.1,3) -- (7.1,2); 
\path (3.85,2.5) node {$\cdots$}; \path (5.15,2.5) node {$1,\underline{4}$};  \path (6.45,2.5) node {$5,\underline{6}$};  \path (7.9,2.5) node {$9,\underline{10}$}; 
\path (0.5,1.5) node {$\underline{2}$}; \path (1.5,1.5) node {$\underline{8}$}; \path (2.5,1.5) node {$\underline{7}$}; 
\path (0.5,0.5) node {$\underline{3}$}; 
\end{tikzpicture} } 
\end{array}\hspace{-0.25cm} = d(\T).
\end{equation*}

\item \emph{Temperley-Lieb algebra.} In the example below $d\cdot \N_\T = \N_{d(\T)}.$
\begin{equation*}
\hspace{-0.75cm}
\begin{array}{c}
\scalebox{0.95}{
\begin{tikzpicture}[scale=0.55] 
\foreach \i in {1,...,10} 
{ \path (\i-0.8,5.4) coordinate (T\i); \path (\i-0.8,7.4) coordinate (t\i); \path (\i-0.8,5) coordinate (b\i);}
\path (6.15,2.65) coordinate (B1);
\path (7.45,2.65) coordinate (B2);
\path (0.75,1.65) coordinate (B3);
\path (1.8,1.65) coordinate (B4);
\path (2.75,1.55) coordinate (B5);
\path (8.75,2.6) coordinate (B6);
\path (3.85,1.85) coordinate (B7);
\path (-0.65,6.4) node {$d=$};
\draw[rounded corners=.15mm, fill=gray!25,draw=gray!25,line width=4.5pt]  (t1) -- (t10) -- (T10) -- (T1) -- cycle;
\draw[line width = 1.35] (t1) .. controls +(.1,-.5) and +(-.1,-.5) .. (t2);
\draw[line width = 1.35] (t6) .. controls +(.1,-.5) and +(-.1,-.5) .. (t7);
\draw[line width = 1.35] (t8) .. controls +(.1,-.5) and +(-.1,-.5) .. (t9);
\draw[line width = 1.35] (T4) .. controls +(.1,.5) and +(-.1,.5) .. (T5);
\draw[line width = 1.35] (T3) .. controls +(.1,1) and +(-.1,1) .. (T6);
\draw[line width = 1.35] (T9) .. controls +(.1,.5) and +(-.1,.5) .. (T10);
\draw[line width = 1.35] (t3) .. controls +(.1,-1) and +(-.1,1) .. (T1);
\draw[line width = 1.35] (t4) .. controls +(.1,-1) and +(-.1,1) .. (T2);
\draw[line width = 1.35] (t5) .. controls +(.1,-1) and +(-.1,1) .. (T7);
\draw[line width = 1.35] (t10) .. controls +(.1,-1) and +(-.1,1) .. (T8);
\path (-1,2) node {$\T=$};
\fill[gray!25, draw=black] (0,2) rectangle (5.5,3);
\draw[line width=1] (2,2) -- (2,1);\draw[line width=1] (3,2) -- (3,1);
\foreach \i in {1,...,4} {\draw (\i,3) -- (\i,2);}
\draw[line width=1] (1,2) -- (1,1) -- (0,1) -- (0,2) -- (4,2) -- (4,1) -- (0,1);
\draw[line width=1] (5.5,2) -- (5.5,3) -- (9.4,3) -- (9.4,2) -- (5.5,2);
\draw[line width=1] (6.8,3) -- (6.8,2); \draw[line width=1] (8.1,3) -- (8.1,2); 
\path (4.85,2.5) node {$\cdots$}; \path (6.15,2.5) node {$2,\underline{3}$};  \path (7.45,2.5) node {$1,\underline{4}$};  \path (8.75,2.5) node {$8,\underline{9}$}; 
\path (0.5,1.5) node {$\underline{5}$}; \path (1.5,1.5) node {$\underline{6}$}; \path (2.5,1.5) node {$\underline{7}$}; \path (3.5,1.5) node {$\underline{10}$}; 
\begin{knot}[clip width=3,end tolerance=1pt, clip radius=3pt]
\strand[line width = 1] (b2) .. controls +(.1,-.7) and +(-.1,-.7) .. (b3);
\strand[line width = 1] (b8) .. controls +(.1,-.7) and +(-.1,-.7) .. (b9);
\strand[line width = 1] (b10) .. controls +(0,-1.5) and +(0,2) .. (B7);
\strand[line width = 1] (b7)  .. controls +(0,-1) and +(0,1.5) .. (B5);
\strand[line width = 1] (b6) .. controls +(0,-1) and +(0,1.5) .. (B4);
\strand[line width = 1] (b5) .. controls +(0,-1) and +(0,1.5) .. (B3);
\strand[line width = 1] (b4) .. controls +(.5,-1) and +(-.5,1) .. (B2);
\strand[line width = 1] (b9)  .. controls +(.1,-1) and +(.1,0.5) .. (B6);
\strand[line width = 1] (b3) .. controls +(.1,-1) and +(-.1,1) .. (B1);
\strand[line width = 1] (b1) .. controls +(.1,-1.2) and +(-.1,-1.2) .. (b4);
\end{knot}
\foreach \i in {1,...,7} { \fill (B\i) circle (2.5pt); } 
\foreach \i in {1,...,10} {\draw[line width=1] (b\i) -- (T\i); }
\foreach \i in {1,...,10} {
\fill (b\i) circle (4pt); \fill (T\i) circle (4pt); \fill (t\i) circle (4pt); 
\draw  (t\i)  node[above=0.05cm]{${\scriptstyle \i}$}; } 
\end{tikzpicture} } 
\end{array} = \begin{array}{c}
\scalebox{0.95}{
\begin{tikzpicture}[scale=0.55] 
\fill[gray!25, draw=black] (0,2) rectangle (5.5,3);
\draw[line width=1] (2,2) -- (2,1);\draw[line width=1] (3,2) -- (3,1);
\foreach \i in {1,...,4} {\draw (\i,3) -- (\i,2);}
\draw[line width=1] (1,2) -- (1,1) -- (0,1) -- (0,2) -- (4,2) -- (4,1) -- (0,1);
\draw[line width=1] (5.5,2) -- (5.5,3) -- (9.4,3) -- (9.4,2) -- (5.5,2);
\draw[line width=1] (6.8,3) -- (6.8,2); \draw[line width=1] (8.1,3) -- (8.1,2); 
\path (4.85,2.5) node {$\cdots$}; \path (6.15,2.5) node {$1,\underline{2}$};  \path (7.45,2.5) node {$6,\underline{7}$};  \path (8.75,2.5) node {$8,\underline{9}$}; 
\path (0.5,1.5) node {$\underline{3}$}; \path (1.5,1.5) node {$\underline{4}$}; \path (2.5,1.5) node {$\underline{5}$}; \path (3.5,1.5) node {$\underline{10}$}; 
\end{tikzpicture} } 
\end{array}\hspace{-0.25cm} = d(\T).
\end{equation*}

\item \emph{Rook monoid algebra.} In the example below $d\cdot \N_\T =n^3 \N_{d(\T)}.$
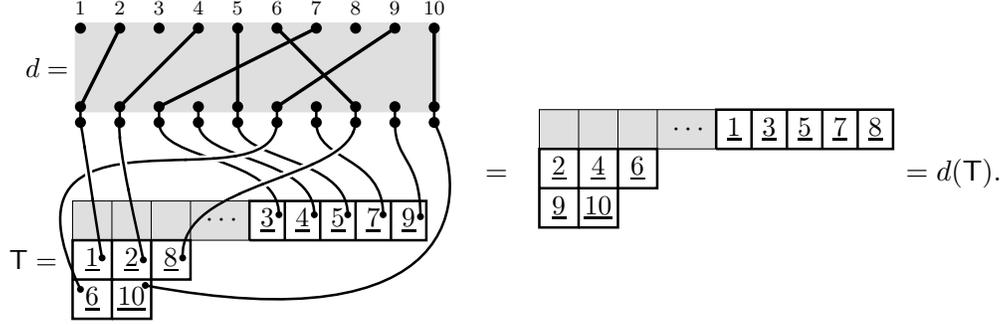
\begin{equation*}
\hspace{-0.75cm}
\begin{array}{c}
\scalebox{0.95}{
\begin{tikzpicture}[scale=0.55] 
\foreach \i in {1,...,10} 
{ \path (\i-0.8,5.4) coordinate (T\i); \path (\i-0.8,7.4) coordinate (t\i);  \path (\i-0.8,5) coordinate (b\i);}
\path (0.75,1.55) coordinate (B1);
\path (1.8,1.5) coordinate (B2);
\path (2.8,1.55) coordinate (B8);
\path (0.2,0.75) coordinate (B6);
\path (1.85,0.85) coordinate (B10);
\path (5.25,2.65) coordinate (B3);
\path (6.15,2.65) coordinate (B4);
\path (7,2.65) coordinate (B5);
\path (7.9,2.65) coordinate (B7);
\path (8.85,2.6) coordinate (B9);
\path (-0.65,6.4) node {$d=$};
\draw[rounded corners=.15mm, fill=gray!25,draw=gray!25,line width=4.5pt]  (t1) -- (t10) -- (T10) -- (T1) -- cycle;
\draw[line width = 1.35] (t2) -- (T1);
\draw[line width = 1.35] (t4) -- (T2);
\draw[line width = 1.35] (t7) -- (T3);
\draw[line width = 1.35] (t5) -- (T5);
\draw[line width = 1.35] (t9) -- (T6);
\draw[line width = 1.35] (t6) -- (T8);
\draw[line width = 1.35] (t10) -- (T10);
\path (-1,1.5) node {$\T=$};
\fill[gray!25, draw=black] (0,2) rectangle (4.5,3);
\foreach \i in {1,...,3} {\draw (\i,3) -- (\i,2);}
\draw[line width=1] (1,2) -- (1,0) -- (0,0) -- (0,2) -- (3,2) -- (3,1) -- (0,1);
\draw[line width=1] (2,2) -- (2,0) -- (1,0);
\draw[line width=1] (4.5,2) -- (4.5,3) -- (9,3) -- (9,2) -- (4.5,2);
\foreach \i in {5.4,6.3,7.2,8.1} {\draw[line width=1] (\i,3) -- (\i,2);}
\path (3.85,2.5) node {$\cdots$};  \path (4.95,2.5) node {$\underline{3}$};  \path (5.85,2.5) node {$\underline{4}$}; \path (6.75,2.5) node {$\underline{5}$}; \path (7.65,2.5) node {$\underline{7}$}; \path (8.55,2.5) node {$\underline{9}$}; 
\path (0.5,1.5) node {$\underline{1}$}; \path (1.5,1.5) node {$\underline{2}$};  \path (2.5,1.5) node {$\underline{8}$}; 
\path (0.5,0.5) node {$\underline{6}$};  \path (1.5,0.5) node {$\underline{10}$}; 
\begin{knot}[clip width=3,end tolerance=1pt]
\strand[line width = 1] (b2) .. controls +(-.1,-0.5) and +(-.1,0.5) .. (B2);
\strand[line width = 1] (b8) .. controls +(-.1,-1.5) and +(-.1,2) .. (B8);
\strand[line width = 1] (b1) .. controls +(.1,-0.5) and +(-.1,0.5) .. (B1);
\strand[line width = 1] (b6)  .. controls +(-.1,-2.25) and +(-2.5,5.25) .. (B6);
\strand[line width = 1] (b3) .. controls +(0,-1) and +(0,1.5) .. (B3);
\strand[line width = 1] (b4) .. controls +(0,-1) and +(0,1.5) .. (B4);
\strand[line width = 1] (b5)  .. controls +(0,-1) and +(0,1.5) .. (B5);
\strand[line width = 1] (b7) .. controls +(0,-1) and +(0,1.5) .. (B7);
\strand[line width = 1] (b9) .. controls +(0,-1) and +(0,1.5) .. (B9);
\strand[line width = 1] (b10)  .. controls +(1,-2) and +(8,-1.5) .. (B10);
\end{knot}
\foreach \i in {1,...,10} { \fill (B\i) circle (2.5pt); } 
\foreach \i in {1,...,10} {\draw[line width=1] (b\i) -- (T\i); \fill (b\i) circle (4pt); \fill (T\i) circle (4pt); \fill (t\i) circle (4pt); 
\draw  (t\i)  node[above=0.05cm]{${\scriptstyle \i}$}; } 
\end{tikzpicture} } 
\end{array} \hspace{-.25cm} = \begin{array}{c}
\scalebox{0.95}{
\begin{tikzpicture}[scale=0.55] 
\fill[gray!25, draw=black] (0,2) rectangle (4.5,3);
\foreach \i in {1,...,3} {\draw (\i,3) -- (\i,2);}
\draw[line width=1] (1,2) -- (1,0) -- (0,0) -- (0,2) -- (3,2) -- (3,1) -- (0,1);
\draw[line width=1] (2,2) -- (2,0) -- (1,0);
\draw[line width=1] (4.5,2) -- (4.5,3) -- (9,3) -- (9,2) -- (4.5,2);
\foreach \i in {5.4,6.3,7.2,8.1} {\draw[line width=1] (\i,3) -- (\i,2);}
\path (3.85,2.5) node {$\cdots$};  \path (4.95,2.5) node {$\underline{1}$};  \path (5.85,2.5) node {$\underline{3}$}; \path (6.75,2.5) node {$\underline{5}$}; \path (7.65,2.5) node {$\underline{7}$}; \path (8.55,2.5) node {$\underline{8}$}; 
\path (0.5,1.5) node {$\underline{2}$}; \path (1.5,1.5) node {$\underline{4}$};  \path (2.5,1.5) node {$\underline{6}$}; 
\path (0.5,0.5) node {$\underline{9}$};  \path (1.5,0.5) node {$\underline{10}$}; 
\end{tikzpicture} } 
\end{array}\hspace{-0.25cm} = d(\T).
\end{equation*}
\end{enumerate}
\end{examp}

The action of the generator $\mathfrak e_i$ on set-partition tableaux takes a nice form that can be verified using Theorem \ref{thm:generatorsactionSPT} and the relation $\mathfrak e_i = \mathfrak b_i \mathfrak p_i \mathfrak p_{i+1} \mathfrak b_i$,
\begin{equation}
\mathfrak e_i \cdot \N_\T = \begin{dcases}
n \N_{\T} & \text{if $\{i,i+1\}$ is a non-propagating block in $\T$, }\\
0 & \!\!\begin{array}{l} \text{if $i$ and $i+1$ are in propagating blocks in $\T$, or if $\{i\}$ and} \\
					 \text{$\{i+1\}$ are singleton blocks in $\T$ with one propagating,} \end{array} \\
n \N_{\mathfrak e_i(\T)} & \text{if $\{i\}$ and $\{i+1\}$ are non-propagating singleton blocks in $\T$, }\\ 
\N_{\mathfrak e_i(\T)} & \text{otherwise, } 
\end{dcases}
\end{equation}
where $\mathfrak e_i(\T)$ is the set-partition tableau in $\SPTb(\lambda,k)$ obtained from $\T$ by removing $i$ and $i+1$ from their blocks, making $\{i,i+1\}$ into a non-propagating block,  joining the remaining elements from the blocks which contained $i$ and $i+1$, and standardizing the first row.  The resulting block becomes propagating if one of the original blocks was propagating, and otherwise stays non-propagating. 

\begin{examp}  
Below are examples of the action of $\mathfrak e_i$ on set-partition tableaux of Brauer and rook-Brauer type. Consider the following set-partition tableau $\T$ in $\SSPT(\lambda, \calW^4_{\!\mathcal{B}_{10}})$, where~$\lambda = [n-4,3,1]$, 
\begin{align*}
\mathfrak e_7 \left(
\begin{array}{c}
\scalebox{0.95}{
\begin{tikzpicture}[scale=0.55] 
\fill[gray!25, draw=black] (0,2) rectangle (4.5,3);
\draw[line width=1] (2,2) -- (2,1);
\foreach \i in {1,...,3} {\draw (\i,3) -- (\i,2);}
\draw[line width=1] (1,2) -- (1,0) -- (0,0) -- (0,2) -- (3,2) -- (3,1) -- (0,1);
\draw[line width=1] (4.5,2) -- (4.5,3) -- (8.4,3) -- (8.4,2) -- (4.5,2);
\draw[line width=1] (5.8,3) -- (5.8,2); \draw[line width=1] (7.1,3) -- (7.1,2); 
\path (3.85,2.5) node {$\cdots$}; \path (5.15,2.5) node {$1,\underline{3}$};  \path (6.45,2.5) node {$5,\underline{6}$};  \path (7.75,2.5) node {$4,\underline{8}$}; 
\path (0.5,1.5) node {$\underline{2}$}; \path (1.5,1.5) node {$\underline{7}$}; \path (2.5,1.5) node {$\underline{10}$}; 
\path (0.5,0.5) node {$\underline{9}$}; 
\end{tikzpicture} } 
\end{array} \right) &= 
\begin{array}{c}
\scalebox{0.95}{
\begin{tikzpicture}[scale=0.55] 
\fill[gray!25, draw=black] (0,2) rectangle (4.5,3);
\draw[line width=1] (2,2) -- (2,1);
\foreach \i in {1,...,3} {\draw (\i,3) -- (\i,2);}
\draw[line width=1] (1,2) -- (1,0) -- (0,0) -- (0,2) -- (3,2) -- (3,1) -- (0,1);
\draw[line width=1] (4.5,2) -- (4.5,3) -- (8.4,3) -- (8.4,2) -- (4.5,2);
\draw[line width=1] (5.8,3) -- (5.8,2); \draw[line width=1] (7.1,3) -- (7.1,2); 
\path (3.85,2.5) node {$\cdots$}; \path (5.15,2.5) node {$1,\underline{3}$};  \path (6.45,2.5) node {$5,\underline{6}$};  \path (7.75,2.5) node {$7,\underline{8}$}; 
\path (0.5,1.5) node {$\underline{2}$}; \path (1.5,1.5) node {$\underline{4}$}; \path (2.5,1.5) node {$\underline{10}$}; 
\path (0.5,0.5) node {$\underline{9}$}; 
\end{tikzpicture} } 
\end{array}.
\end{align*}
Since $9$ and $10$ are distinct propagating singletons, $\mathfrak e_9$ acts as zero on $\T$. Since $\{5,6\}$ is a non-propagating block, $\mathfrak e_5 \cdot \N_\T = n \N_\T$. When $\mathfrak e_7$ acts on $\T$, $\{7,8\}$ becomes a non-propagating block and $4$ becomes a propagating singleton, so that $\mathfrak e_7 \cdot \N_\T =  \N_{\mathfrak e_7(\T)}$.

Consider the following set-partition tableau in $\SSPT(\lambda, \calW^3_{\!\mathcal{RB}_{10}})$, where $\lambda = [n-3,2,1]$,
\begin{align*}
\mathfrak e_5 \left(
\begin{array}{c}
\scalebox{0.95}{
\begin{tikzpicture}[scale=0.55] 
\fill[gray!25, draw=black] (0,2) rectangle (3.5,3);
\foreach \i in {1,...,2} {\draw (\i,3) -- (\i,2);}
\draw[line width=1] (1,2) -- (1,0) -- (0,0) -- (0,2) -- (2,2) -- (2,1) -- (0,1);
\draw[line width=1] (3.5,2) -- (3.5,3) -- (9.1,3) -- (9.1,2) -- (3.5,2);
\foreach \i in {4.3,5.7,6.6,7.5} {\draw[line width=1] (\i,3) -- (\i,2);}
\path (2.85,2.5) node {$\cdots$}; \path (3.9,2.5) node {$\underline{2}$};  \path (4.95,2.5) node {$1,\underline{4}$};  \path (6.15,2.5) node {$\underline{5}$}; \path (7.05,2.5) node {$\underline{6}$}; \path (8.3,2.5) node {$8,\underline{10}$}; 
\path (0.5,1.5) node {$\underline{3}$}; \path (1.5,1.5) node {$\underline{9}$}; 
\path (0.5,0.5) node {$\underline{7}$}; 
\end{tikzpicture} } 
\end{array} \right) &= 
\begin{array}{c}
\scalebox{0.95}{
\begin{tikzpicture}[scale=0.55] 
\fill[gray!25, draw=black] (0,2) rectangle (3.5,3);
\foreach \i in {1,...,2} {\draw (\i,3) -- (\i,2);}
\draw[line width=1] (1,2) -- (1,0) -- (0,0) -- (0,2) -- (2,2) -- (2,1) -- (0,1);
\draw[line width=1] (3.5,2) -- (3.5,3) -- (8.7,3) -- (8.7,2) -- (3.5,2);
\foreach \i in {4.3,5.7,7.1} {\draw[line width=1] (\i,3) -- (\i,2);}
\path (2.85,2.5) node {$\cdots$}; \path (3.9,2.5) node {$\underline{2}$};  \path (4.95,2.5) node {$1,\underline{4}$};\path (6.4,2.5) node {$5,\underline{6}$}; \path (7.9,2.5) node {$8,\underline{10}$}; 
\path (0.5,1.5) node {$\underline{3}$}; \path (1.5,1.5) node {$\underline{9}$}; 
\path (0.5,0.5) node {$\underline{7}$}; 
\end{tikzpicture} } 
\end{array}.
\end{align*}
Since $2$ is a non-propagating singleton and $3$ is a propagating singleton, $\mathfrak e_2$ acts as zero on $\T$. The same is true for $\mathfrak e_6$. When $\mathfrak e_5$ acts on $\T$, $\{5,6\}$ becomes a non-propagating block, and $\mathfrak e_5 \cdot \N_\T = n \N_{\mathfrak e_5(\T)}$.
\end{examp}


\section{Characters}
\label{sec:characters}
\noindent
As an application of the explicit construction of the simple module $\A^\lambda_k$, we provide a closed form for the irreducible characters of the partition algebra and its diagram subalgebras. If $d \in \calA_k$, then taking the trace in the diagram basis, with the action defined in~\eqref{eqn:twistedaction}, gives the following result. 

\begin{thm} \label{thm:char1} \setcounter{equation}{\value{equation}-1}
Let $d\in \calA_k$ be a basis diagram for $\A_k$ and let $\lambda \in \Lambda^{\A_k}_{n}$ with $|\lambda^\ast| = m$. The value of the irreducible character $\chi^\lambda_{\A_k}$ on the diagram $d \in \calA_k$ is given by 
\begin{subequations}
\begin{equation}
\chi^\lambda_{\A_k}(d) = \sum_{w \in \calF^{\,m}_{\!\calA_k}(d)} n^{\ell(d,w)} \chi^{\lambda^\ast}_{\Sb_m}(\sigma_{d,w}), \label{eqn:char1}
\end{equation}
where $n^{\ell(d,w)}$ is the number of connected components removed in the concatenation of $d$ and $w$, $\sigma_{d,w}$ is the twist of $d\circ w\circ d^T$, and $\calF^{\,m}_{\!\calA_k}(d)$ is the set of diagrams in $\Wm$ fixed under conjugation by $d$,
\begin{equation}
\calF^{\,m}_{\!\calA_k}(d) := \left\{ w \in \Wm \ \middle| \ d\circ w \circ d^T = w \right\}.
\end{equation}
\end{subequations}
\end{thm}

Let $\gamma_r$ be the $r$-cycle $(r,r-1,\ldots,1)$ in $\Sb_r \subseteq \Pb_r(n)$, and for a partition $\kappa = [\kappa_1, \kappa_2, \ldots, \kappa_\ell]$ define
\begin{equation} \label{eqn:cyclediagram}
\gamma_\kappa = \gamma_{\kappa_1} \otimes \gamma_{\kappa_2} \otimes \cdots \otimes \gamma_{\kappa_\ell},
\end{equation}
where here the tensor product denotes the juxtaposition of diagrams.  It follows from the basic construction (see Section~\ref{sec:BasicConstruction} and \cite[Lem.~2.8]{HR2}) that the irreducible characters of $\A_k$ are completely determined by their values on diagrams  $\gamma_\kappa \e_|\kappa|$, where 
\begin{subequations}
\begin{align}
\e = \frac{1}{n}\mathfrak{e}_{1} &= \frac{1}{n}\!\!
\begin{array}{c}
\scalebox{0.6}{
\begin{tikzpicture}[scale=.55,line width=1.5pt] 
 \path (0,2) coordinate (T1); \path (0,0) coordinate (B1); 
  \path (0.9,2) coordinate (T2); \path (0.9,0) coordinate (B2); 
\filldraw[fill=gray!25,draw=gray!25,line width=5pt]  (T1) -- (T2) -- (B2) -- (B1) -- (T1);
\draw (T1) .. controls +(.1,-.4) and +(-.1,-.4) .. (T2) ;
\draw (B1) .. controls +(.1,.4) and +(-.1,.4) .. (B2) ;
 \fill (T1) circle (4.5pt);\fill (B1) circle (4.5pt);
 \fill (T2) circle (4.5pt);\fill (B2) circle (4.5pt);
\end{tikzpicture} }
\end{array} \text{ and $|\kappa| + 2s = k$ for  $\B_k(n)$ and $\TL_k(n)$}, \\
\e = \frac{1}{n}\mathfrak{p}_{1} &= \frac{1}{n}
\begin{array}{c}
\scalebox{0.6}{
\begin{tikzpicture}[scale=.55,line width=1.5pt] 
 \path (0,2) coordinate (T1); \path (0,0) coordinate (B1); 
\filldraw[fill=gray!25,draw=gray!25,line width=5pt]  (T1) -- (T1) -- (B1) -- (B1) -- (T1);
 \fill (T1) circle (4.5pt);\fill (B1) circle (4.5pt);
\end{tikzpicture} }
\end{array} \text{ and $|\kappa| + s = k$ for $\Pb_k(n)$ and its other diagram subalgebras}.
\end{align}
\end{subequations}
Thus, the diagrams $\gamma_\kappa \ot \e^{\otimes s}$ are conjugacy class analogs for $\A_k$.
For example, if $k = 18$ and $\kappa = [6,5,2,1]\vdash 14$, then 
\begin{align*}
\gamma_{\kappa} \otimes \mathfrak{p}_1^{\otimes 4} =\frac{1}{n^4} \!\!
\begin{array}{c}
\scalebox{0.65}{
\begin{tikzpicture}[scale=.55,line width=1.5pt] 
\foreach \i in {1,...,18} 
{ \path (0.75*\i,2) coordinate (T\i); \path (0.75*\i,0) coordinate (B\i); } 
\filldraw[fill=gray!25,draw=gray!25,line width=4pt]  (T1) -- (T18) -- (B18) -- (B1) -- (T1);
\draw (B1) .. controls +(.1,.4) and +(-.1,-.4) .. (T6) ;
\draw (B6) -- (T5) ;
\draw (B5) -- (T4) ;
\draw (B4) -- (T3) ;
\draw (B3) -- (T2) ;
\draw (B2) -- (T1) ;
\draw (B7) .. controls +(.1,.4) and +(-.1,-.4) .. (T11) ;
\draw (B11) -- (T10) ;
\draw (B10) -- (T9) ;
\draw (B9) -- (T8) ;
\draw (B8) -- (T7) ;
\draw (B12) .. controls +(.1,.4) and +(-.1,-.4) .. (T13) ;
\draw (B14) -- (T14) ;
\draw (B13) -- (T12) ;
\foreach \i in {1,...,18} 
{ \fill (T\i) circle (4.5pt);\fill (B\i) circle (4.5pt); } 
\end{tikzpicture} }
\end{array} \text{ and } \gamma_{\kappa} \otimes \mathfrak{e}_1^{\otimes 2} =\frac{1}{n^2} \!\! 
\begin{array}{c}
\scalebox{0.65}{
\begin{tikzpicture}[scale=.55,line width=1.5pt] 
\foreach \i in {1,...,18} 
{ \path (0.75*\i,2) coordinate (T\i); \path (0.75*\i,0) coordinate (B\i); } 
\filldraw[fill=gray!25,draw=gray!25,line width=4pt]  (T1) -- (T18) -- (B18) -- (B1) -- (T1);
\draw (B1) .. controls +(.1,.4) and +(-.1,-.4) .. (T6) ;
\draw (B6) -- (T5) ;
\draw (B5) -- (T4) ;
\draw (B4) -- (T3) ;
\draw (B3) -- (T2) ;
\draw (B2) -- (T1) ;
\draw (B7) .. controls +(.1,.4) and +(-.1,-.4) .. (T11) ;
\draw (B11) -- (T10) ;
\draw (B10) -- (T9) ;
\draw (B9) -- (T8) ;
\draw (B8) -- (T7) ;
\draw (B12) .. controls +(.1,.4) and +(-.1,-.4) .. (T13) ;
\draw (B14) -- (T14) ;
\draw (B13) -- (T12) ;
\draw (T15) .. controls +(.1,-.4) and +(-.1,-.4) .. (T16) ;
\draw (B15) .. controls +(.1,.4) and +(-.1,.4) .. (B16) ;
\draw (T17) .. controls +(.1,-.4) and +(-.1,-.4) .. (T18) ;
\draw (B17) .. controls +(.1,.4) and +(-.1,.4) .. (B18) ;
\foreach \i in {1,...,18} 
{ \fill (T\i) circle (4.5pt);\fill (B\i) circle (4.5pt); } 
\end{tikzpicture} }
\end{array}.
\end{align*}
are conjugacy class representatives in $\Pb_{18}(n)$ and $\B_{18}(n)$, respectively. If the algebra $\A_k$ is planar, then the only partition $\kappa$ used is $\kappa = [1, \ldots, 1]$ so that $\gamma_\kappa$ is the identity diagram.
Furthermore from \cite[Eq.~2.17, Eq.~2.22]{HR2} and \cite[Cor.~4.2.3]{Ha}, the irreducible characters satisfy
\begin{equation}
\chi_{\A_k}^\lambda(\gamma_\kappa \ot \e^{\ot s}) = \begin{cases} 0 & \text{ if $|\kappa| < |\lambda^\ast|$, } \\ 
 \chi_{\A_{|\kappa|}}^\lambda(\gamma_\kappa) & \text{ if $|\kappa| \geq |\lambda^\ast|$. }
\end{cases} \label{eqn:charformulageneral1} 
\end{equation}
It follows from ~\ref{eqn:charformulageneral1}  that characters of $\A_k$ are determined by the characters of $\A_{k-1}$ and the values  $\chi^\lambda_{\A_k} (\gamma_\kappa)$ for $\kappa\vdash k$. When $\kappa\vdash k$, Theorem~\ref{thm:char1} simplifies to the following.

\begin{cor} \label{cor:char1} \setcounter{equation}{\value{equation}-1}
For $\lambda \in \Lambda^{\A_k}_{n}$ such that $|\lambda^\ast| = m$ and  $\kappa \vdash k$, we have
\begin{subequations}
\begin{align}
\chi^\lambda_{\A_k}(\gamma_\kappa) = \sum_{\mu \vdash m} \F_{\A_k}^{\mu,\kappa}\,\chi^{\lambda^\ast}_{\Sb_m}(\gamma_\mu), \label{eqn:char2}
\end{align}
where $\F_{\A_k}^{\mu,\kappa} := |\calF^{\,\mu}_{\!\calA_k}(\kappa)|$ is the cardinality of the following set,
\begin{equation} \label{eqn:fixedpoints}
\calF^{\,\mu}_{\!\calA_k}(\kappa) := \left\{ w \in \calW^{\calA_k}_m \ \middle| \ \gamma_\kappa \circ w \circ \gamma_\kappa^T = w, \ \sigma_{\gamma_\kappa,w} \in \Sb_m \text{ has cycle type } \mu \right\} \subseteq \calF^{\,m}_{\!\calA_k}(\gamma_\kappa).
\end{equation}
\end{subequations}
\end{cor}

\begin{proof}
Clearly $n^{\ell(\gamma_\kappa,w)} = 1$ for all $\gamma_\kappa$ and $w$, and on these special elements the sum~\eqref{eqn:char1} becomes 
\begin{align*}
\chi^\lambda_{\A_k}(\gamma_\kappa) = \sum_{w \in \calF^{\,m}_{\!\calA_k}(\gamma_\kappa)} \chi^{\lambda^\ast}_{\Sb_m}(\sigma_{\gamma_\kappa,w}) &= \sum_{\mu \vdash m}\ \sum_{w\in \calF^{\,\mu}_{\!\calA_k}(\kappa)} \chi^{\lambda^\ast}_{\Sb_m}(\gamma_\mu)   = \sum_{\mu \vdash m} \F_{\!\A_k}^{\mu,\kappa}\,\chi^{\lambda^\ast}_{\Sb_m}(\gamma_\mu), \label{eqn:char2}
\end{align*}
where in the third equality we use the fact that characters are a class function.
\end{proof}

\subsection{Fixed points $\calF^{\,\mu}_{\!\calA_k}(\kappa)$}

We now characterize the fixed diagrams $\calF^{\,\mu}_{\!\calA_k}(\kappa)$ defined in~\eqref{eqn:fixedpoints}. Many of the statements in this section are straightforward generalizations of the $m=0$ case to $m \geq 0$, and the proofs of Lemma~\ref{lemma:2a}, Proposition~\ref{prop:4a}, and Lemma~\ref{lemma:iff} are nearly identical to the proofs of Lemma 2, Proposition 4, and Lemma 6 in \cite{FaHa}. 

A symmetric $m$-diagram $w$ in $\Wm$ is determined uniquely by the set partition $\tp(w)$ making up its top row, and the $m$ blocks of $\tp(w)$ distinguished as propagating. The bottom row $\bt(w)$ is the mirror image of $\tp(w)$, so we use $\tp(w)$ to denote the set partition of both the top and bottom rows of $w$.  

\begin{lemma} \label{lemma:2a} The $k$-cycle $\gamma_k$ fixes $w \in \Wm$ if and only if the following conditions hold:
\begin{enumerate}[label={\rm(\alph*)}]
\item all of the blocks of $w$ propagate if $m>0$,
\item none of the blocks of $w$ propagate if $m=0$, and
\item $i \stackrel{w}{\sim} j \ \textrm{ if and only if } \ (i+r) \stackrel{w}{\sim} (j+r), \ \textrm{ for all } \ r \in \ZZ$, where $i+r$ and $j+r$ are computed $\mod k$
\end{enumerate}
\end{lemma}

\begin{proof}
The action of $\gamma_k$ on $w$ is to shift each vertex one step to the left, mod $k$.  Thus, if $i \stackrel{w}{\sim} j$ then $(i-1) \stackrel{\gamma_k\cdot w}{\sim} (j-1)$, viewing the subtraction mod $k$. Now, if $w \in \calF^{\calA_k}_m(\gamma_k)$, then $w = \gamma_k^r \cdot w$ for any $r\in\ZZ$. Thus $i \stackrel{w}{\sim} j$ implies $(i-r) \stackrel{w}{\sim} (j-r)$. If $w \in \calF^{\calA_k}_m(\gamma_k)$, then the blocks of $w$  either all propagate or all do not propagate, for if this were not the case, $\gamma_k$ would send a propagating block to a non-propagating block and visa versa. 
\end{proof}

\begin{definition}{\rm \cite[Def.~3]{FaHa}}
For each divisor $d$ of $k$, define the set partition $y_{d,k}$ of $\{1,\dots,k\}$ by the rule
\begin{equation*}
a \stackrel{y_{d,k}}{\sim} b \ \textrm{ if and only if } \ a\equiv b \mod d.
\end{equation*}
The set partition $y_{d,k}$ has $d$ connected components each of size $k/d$. We refer to the connected components of $y_{d,k}$ as \emph{$d$-components}. 
\end{definition}

\begin{examp}
When $k=6$ there are four set partitions $y_{d,6}$, one for each divisor of $6$. 
\begin{align*}
\begin{array}{c} 
\scalebox{.95}{
\begin{tikzpicture}[scale=.55,line width=1pt] 
\path (-0.4,1.85) node {$y_{1,6} =$};
\foreach \i in {1,...,6} 
{ \path (0.85*\i,2) coordinate (T\i); } 
\draw (T1) .. controls +(.1,-.5) and +(-.1,-.5) .. (T2) ;
\draw (T2) .. controls +(.1,-.5) and +(-.1,-.5) .. (T3) ;
\draw (T3) .. controls +(.1,-.5) and +(-.1,-.5) .. (T4) ;
\draw (T4) .. controls +(.1,-.5) and +(-.1,-.5) .. (T5) ;
\draw (T5) .. controls +(.1,-.5) and +(-.1,-.5) .. (T6) ;
\phantom{\draw (T3) .. controls +(.1,-.8) and +(-.1,-.8) .. (T6) ;}
\foreach \i in {1,...,6} 
{ \fill (T\i) circle (4.5pt); } 
\end{tikzpicture} }
\end{array} 
\begin{array}{c}
\scalebox{.95}{\begin{tikzpicture}[scale=.55,line width=1pt] 
\path (-0.4,1.85) node {$y_{2,6} =$};
\foreach \i in {1,...,6} 
{ \path (0.85*\i,2) coordinate (T\i); } 
\draw (T1) .. controls +(.1,-.6) and +(-.1,-.6) .. (T3) ;
\draw (T3) .. controls +(.1,-.6) and +(-.1,-.6) .. (T5) ;
\draw (T2) .. controls +(.1,-.6) and +(-.1,-.6) .. (T4) ;
\draw (T4) .. controls +(.1,-.6) and +(-.1,-.6) .. (T6) ;
\phantom{\draw (T3) .. controls +(.1,-.8) and +(-.1,-.8) .. (T6) ;}
\foreach \i in {1,...,6} 
{ \fill (T\i) circle (4.5pt); } 
\end{tikzpicture} }
\end{array}
\begin{array}{c}
\scalebox{.95}{\begin{tikzpicture}[scale=.55,line width=1pt] 
\path (-0.4,1.85) node {$y_{3,6} =$};
\foreach \i in {1,...,6} 
{ \path (0.85*\i,2) coordinate (T\i); } 
\draw (T1) .. controls +(.1,-.8) and +(-.1,-.8) .. (T4) ;
\draw (T2) .. controls +(.1,-.8) and +(-.1,-.8) .. (T5) ;
\draw (T3) .. controls +(.1,-.8) and +(-.1,-.8) .. (T6) ;
\phantom{\draw (T3) .. controls +(.1,-.8) and +(-.1,-.8) .. (T6) ;}
\foreach \i in {1,...,6} 
{ \fill (T\i) circle (4.5pt); } 
\end{tikzpicture} }
\end{array} 
\begin{array}{c}
\scalebox{.95}{\begin{tikzpicture}[scale=.55,line width=1pt] 
\path (-0.4,1.85) node {$y_{6,6} =$};
\phantom{\draw (T3) .. controls +(.1,-.8) and +(-.1,-.8) .. (T6) ;}
\foreach \i in {1,...,6} 
{ \path (0.85*\i,2) coordinate (T\i); } 
\foreach \i in {1,...,6} 
{ \fill (T\i) circle (4.5pt); } 
\end{tikzpicture} }
\end{array} 
\end{align*}
\end{examp}

\begin{prop} \label{prop:4a}
A diagram $w \in \Wm$ is fixed by $\gamma_k$ if and only if $\tp(w)=y_{d,k}$ for $d | k$, so that 
\begin{equation*}
\calF^{\,m}_{\!\calA_k}(\gamma_\kappa) = 
\begin{cases}
\big\{ w \in \Wm \ \big| \ \tp(w)= y_{d,k}, \text{ where $d|k$ } \big\} & \text{ if $m=0$, } \\
\big\{ w \in \Wm \ \big| \ \tp(w)= y_{m,k} \big\} & \text{ if $m>0$ and $m \mid k$, }   \\
\ \ \emptyset & \text{ if $m>0$ and $m\nmid k$. }
\end{cases}
\end{equation*}

\end{prop}
\begin{proof}
If $d|k$ and $\tp(w)=y_{d,k}$ then $w$ satisfies the conditions of Lemma~\ref{lemma:2a} and $w$ is fixed by $\gamma_k$. If $m=0$, none of the blocks propagate and we can construct $w$ from $y_{d,k}$ for any $d|k$. If $m>0$, then by Lemma~\ref{lemma:2a} $w$ must have $m$ blocks, all of which propagate, and so $\tp(w)=y_{m,k}$. Conversely, let $w \in \calF^{\,m}_{\!\calA_k}(\gamma_\kappa)$, and let $d$ be the \emph{minimum} horizontal distance between two vertices that are connected by an edge in $w$. That is,
\begin{equation*}
d = \begin{cases}
k, & \textrm{if $w$ has no horizontal connections,} \\
\min\left\{(i-j)\mod k \mid i \stackrel{w}{\sim} j, i \not=j\right\}, & \textrm{otherwise.}
\end{cases}
\end{equation*}
Choose $i$ and $j$ so that $i \stackrel{w}{\sim} j$ with $(i-j)\mod k = d$. Then by Lemma~\ref{lemma:2a}, we have $(i+r) \stackrel{w}{\sim} (j+r)$ for $0 \leq r \leq k$. Now, $d$ must divide $k$, otherwise all of the vertices of $w$ are connected implying $d=1$, which divides $k$. 
If there were a connection in $\tp(w)$ not in $y_{d,k}$, then $\tp(w)$ would connect two vertices which are closer together than $d$, contradicting the minimality of $d$. Thus $\tp(w) = y_{d,k}$.\end{proof}

\begin{lemma} \label{lemma:perm}
If $m>0$ divides $k$ and $w \in \Wm$ such that $\tp(w) =y_{m,k}$, then the permutation induced when $\gamma_k$ conjugates $w$ is $\sigma_{\gamma_k,w} = \gamma_m$, where $\gamma_m$ is the $m$-cycle $(m,m-1,\ldots,1) \in \Sb_m$. 
\end{lemma}

\begin{proof}
If $m>0$ and $w \in \Wm$, then all $m$ connected components in $w$ are fixed blocks. Using max-entry order, label these fixed blocks in increasing order mod $m$, so that $w_i < w_j$ if $i<j$. The action of $\gamma_k$ on $w$ is to shift the fixed blocks one step to the left, which shifts $w_1$ to $w_m$, $w_m$ to $w_{m-1}$, and so on, down to $w_2$ being sent to $w_1$. 
\end{proof}

\begin{examp}
Let $k=10$ and $m=5$. In the example below,  we conjugate a  $5$-diagram $w$, whose propagating blocks are $y_{5,10}$, by the cycle $\gamma_{10}$.  The induced permutation on the fixed blocks of $w$ is $\sigma_{\gamma_{10},w} = (5,4,3,2,1) = \gamma_5$. 
\begin{equation*}
\begin{array}{cc}
\gamma_{10} & = 
\begin{array}{c}
\scalebox{0.75}{
\begin{tikzpicture}[scale=.55,line width=1.5pt] 
\foreach \i in {1,...,10} 
{ \path (0.85*\i,2) coordinate (T\i); \path (0.85*\i,0) coordinate (B\i); } 
\filldraw[fill=gray!25,draw=gray!25,line width=3.2pt]  (T1) -- (T10) -- (B10) -- (B1) -- (T1);
\draw (B1) .. controls +(.1,1) and +(-.1,-1) .. (T10) ;
\draw (B10) -- (T9) ;
\draw (B9) -- (T8) ;
\draw (B8) -- (T7) ;
\draw (B7) -- (T6) ;
\draw (B6) -- (T5) ;
\draw (B5) -- (T4) ;
\draw (B4) -- (T3) ;
\draw (B3) -- (T2) ;
\draw (B2) -- (T1) ;
\foreach \i in {1,...,10} 
{ \fill (T\i) circle (4.5pt);\fill (B\i) circle (4.5pt); } 
\end{tikzpicture} } 
\end{array} \\
w & = 
\begin{array}{c}
\scalebox{0.75}{
\begin{tikzpicture}[scale=.55,line width=1.5pt]
\foreach \i in {1,...,10} 
{ \path (.85*\i,2.5) coordinate (T\i);
 \path (.85*\i,0) coordinate (B\i);}
\draw[rounded corners=.15mm, fill= gray!25!blue!10,draw=gray!25!blue!10,line width=4pt]  (T1) -- (T10) -- (B10) -- (B1) -- cycle;
\draw[blue!30!green] (T1) .. controls +(.1,-1) and +(-.1,-1) .. (T6) ;
\draw[blue!30!green] (B1) .. controls +(.1,1) and +(-.1,1) .. (B6) ;
\draw[blue!30!yellow] (T2) .. controls +(.1,-1) and +(-.1,-1) .. (T7) ;
\draw[blue!30!yellow] (B2) .. controls +(.1,1) and +(-.1,1) .. (B7) ;
\draw[blue!30!magenta] (T3) .. controls +(.1,-1) and +(-.1,-1) .. (T8) ;
\draw[blue!30!magenta] (B3) .. controls +(.1,1) and +(-.1,1) .. (B8) ;
\draw[cyan] (T4) .. controls +(.1,-1) and +(-.1,-1) .. (T9) ;
\draw[cyan] (B4) .. controls +(.1,1) and +(-.1,1) .. (B9) ;
\draw[purple] (T5) .. controls +(.1,-1) and +(-.1,-1) .. (T10) ;
\draw[purple] (B5) .. controls +(.1,1) and +(-.1,1) .. (B10) ;
\draw[blue!30!green] (T6) -- (B6);
\draw[blue!30!yellow] (T7) -- (B7);
\draw[blue!30!magenta] (T8) -- (B8);
\draw[cyan] (T9) -- (B9);
\draw[purple] (T10) -- (B10);
\foreach \i in {1,...,10} { \fill (T\i) circle (4.5pt);   \fill (B\i) circle (4.5pt);} 
\end{tikzpicture} }
\end{array} \\
\gamma_{10}^T & = 
\begin{array}{c}
\scalebox{0.75}{
\begin{tikzpicture}[scale=.55,line width=1.5pt] 
\foreach \i in {1,...,10} 
{ \path (0.85*\i,2) coordinate (B\i); \path (0.85*\i,0) coordinate (T\i); } 
\filldraw[fill=gray!25,draw=gray!25,line width=3.2pt]  (T1) -- (T10) -- (B10) -- (B1) -- (T1);
\draw (B1) .. controls +(.1,-1) and +(-.1,1) .. (T10) ;
\draw (B10) -- (T9) ;
\draw (B9) -- (T8) ;
\draw (B8) -- (T7) ;
\draw (B7) -- (T6) ;
\draw (B6) -- (T5) ;
\draw (B5) -- (T4) ;
\draw (B4) -- (T3) ;
\draw (B3) -- (T2) ;
\draw (B2) -- (T1) ;
\foreach \i in {1,...,10} 
{ \fill (T\i) circle (4.5pt);\fill (B\i) circle (4.5pt); } 
\end{tikzpicture} }
\end{array}
\end{array} = 
\begin{array}{c}
\scalebox{0.75}{
\begin{tikzpicture}[scale=.55,line width=1.5pt]
\foreach \i in {1,...,10} 
{ \path (.85*\i,2.5) coordinate (T\i);
 \path (.85*\i,0) coordinate (B\i);}
\draw[rounded corners=.15mm, fill= gray!25!blue!10,draw=gray!25!blue!10,line width=4pt]  (T1) -- (T10) -- (B10) -- (B1) -- cycle;
\draw[blue!30!yellow] (T1) .. controls +(.1,-1) and +(-.1,-1) .. (T6) ;
\draw[blue!30!yellow] (B1) .. controls +(.1,1) and +(-.1,1) .. (B6) ;
\draw[blue!30!magenta] (T2) .. controls +(.1,-1) and +(-.1,-1) .. (T7) ;
\draw[blue!30!magenta] (B2) .. controls +(.1,1) and +(-.1,1) .. (B7) ;
\draw[cyan] (T3) .. controls +(.1,-1) and +(-.1,-1) .. (T8) ;
\draw[cyan] (B3) .. controls +(.1,1) and +(-.1,1) .. (B8) ;
\draw[purple] (T4) .. controls +(.1,-1) and +(-.1,-1) .. (T9) ;
\draw[purple] (B4) .. controls +(.1,1) and +(-.1,1) .. (B9) ;
\draw[blue!30!green] (T5) .. controls +(.1,-1) and +(-.1,-1) .. (T10) ;
\draw[blue!30!green] (B5) .. controls +(.1,1) and +(-.1,1) .. (B10) ;
\draw[blue!30!green] (T10) -- (B10);
\draw[blue!30!yellow] (T6) -- (B6);
\draw[blue!30!magenta] (T7) -- (B7);
\draw[cyan] (T8) -- (B8);
\draw[purple] (T9) -- (B9);
\foreach \i in {1,...,10} { \fill (T\i) circle (4.5pt);   \fill (B\i) circle (4.5pt);} 
\end{tikzpicture} }
\end{array}.
\end{equation*}
\end{examp}

\begin{definition}{\rm \cite[Def.~5]{FaHa}}
For a partition $\kappa = [\kappa_1,\dots,\kappa_\ell]$ of $k$ and a set partition $\pi$ of $\{1,\dots,k\}$, we say that the $\kappa$-blocks of $\pi$ are the $\ell$ sub-set partitions given by grouping the elements of $\{1,\dots,k\}$ into the subsets
\begin{equation*}
\{1,\dots,\kappa_1\},\{\kappa_1+1,\dots,\kappa_1+\kappa_2\},\dots,\{\kappa_1+ \cdots + \kappa_{\ell-1}+1, \dots,k\}
\end{equation*}
where within a $\kappa$-block we inherit any connection from $\pi$, but ignore any connections between different $\kappa$-blocks. A $\kappa$-block is said to be \emph{of type $d$} if it has $d$ connected components.
\end{definition}

\begin{examp}\label{examp:Pkappa-blocks} 
Below is a set partition of $\{1,\ldots,13\}$. If $\kappa = [5,5,3]\vdash 13$, the $\kappa$-blocks of the set partition are of type $3$, $2$, and $2$, respectively. 
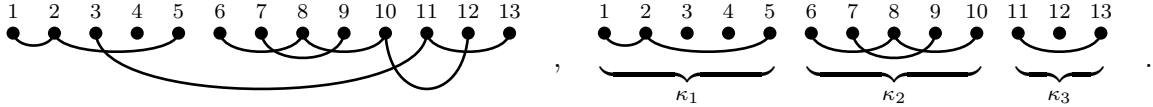
\begin{equation*}
\begin{array}{c}
\scalebox{1}{
\begin{tikzpicture}[scale=.55,line width=1pt] 
\foreach \i in {1,...,13} 
{ \path (\i,2) coordinate (T\i); } 
\draw (T1) .. controls +(.1,-.4) and +(-.1,-.4) .. (T2) ;
\draw (T2) .. controls +(.1,-.6) and +(-.1,-.6) .. (T5) ;
\draw (T6) .. controls +(.1,-.6) and +(-.1,-.6) .. (T8) ;
\draw (T7) .. controls +(.1,-.8) and +(-.1,-.8) .. (T9) ;
\draw (T8) .. controls +(.1,-.6) and +(-.1,-.6) .. (T10) ;
\draw (T11) .. controls +(.1,-.6) and +(-.1,-.6) .. (T13) ;
\draw (T3) .. controls +(.1,-1.8) and +(-.1,-1.8) .. (T11) ;
\draw (T10) .. controls +(.1,-1.8) and +(-.1,-1.8) .. (T12) ;
\foreach \i in {1,...,13} 
{ \fill (T\i) circle (4.5pt); } 
\draw  (T1)  node[above=0.05cm]{${\scriptstyle 1}$};
\draw  (T2)  node[above=0.05cm]{${\scriptstyle 2}$};
\draw  (T3)  node[above=0.05cm]{${\scriptstyle 3}$};
\draw  (T4)  node[above=0.05cm]{${\scriptstyle 4}$};
\draw  (T5)  node[above=0.05cm]{${\scriptstyle 5}$};
\draw  (T6)  node[above=0.05cm]{${\scriptstyle 6}$};
\draw  (T7)  node[above=0.05cm]{${\scriptstyle 7}$};
\draw  (T8)  node[above=0.05cm]{${\scriptstyle 8}$};
\draw  (T9)  node[above=0.05cm]{${\scriptstyle 9}$};
\draw  (T10)  node[above=0.05cm]{${\scriptstyle 10}$};
\draw  (T11)  node[above=0.05cm]{${\scriptstyle 11}$};
\draw  (T12)  node[above=0.05cm]{${\scriptstyle 12}$};
\draw  (T13)  node[above=0.05cm]{${\scriptstyle 13}$};
\path (8,0.75) node {$\phantom{\underbrace{\quad\quad\quad\quad\quad\quad}_{\kappa_2}}$};
\end{tikzpicture} }
\end{array},
\begin{array}{c}
\scalebox{1}{
\begin{tikzpicture}[scale=.55,line width=1pt] 
\foreach \i in {1,...,13} 
{ \path (\i,2) coordinate (T\i); } 
\draw (T1) .. controls +(.1,-.4) and +(-.1,-.4) .. (T2) ;
\draw (T2) .. controls +(.1,-.6) and +(-.1,-.6) .. (T5) ;
\path (3,0.75) node {$\underbrace{\quad\quad\quad\quad\quad\quad}_{\kappa_1}$};
\draw (T6) .. controls +(.1,-.6) and +(-.1,-.6) .. (T8) ;
\draw (T7) .. controls +(.1,-.8) and +(-.1,-.8) .. (T9) ;
\draw (T8) .. controls +(.1,-.6) and +(-.1,-.6) .. (T10) ;
\path (8,0.75) node {$\underbrace{\quad\quad\quad\quad\quad\quad}_{\kappa_2}$};
\draw (T11) .. controls +(.1,-.6) and +(-.1,-.6) .. (T13) ;
\path (12,0.75) node {$\underbrace{\quad\quad\quad}_{\kappa_3}$};
\foreach \i in {1,...,13} 
{ \fill (T\i) circle (4.5pt); } 
\draw  (T1)  node[above=0.05cm]{${\scriptstyle 1}$};
\draw  (T2)  node[above=0.05cm]{${\scriptstyle 2}$};
\draw  (T3)  node[above=0.05cm]{${\scriptstyle 3}$};
\draw  (T4)  node[above=0.05cm]{${\scriptstyle 4}$};
\draw  (T5)  node[above=0.05cm]{${\scriptstyle 5}$};
\draw  (T6)  node[above=0.05cm]{${\scriptstyle 6}$};
\draw  (T7)  node[above=0.05cm]{${\scriptstyle 7}$};
\draw  (T8)  node[above=0.05cm]{${\scriptstyle 8}$};
\draw  (T9)  node[above=0.05cm]{${\scriptstyle 9}$};
\draw  (T10)  node[above=0.05cm]{${\scriptstyle 10}$};
\draw  (T11)  node[above=0.05cm]{${\scriptstyle 11}$};
\draw  (T12)  node[above=0.05cm]{${\scriptstyle 12}$};
\draw  (T13)  node[above=0.05cm]{${\scriptstyle 13}$};
\end{tikzpicture} }
\end{array}.
\end{equation*}
\end{examp}

\begin{lemma} \label{lemma:iff} 
Let $\kappa = [\kappa_1, \dots, \kappa_\ell] \vdash k$ and $\mu = [\mu_1,\dots,\mu_s]\vdash m \leq k$. Then $w \in \Wm$ is in $\calF^{\,\mu}_{\!\calA_k}(\kappa)$ if and only if the following conditions hold:
\begin{enumerate}[label=\rm{(\alph*)}]
\item for each $i$, the $\kappa_i$-block of $\tp(w)$ is of the form $y_{d_i,\kappa_i}$ for some divisor $d_i$ of $\kappa_i$,
\item if a $\kappa_i$-block of type $d_i$ and a $\kappa_j$-block of type $d_j$ have connections between them in $w$, then
\begin{enumerate}[label=\rm{(\roman*)}]
\item $d_i = d_j$, 
\item each $d_i$-component of the $\kappa_i$-block is connected to a unique $d_i$-component of the $\kappa_j$-block,
\item there are no further connections between these two blocks,
\end{enumerate}
\item for each $i$, there are $\m_i(\mu)$ sets of connected $\kappa$-blocks of type $i$ that propagate, where $\m_i(\mu)$ is the multiplicity of $i$ in $\mu$.
\end{enumerate}
\end{lemma}

\begin{proof}
(a) When $\gamma_\kappa$ acts on $w$, the cycle $\gamma_{\kappa_i}$ acts on the $\kappa_i$-block, so by Proposition~\ref{prop:4a} the $\kappa_i$-block must be of the form $y_{d_i,\kappa_i}$ for some divisor $d_i$ of $\kappa_i$. 

(b) If a $d_i$-component in the $\kappa_i$-block is connected to two $d_j$-components $d_j^1$ and $d_j^2$ in the $\kappa_j$-block, then by transitivity $d_j^1 \sim d_j^2$. Thus, each $d_i$-component is connected to a unique $d_j$-component. When $\gamma_\kappa$ acts on $w$ it permutes the $d_i$-components in the $\kappa_i$-block and the $d_j$-components in the $\kappa_j$-block. If $d_j > d_i$ then $\gamma_\kappa$ sends a $d_j$-component that is not connected to the $\kappa_i$-block to a $d_j$-component that is connected to the $\kappa_i$-block, which cannot happen. The same is true when $d_j < d_i$, and thus $d_i = d_j$. There can be no further connections between blocks because that would  force two components in one to be connected to a single component in the other. 

(c) Now suppose that a set of $\kappa$-blocks of type $i$ are connected to each other and all propagate. Then there are $i$ propagating edges from the rightmost $\kappa$-block in the set, and by Lemma~\ref{lemma:perm}, when $\gamma_\kappa$ acts on $w$ the $i$ edges are permuted according to the $i$-cycle $\gamma_i$. Hence, if for each $i$ there are $\m_i(\mu)$ sets of connected $\kappa$-blocks of type $i$ that propagate then there are $\sum_{i=1}^m \m_i(\mu) =m$ propagating blocks, which are permuted by $\sigma_{\gamma_\kappa,w} = \gamma_{i_1} \ot \gamma_{i_2} \ot \cdots \ot \gamma_{i_s}$, where $i_j \in \mu$. Clearly $\sigma_{\gamma_\kappa,w}$ has cycle type $\mu$, and any other choices for the number and type of propagating blocks gives a different cycle type, so that $w \in \calF^{\,\nu}_{\!\calA_k}(\kappa)$ for $\nu \not=\mu$. 
\end{proof}

\begin{examp} \label{ex:fixedpoints} 
Let $\kappa = [6,5,3,3,2,2]\vdash 21$ and $\mu = [3,3,2]\vdash 8$. The following diagram is fixed under conjugation by $\gamma_\kappa$, and the permutation of the fixed blocks is $\sigma_{\gamma_\kappa,w_1} = (3,2,1)(6,5,4)(7,8) \in~\Sb_8$.
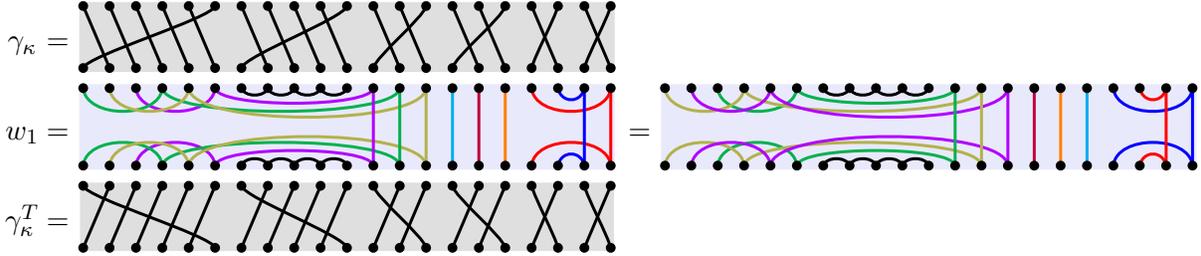
\begin{equation*}
\begin{array}{cc}
\gamma_\kappa \!\!\!\! & =  \!\!\!\!
\begin{array}{c}
\scalebox{0.75}{
\begin{tikzpicture}[scale=.55,line width=1.5pt] 
\foreach \i in {1,...,21} 
{ \path (0.85*\i,2) coordinate (T\i); \path (0.85*\i,0) coordinate (B\i); } 
\filldraw[fill=gray!25,draw=gray!25,line width=4pt]  (T1) -- (T21) -- (B21) -- (B1) -- (T1);
\draw (B1) .. controls +(.1,.4) and +(-.1,-.4) .. (T6) ;
\draw (B6) -- (T5) ;
\draw (B5) -- (T4) ;
\draw (B4) -- (T3) ;
\draw (B3) -- (T2) ;
\draw (B2) -- (T1) ;
\draw (B7) .. controls +(.1,.4) and +(-.1,-.4) .. (T11) ;
\draw (B11) -- (T10) ;
\draw (B10) -- (T9) ;
\draw (B9) -- (T8) ;
\draw (B8) -- (T7) ;
\draw (B12) .. controls +(.1,.4) and +(-.1,-.4) .. (T14) ;
\draw (B14) -- (T13) ;
\draw (B13) -- (T12) ;
\draw (B15) .. controls +(.1,.4) and +(-.1,-.4) .. (T17) ;
\draw (B17) -- (T16) ;
\draw (B16) -- (T15) ;
\draw (B18) .. controls +(.1,.4) and +(-.1,-.4) .. (T19) ;
\draw (B19) -- (T18) ;
\draw (B20) .. controls +(.1,.4) and +(-.1,-.4) .. (T21) ;
\draw (B21) -- (T20) ;
\foreach \i in {1,...,21} 
{ \fill (T\i) circle (4.5pt);\fill (B\i) circle (4.5pt); } 
\end{tikzpicture} } 
\end{array} \\
w_1  \!\!\!\!& =  \!\!\!\!
\begin{array}{c}
\scalebox{0.75}{
\begin{tikzpicture}[scale=.55,line width=1.5pt]
\foreach \i in {1,...,21} 
{ \path (0.85*\i,2.5) coordinate (T\i);
 \path (0.85*\i,0) coordinate (B\i);}
\draw[rounded corners=.15mm, fill= gray!25!blue!10,draw=gray!25!blue!10,line width=4pt]  (T1) -- (T21) -- (B21) -- (B1) -- cycle;
\draw[blue!30!magenta] (T3) .. controls +(.1,-1) and +(-.1,-1) .. (T6) ;
\draw[blue!30!magenta] (B3) .. controls +(.1,1) and +(-.1,1) .. (B6) ;
\draw[blue!30!magenta] (T6) .. controls +(.1,-.65) and +(-.1,-.65) .. (T12) ;
\draw[blue!30!magenta] (B6) .. controls +(.1,.65) and +(-.1,.65) .. (B12) ;
\draw[blue!30!green] (T1) .. controls +(.1,-1) and +(-.1,-1) .. (T4) ;
\draw[blue!30!green] (B1) .. controls +(.1,1) and +(-.1,1) .. (B4) ;
\draw[blue!30!green] (T4) .. controls +(.1,-1) and +(-.1,-1) .. (T13) ;
\draw[blue!30!green] (B4) .. controls +(.1,1) and +(-.1,1) .. (B13) ;
\draw[blue!30!yellow] (T2) .. controls +(.1,-1) and +(-.1,-1) .. (T5) ;
\draw[blue!30!yellow] (B2) .. controls +(.1,1) and +(-.1,1) .. (B5) ;
\draw[blue!30!yellow] (T5) .. controls +(.1,-1.25) and +(-.1,-1.25) .. (T14) ;
\draw[blue!30!yellow] (B5) .. controls +(.1,1.25) and +(-.1,1.25) .. (B14) ;
\draw[blue!30!magenta] (T12) -- (B12);
\draw[blue!30!green] (T13) -- (B13);
\draw[blue!30!yellow] (T14) -- (B14);
\draw (T7) .. controls +(.1,-.3) and +(-.1,-.3) .. (T8);
\draw (T8) .. controls +(.1,-.3) and +(-.1,-.3) .. (T9);
\draw (T9) .. controls +(.1,-.3) and +(-.1,-.3) .. (T10);
\draw (T10) .. controls +(.1,-.3) and +(-.1,-.3) .. (T11);
\draw (B7) .. controls +(.1,.3) and +(-.1,.3) .. (B8);
\draw (B8) .. controls +(.1,.3) and +(-.1,.3) .. (B9);
\draw (B9) .. controls +(.1,.3) and +(-.1,.3) .. (B10);
\draw (B10) .. controls +(.1,.3) and +(-.1,.3) .. (B11);
\draw[cyan] (T15) -- (B15);
\draw[purple] (T16) -- (B16);
\draw[orange] (T17) -- (B17);
\draw[blue] (T19) .. controls +(.1,-.5) and +(-.1,-.5) .. (T20);
\draw[blue] (B19) .. controls +(.1,.5) and +(-.1,.5) .. (B20);
\draw[red] (T18) .. controls +(.1,-1) and +(-.1,-1) .. (T21);
\draw[red] (B18) .. controls +(.1,1) and +(-.1,1) .. (B21);
\draw[blue] (T20) -- (B20);
\draw[red] (T21) -- (B21);
\foreach \i in {1,...,21} { \fill (T\i) circle (4.5pt);   \fill (B\i) circle (4.5pt);} 
\end{tikzpicture} }
\end{array} \\
\gamma_\kappa^T \!\!\!\! & = \!\!\!\!
\begin{array}{c}
\scalebox{0.75}{
\begin{tikzpicture}[scale=.55,line width=1.5pt] 
\foreach \i in {1,...,21} 
{ \path (0.85*\i,2) coordinate (B\i); \path (0.85*\i,0) coordinate (T\i); } 
\filldraw[fill=gray!25,draw=gray!25,line width=3.2pt]  (T1) -- (T21) -- (B21) -- (B1) -- (T1);
\draw (B1) .. controls +(.1,-.4) and +(-.1,.4) .. (T6) ;
\draw (B6) -- (T5) ;
\draw (B5) -- (T4) ;
\draw (B4) -- (T3) ;
\draw (B3) -- (T2) ;
\draw (B2) -- (T1) ;
\draw (B7) .. controls +(.1,-.4) and +(-.1,.4) .. (T11) ;
\draw (B11) -- (T10) ;
\draw (B10) -- (T9) ;
\draw (B9) -- (T8) ;
\draw (B8) -- (T7) ;
\draw (B12) .. controls +(.1,-.4) and +(-.1,.4) .. (T14) ;
\draw (B14) -- (T13) ;
\draw (B13) -- (T12) ;
\draw (B15) .. controls +(.1,-.4) and +(-.1,.4) .. (T17) ;
\draw (B17) -- (T16) ;
\draw (B16) -- (T15) ;
\draw (B18) .. controls +(.1,-.4) and +(-.1,.4) .. (T19) ;
\draw (B19) -- (T18) ;
\draw (B20) .. controls +(.1,-.4) and +(-.1,.4) .. (T21) ;
\draw (B21) -- (T20) ;
\foreach \i in {1,...,21} 
{ \fill (T\i) circle (4.5pt);\fill (B\i) circle (4.5pt); } 
\end{tikzpicture} }
\end{array}
\end{array} \!\!\!\!\!\!=\!\!\!\! 
\begin{array}{c}
\scalebox{0.75}{
\begin{tikzpicture}[scale=.55,line width=1.5pt]
\foreach \i in {1,...,21} 
{ \path (.85*\i,2.5) coordinate (T\i);
 \path (.85*\i,0) coordinate (B\i);}
\draw[rounded corners=.15mm, fill= gray!25!blue!10,draw=gray!25!blue!10,line width=4pt]  (T1) -- (T21) -- (B21) -- (B1) -- cycle;
\draw[blue!30!green] (T3) .. controls +(.1,-1) and +(-.1,-1) .. (T6) ;
\draw[blue!30!green] (B3) .. controls +(.1,1) and +(-.1,1) .. (B6) ;
\draw[blue!30!green] (T6) .. controls +(.1,-.65) and +(-.1,-.65) .. (T12) ;
\draw[blue!30!green] (B6) .. controls +(.1,.65) and +(-.1,.65) .. (B12) ;
\draw[blue!30!yellow] (T1) .. controls +(.1,-1) and +(-.1,-1) .. (T4) ;
\draw[blue!30!yellow] (B1) .. controls +(.1,1) and +(-.1,1) .. (B4) ;
\draw[blue!30!yellow] (T4) .. controls +(.1,-1) and +(-.1,-1) .. (T13) ;
\draw[blue!30!yellow] (B4) .. controls +(.1,1) and +(-.1,1) .. (B13) ;
\draw[blue!30!magenta] (T2) .. controls +(.1,-1) and +(-.1,-1) .. (T5) ;
\draw[blue!30!magenta] (B2) .. controls +(.1,1) and +(-.1,1) .. (B5) ;
\draw[blue!30!magenta] (T5) .. controls +(.1,-1.25) and +(-.1,-1.25) .. (T14) ;
\draw[blue!30!magenta] (B5) .. controls +(.1,1.25) and +(-.1,1.25) .. (B14) ;
\draw[blue!30!green] (T12) -- (B12);
\draw[blue!30!yellow] (T13) -- (B13);
\draw[blue!30!magenta] (T14) -- (B14);
\draw (T7) .. controls +(.1,-.3) and +(-.1,-.3) .. (T8);
\draw (T8) .. controls +(.1,-.3) and +(-.1,-.3) .. (T9);
\draw (T9) .. controls +(.1,-.3) and +(-.1,-.3) .. (T10);
\draw (T10) .. controls +(.1,-.3) and +(-.1,-.3) .. (T11);
\draw (B7) .. controls +(.1,.3) and +(-.1,.3) .. (B8);
\draw (B8) .. controls +(.1,.3) and +(-.1,.3) .. (B9);
\draw (B9) .. controls +(.1,.3) and +(-.1,.3) .. (B10);
\draw (B10) .. controls +(.1,.3) and +(-.1,.3) .. (B11);
\draw[purple] (T15) -- (B15);
\draw[orange] (T16) -- (B16);
\draw[cyan] (T17) -- (B17);
\draw[red] (T19) .. controls +(.1,-.5) and +(-.1,-.5) .. (T20);
\draw[red] (B19) .. controls +(.1,.5) and +(-.1,.5) .. (B20);
\draw[blue] (T18) .. controls +(.1,-1) and +(-.1,-1) .. (T21);
\draw[blue] (B18) .. controls +(.1,1) and +(-.1,1) .. (B21);
\draw[red] (T20) -- (B20);
\draw[blue] (T21) -- (B21);
\foreach \i in {1,...,21} { \fill (T\i) circle (4.5pt);   \fill (B\i) circle (4.5pt);} 
\end{tikzpicture} }
\end{array}.
\end{equation*}
The following diagrams are also in $\calF^{\,\mu}_{\!\calP_{21}}(\kappa)$, 
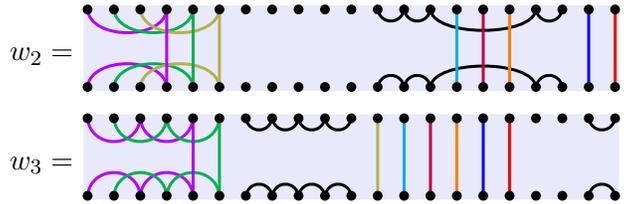
\begin{align*}
w_2 &= \!\!\!\!
\begin{array}{c}
\scalebox{0.75}{
\begin{tikzpicture}[scale=.55,line width=1.5pt]
\foreach \i in {1,...,21} 
{ \path (.85*\i,2.5) coordinate (T\i);
 \path (.85*\i,0) coordinate (B\i);}
\draw[rounded corners=.15mm, fill= gray!25!blue!10,draw=gray!25!blue!10,line width=4pt]  (T1) -- (T21) -- (B21) -- (B1) -- cycle;
\draw[blue!30!magenta] (T1) .. controls +(.1,-1) and +(-.1,-1) .. (T4) ;
\draw[blue!30!magenta] (B1) .. controls +(.1,1) and +(-.1,1) .. (B4) ;
\draw[blue!30!green] (T2) .. controls +(.1,-1) and +(-.1,-1) .. (T5) ;
\draw[blue!30!green] (B2) .. controls +(.1,1) and +(-.1,1) .. (B5) ;
\draw[blue!30!yellow] (T3) .. controls +(.1,-1) and +(-.1,-1) .. (T6) ;
\draw[blue!30!yellow] (B3) .. controls +(.1,1) and +(-.1,1) .. (B6) ;
\draw[blue!30!magenta] (T4) -- (B4);
\draw[blue!30!green] (T5) -- (B5);
\draw[blue!30!yellow] (T6) -- (B6);
\draw (T12) .. controls +(.1,-.5) and +(-.1,-.5) .. (T13);
\draw (T13) .. controls +(.1,-.5) and +(-.1,-.5) .. (T14);
\draw (B12) .. controls +(.1,.5) and +(-.1,.5) .. (B13);
\draw (B13) .. controls +(.1,.5) and +(-.1,.5) .. (B14);
\draw (T14) .. controls +(.1,-.9) and +(-.1,-.9) .. (T18);
\draw (B14) .. controls +(.1,.9) and +(-.1,.9) .. (B18);
\draw (T18) .. controls +(.1,-.5) and +(-.1,-.5) .. (T19);
\draw (B18) .. controls +(.1,.5) and +(-.1,.5) .. (B19);
\draw[cyan] (T15) -- (B15);
\draw[purple] (T16) -- (B16);
\draw[orange] (T17) -- (B17);
\draw[blue] (T20) -- (B20);
\draw[red] (T21) -- (B21);
\foreach \i in {1,...,21} { \fill (T\i) circle (4.5pt);   \fill (B\i) circle (4.5pt);} 
\end{tikzpicture} } 
\end{array},\\
 w_3 &= \!\!\!\!
\begin{array}{c}
\scalebox{0.75}{
\begin{tikzpicture}[scale=.55,line width=1.5pt]
\foreach \i in {1,...,21} 
{ \path (.85*\i,2.5) coordinate (T\i);
 \path (.85*\i,0) coordinate (B\i);}
\draw[rounded corners=.15mm, fill= gray!25!blue!10,draw=gray!25!blue!10,line width=4pt]  (T1) -- (T21) -- (B21) -- (B1) -- cycle;
\draw[blue!30!magenta] (T1) .. controls +(.1,-1) and +(-.1,-1) .. (T3) ;
\draw[blue!30!magenta] (B1) .. controls +(.1,1) and +(-.1,1) .. (B3) ;
\draw[blue!30!magenta] (T3) .. controls +(.1,-1) and +(-.1,-1) .. (T5) ;
\draw[blue!30!magenta] (B3) .. controls +(.1,1) and +(-.1,1) .. (B5) ;
\draw[blue!30!green] (T2) .. controls +(.1,-1) and +(-.1,-1) .. (T4) ;
\draw[blue!30!green] (B2) .. controls +(.1,1) and +(-.1,1) .. (B4) ;
\draw[blue!30!green] (T4) .. controls +(.1,-1) and +(-.1,-1) .. (T6) ;
\draw[blue!30!green] (B4) .. controls +(.1,1) and +(-.1,1) .. (B6) ;
\draw[blue!30!magenta] (T5) -- (B5);
\draw[blue!30!green] (T6) -- (B6);
\draw[blue!30!yellow] (T12) -- (B12);
\draw[cyan] (T13) -- (B13);
\draw[purple] (T14) -- (B14);
\draw[orange] (T15) -- (B15);
\draw[blue] (T16) -- (B16);
\draw[red] (T17) -- (B17);
\draw (T7) .. controls +(.1,-.5) and +(-.1,-.5) .. (T8);
\draw (T8) .. controls +(.1,-.5) and +(-.1,-.5) .. (T9);
\draw (T9) .. controls +(.1,-.5) and +(-.1,-.5) .. (T10);
\draw (T10) .. controls +(.1,-.5) and +(-.1,-.5) .. (T11);
\draw (B7) .. controls +(.1,.5) and +(-.1,.5) .. (B8);
\draw (B8) .. controls +(.1,.5) and +(-.1,.5) .. (B9);
\draw (B9) .. controls +(.1,.5) and +(-.1,.5) .. (B10);
\draw (B10) .. controls +(.1,.5) and +(-.1,.5) .. (B11);
\draw (T20) .. controls +(.1,-.5) and +(-.1,-.5) .. (T21);
\draw (B20) .. controls +(.1,.5) and +(-.1,.5) .. (B21);
\foreach \i in {1,...,21} { \fill (T\i) circle (4.5pt);   \fill (B\i) circle (4.5pt);} 
\end{tikzpicture} } 
\end{array},
\end{align*}
with permutations $\sigma_{\gamma_\kappa,w_2} = (3,2,1)(6,5,4)(7,8)$ and $\sigma_{\gamma_\kappa,w_3} = (1,2)(5,4,3)(8,7,6)$, respectively, which both have cycle type $[3,3,2]$. It is easy to verify that these three diagrams satisfy the properties of Lemma~\ref{lemma:iff}. 
\end{examp}

\subsection{The partition algebra}

We now count the number of symmetric $m$-diagrams in $\calW^m_{\!\calP_k}$ that satisfy the conditions of Lemma~\ref{lemma:iff}. 

\begin{definition} \label{def:divisor}
Let $\kappa=[\kappa_1,\ldots,\kappa_\ell]$ be an integer partition of $k$ into $\ell$ parts. We say that a \emph{divisor} of $\kappa$ is a composition $\nu = [\nu_1, \ldots, \nu_\ell]$ such that $\nu_i | \kappa_i$ for all $i=1,\ldots, \ell$, and we let $\nu | \kappa$ indicate that $\nu$ is a divisor of $\kappa$. 
\end{definition}

\begin{examp}
The following diagrams depict the eight divisors of $\kappa = [6,5,1]\vdash 12$.
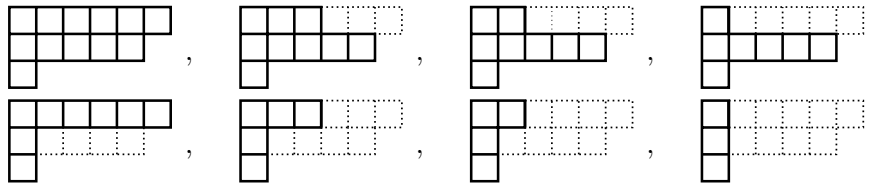
\begin{equation*}
\begin{array}{c c c c} 
\begin{array}{c}
\scalebox{0.65}{
\begin{tikzpicture}[scale=.55, line width=1.5]
\draw (1,3) -- (1,0) -- (0,0) -- (0,3) -- (6,3) -- (6,2) -- (0,2);
\draw (5,3) -- (5,1) -- (0,1);
\draw (2,1) -- (2,3); \draw (3,1) -- (3,3); \draw (4,1) -- (4,3);
\end{tikzpicture}} 
\end{array}, & 
\begin{array}{c}
\scalebox{0.65}{
\begin{tikzpicture}[scale=.55, line width=1.5]
\draw (1,3) -- (1,0) -- (0,0) -- (0,3) -- (3,3) -- (3,2) -- (0,2);
\draw (3,2) -- (5,2) -- (5,1) -- (0,1);
\draw (2,1) -- (2,3);  \draw (3,1) -- (3,2);  \draw (4,1) -- (4,2); 
\draw[dotted, line width = 1] (3,3) -- (6,3) -- (6,2) -- (5,2) -- (5,3);
\draw[dotted, line width = 1] (4,2) -- (4,3); 
\end{tikzpicture}} 
\end{array}, &
\begin{array}{c}
\scalebox{0.65}{
\begin{tikzpicture}[scale=.55, line width=1.5]
\draw (1,3) -- (1,0) -- (0,0) -- (0,3) -- (2,3) -- (2,2) -- (0,2);
\draw (2,2) -- (5,2) -- (5,1) -- (0,1);
\draw (2,1) -- (2,3);  \draw (3,1) -- (3,2);  \draw (4,1) -- (4,2); 
\draw[dotted, line width = 1] (2,3) -- (6,3) -- (6,2) -- (5,2) -- (5,3);
\draw[dotted, line width = 1] (4,2) -- (4,3); \draw[dotted, line width = 0.55] (3,2) -- (3,3); 
\end{tikzpicture}} 
\end{array}, &
\begin{array}{c}
\scalebox{0.65}{
\begin{tikzpicture}[scale=.55, line width=1.5]
\draw (1,3) -- (1,0) -- (0,0) -- (0,3) -- (1,3) -- (1,2) -- (0,2);
\draw (1,2) -- (5,2) -- (5,1) -- (0,1);
\draw (2,1) -- (2,2);  \draw (3,1) -- (3,2);  \draw (4,1) -- (4,2); 
\draw[dotted, line width = 1] (1,3) -- (6,3) -- (6,2) -- (5,2) -- (5,3);
\draw[dotted, line width = 1] (4,2) -- (4,3); \draw[dotted, line width = 1] (3,2) -- (3,3); \draw[dotted, line width = 1] (2,2) -- (2,3); 
\end{tikzpicture}} 
\end{array},  \\
\begin{array}{c}
\scalebox{0.65}{
\begin{tikzpicture}[scale=.55, line width=1.5]
\draw (1,3) -- (1,0) -- (0,0) -- (0,3) -- (6,3) -- (6,2) -- (0,2);
\draw (0,1) -- (1,1);
\draw (2,2) -- (2,3); \draw (3,2) -- (3,3); \draw (4,2) -- (4,3); \draw (5,2) -- (5,3);
\draw[dotted, line width = 1] (5,2) -- (5,1) -- (1,1); \draw[dotted, line width = 1] (2,1) -- (2,2); \draw[dotted, line width = 1] (3,1) -- (3,2); \draw[dotted, line width = 1] (4,1) -- (4,2);
\end{tikzpicture}} 
\end{array}, & 
\begin{array}{c}
\scalebox{0.65}{
\begin{tikzpicture}[scale=.55, line width=1.5]
\draw (1,3) -- (1,0) -- (0,0) -- (0,3) -- (3,3) -- (3,2) -- (0,2);
\draw (1,1) -- (0,1);
\draw (2,2) -- (2,3); \draw (3,2) -- (3,3); 
\draw[dotted, line width = 1] (5,2) -- (5,1) -- (1,1); \draw[dotted, line width = 1] (2,1) -- (2,2); \draw[dotted, line width = 1] (3,1) -- (3,2) -- (5,2); \draw[dotted, line width = 1] (4,1) -- (4,2);
\draw[dotted, line width = 1] (3,3) -- (6,3) -- (6,2) -- (5,2) -- (5,3);
\draw[dotted, line width = 1] (4,2) -- (4,3); 
\end{tikzpicture}} 
\end{array}, &
\begin{array}{c}
\scalebox{0.65}{
\begin{tikzpicture}[scale=.55, line width=1.5]
\draw (1,3) -- (1,0) -- (0,0) -- (0,3) -- (2,3) -- (2,2) -- (0,2);
\draw (1,1) -- (0,1);
\draw (2,2) -- (2,3); 
\draw[dotted, line width = 1] (2,3) -- (6,3) -- (6,2) -- (5,2) -- (5,3);
\draw[dotted, line width = 1] (4,1) -- (4,3); \draw[dotted, line width = 1] (3,1) -- (3,3); 
\draw[dotted, line width = 1] (5,2) -- (5,1) -- (1,1); \draw[dotted, line width = 1] (2,1) -- (2,2) -- (5,2); 
\end{tikzpicture}} 
\end{array}, &
\begin{array}{c}
\scalebox{0.65}{
\begin{tikzpicture}[scale=.55, line width=1.5]
\draw (1,3) -- (1,0) -- (0,0) -- (0,3) -- (1,3) -- (1,2) -- (0,2);
\draw (1,1) -- (0,1);
\draw[dotted, line width = 1] (1,3) -- (6,3) -- (6,2) -- (5,2) -- (5,3);
\draw[dotted, line width = 1] (4,1) -- (4,3); \draw[dotted, line width = 1] (3,1) -- (3,3); 
\draw[dotted, line width = 1] (1,2) -- (5,2) -- (5,1) -- (1,1); \draw[dotted, line width = 1] (2,1) -- (2,3); 
\end{tikzpicture}} 
\end{array}. 
\end{array}
\end{equation*}

\end{examp}

\begin{examp}
Each diagram in $\calF^{\,\mu}_{\calP_k}(\kappa)$ determines a divisor of $\kappa$: by Lemma~\ref{lemma:iff} (a) the blocks of $w$ are of the form $y_{d_i,\kappa_i}$ where $d_i | \kappa_i$, and the collection $\nu = [d_1,\ldots,d_\ell]$ is a divisor of $\kappa$. For the three diagrams shown in Example~\ref{ex:fixedpoints}, the corresponding divisors of $\kappa$ are 
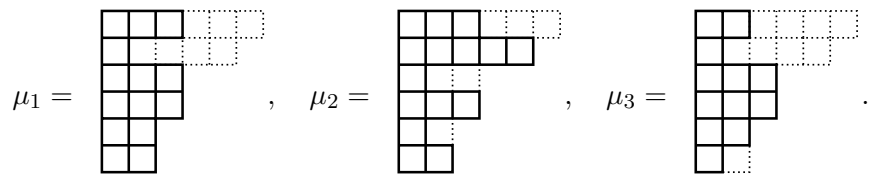
\begin{equation*}
\mu_1 = 
\begin{array}{c}
\scalebox{0.65}{
\begin{tikzpicture}[scale=.55, line width=1.5]
\draw (0,5) -- (3,5) -- (3,6) -- (0,6) -- (0,0) -- (2,0) -- (2,4) -- (3,4) -- (3,2) -- (0,2);
\draw (1,6) -- (1,0); \draw (0,1) -- (2,1); \draw (0,3) -- (3,3); \draw (0,4) -- (2,4); \draw (2,6) -- (2,5);
\draw[dotted, line width = 1] (3,6) -- (6,6) -- (6,5) -- (3,5);
\draw[dotted, line width = 1] (5,6) -- (5,4) -- (2,4) -- (2,5) -- (3,5) -- (3,4);
\draw[dotted, line width = 1] (4,4) -- (4,6);
\end{tikzpicture}} 
\end{array} \!\!, \quad
\mu_2 =  
\begin{array}{c}
\scalebox{0.65}{
\begin{tikzpicture}[scale=.55, line width=1.5]
\draw (0,5) -- (5,5) -- (5,4) -- (3,4) -- (3,6) -- (0,6) -- (0,0) -- (2,0) -- (2,1) -- (0,1);
\draw (1,0) -- (1,6); \draw (0,2) -- (3,2) -- (3,3) -- (0,3); \draw (2,6) -- (2,4); \draw (0,4) -- (3,4); \draw (2,2) -- (2,3); \draw (4,4) -- (4,5);
\draw[dotted, line width = 1] (3,6) -- (6,6) -- (6,5) -- (3,5);
\draw[dotted, line width = 1] (4,6) -- (4,5); 
\draw[dotted, line width = 1] (5,6) -- (5,5);
\draw[dotted, line width = 1] (2,1) -- (2,2); 
\draw[dotted, line width = 1] (2,3) -- (2,4); 
\draw[dotted, line width = 1] (3,3) -- (3,4);
\end{tikzpicture}} 
\end{array} \!\!, \quad
\mu_3 = 
\begin{array}{c}
\scalebox{0.65}{
\begin{tikzpicture}[scale=.55, line width=1.5]
\draw (0,5) -- (2,5) -- (2,6) -- (0,6) -- (0,0) -- (1,0) -- (1,6);
\draw (0,4) -- (3,4) -- (3,2) -- (0,2); \draw (2,4) -- (2,1) -- (0,1); \draw (0,3) -- (3,3);
\draw[dotted, line width = 1] (2,5) -- (6,5) -- (6,6) -- (0,6);
\draw[dotted, line width = 1] (2,5) -- (2,6); 
\draw[dotted, line width = 1] (3,5) -- (3,6); 
\draw[dotted, line width = 1] (4,5) -- (4,6); 
\draw[dotted, line width = 1] (5,5) -- (5,6);
\draw[dotted, line width = 1] (1,0) -- (2,0) -- (2,1);
\draw[dotted, line width = 1] (3,4) -- (5,4) -- (5,5);
\draw[dotted, line width = 1] (2,4) -- (2,5); 
\draw[dotted, line width = 1] (3,4) -- (3,5); 
\draw[dotted, line width = 1] (4,4) -- (4,5);
\end{tikzpicture}} 
\end{array} \!\!.
\end{equation*}
\end{examp}

\begin{prop} \label{prop:coeff} 
Let $\kappa\vdash k$ and $\mu\vdash m$. The number of diagrams in $\calF^{\,\mu}_{\calP_k}(\kappa)$ is given by
\begin{equation*}
\F_{\Pb_k(n)}^{\mu,\kappa} = \sum_{\nu | \kappa}  \prod_{i} \sum_{t} \stirling{\m_i(\nu)}{t}\binom{t}{\m_i(\mu)} i^{\m_i(\nu)-t},
\end{equation*}
where the outer sum is over divisors of $\kappa$ and $\m_i(\nu)$ is the number of parts of size $i$ in $\nu$. 
\end{prop}

\begin{proof}
Given a divisor $\nu | \kappa$ consider the symmetric diagram $w$ whose $\kappa_i$ block is of the form $y_{\nu_i, \kappa_i}$. We count the number of ways of making $w$ into a symmetric $m$-diagram in $\calF^{\,\mu}_{\calP_k}(\kappa)$. By Lemma~\ref{lemma:iff} (c), for some $i$ and $w \in \calF^{\,\mu}_{\calP_k}(\kappa)$ there must be $\m_i(\mu)$ sets of connected $\kappa$-blocks of type $i$ that propagate. The number of available $\kappa$-blocks of type $i$ in $w$ is given by $\m_i(\nu)$. If $\m_i(\nu) < \m_i(\mu)$, there are not enough $\kappa$-blocks of type $i$ to propagate and the sum gives zero. Suppose $\m_i(\nu) \geq \m_i(\mu)$. To construct $w$ we choose a set partition of these $\m_i(\nu)$ $\kappa$-blocks of type $i$ into $t$ blocks, where $\m_i(\mu) \leq t \leq \m_i(\nu)$. There are $\stirling{\m_i(\nu)}{t}$  ways to do this. Then we  choose $\m_i(\mu)$ of these blocks to propagate in $\binom{t}{\m_i(\mu)}$ ways. The remaining blocks do not propagate. 

We now count the number of ways of connecting the individual $\kappa$-blocks. There are $i$ ways of connecting two $\kappa$-blocks of type $i$. For instance, there are three ways of connecting two $\kappa$-blocks of type $3$:
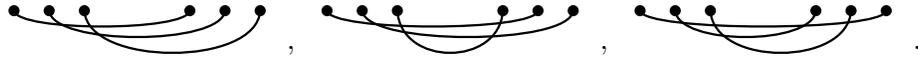
\begin{align*}
\begin{array}{c}
\scalebox{0.85}{
\begin{tikzpicture}[scale=.55,line width=1pt] 
\foreach \i in {1,2,3,6,7,8} 
{ \path (\i,2) coordinate (T\i); } 
\draw (T1) .. controls +(.1,-.6) and +(-.1,-.6) .. (T6) ;
\draw (T2) .. controls +(.1,-1) and +(-.1,-1) .. (T7) ;
\draw (T3) .. controls +(.1,-1.6) and +(-.1,-1.6) .. (T8) ;
\foreach \i in {1,2,3,6,7,8} 
{ \fill (T\i) circle (4.5pt); } 
\end{tikzpicture} }
\end{array}, 
\begin{array}{c}
\scalebox{0.85}{
\begin{tikzpicture}[scale=.55,line width=1pt] 
\foreach \i in {1,2,3,6,7,8}
{ \path (\i,2) coordinate (T\i); } 
\draw (T1) .. controls +(.1,-.6) and +(-.1,-.6) .. (T7) ;
\draw (T2) .. controls +(.1,-1) and +(-.1,-1) .. (T8) ;
\draw (T3) .. controls +(.1,-1.6) and +(-.1,-1.6) .. (T6) ;
\foreach \i in {1,2,3,6,7,8} 
{ \fill (T\i) circle (4.5pt); } 
\end{tikzpicture} }
\end{array},
\begin{array}{c}
\scalebox{0.85}{
\begin{tikzpicture}[scale=.55,line width=1pt] 
\foreach \i in {1,2,3,6,7,8} 
{ \path (\i,2) coordinate (T\i); } 
\draw (T1) .. controls +(.1,-.6) and +(-.1,-.6) .. (T8) ;
\draw (T2) .. controls +(.1,-1) and +(-.1,-1) .. (T6) ;
\draw (T3) .. controls +(.1,-1.6) and +(-.1,-1.6) .. (T7) ;
\foreach \i in {1,2,3,6,7,8} 
{ \fill (T\i) circle (4.5pt); } 
\end{tikzpicture} }
\end{array}.
\end{align*}
A set partition of $\m_i(\nu)$ $\kappa$-blocks of type $i$ into $t$ blocks can be depicted as a one-line set-partition diagram where each edge is labelled by the $i$ ways to connect two $\kappa$-blocks. For instance, if $i=3$ and $\m_i(\nu) = 5$, then the following represents a set partition of $5$ $\kappa$-blocks of type $3$ into two blocks:
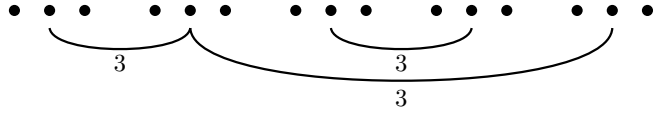
\begin{equation*}
\begin{array}{c}
\scalebox{0.85}{
\begin{tikzpicture}[scale=.55,line width=1pt] 
\foreach \i in {1,2,3,5,6,7,9,10,11,13,14,15,17,18,19} 
{ \path (\i,2) coordinate (B\i);
\path (\i,1.5) coordinate (T\i); } 
\draw (T2) .. controls +(.1,-0.8) and +(-.1,-0.8) .. (T6) ;
\draw (T6) .. controls +(.1,-2) and +(-.1,-2) .. (T18) ;
\draw (T10) .. controls +(.1,-0.8) and +(-.1,-0.8) .. (T14) ;
\path (4,0.5) node {$3$};
\path (12,0.5) node {$3$};
\path (12,-0.5) node {$3$};
\foreach \i in {1,2,3,5,6,7,9,10,11,13,14,15,17,18,19} 
{ \fill (B\i) circle (4.5pt); } 
\end{tikzpicture} }
\end{array}.
\end{equation*}
Thus, the number of ways of connecting $\m_i(\nu)$ $\kappa$-blocks of type $i$ into $t$ blocks is given by $i^{\m_i(\nu)-t}$. The inner sum is over $\m_i(\mu) \leq t \leq \m_i(\nu)$, but $\stirling{\m_i(\nu)}{t}\binom{t}{\m_i(\mu)}$ is identically zero outside this interval, so we can sum over all $t$. The connections between (and propagation of) blocks of each type are independent, so taking the product over all $i$ and summing over the divisors $\nu$ of $\kappa$ completes the proof. 
\end{proof}

\begin{examp} \label{ex:coeff}
Let $\kappa = [2,1]$ and $\mu = [1]$. The two divisors of $\kappa$ are $\kappa$ itself and the trivial divisor $\nu = [1,1]$. The number of symmetric $1$-diagrams in $\calF^{\,\mu}_{\calP_3}(\kappa)$ is 
\begin{equation*}
\left(\sum_{t=1}^1 \stirling{1}{t}\binom{t}{1}1^{1-t}\right)\left(\sum_{t=0}^1 \stirling{1}{t}\binom{t}{0}2^{1-t}\right) + \sum_{t=1}^2 \stirling{2}{t}\binom{t}{1}1^{2-t} = (1)(0+1) + (1+2) = 4. 
\end{equation*}
Indeed, 
\begin{equation*} 
\calF^{\,\mu}_{\calP_3}(\kappa) = \left\{
\begin{array}{c}
\scalebox{0.8}{
\begin{tikzpicture}[scale=.55,line width=1.35pt]
\foreach \i in {1,...,3} 
{ \path (\i,1.75) coordinate (T\i);
 \path (\i,0) coordinate (B\i);}
\draw[rounded corners=.15mm,  fill= gray!25,draw=gray!25,line width=4pt]  (T1) -- (T3) -- (B3) -- (B1) -- cycle;
\draw (T3) -- (B3);
\foreach \i in {1,...,3} { \fill (T\i) circle (4.5pt);   \fill (B\i) circle (4.5pt);} 
\end{tikzpicture} }
\end{array}, 
\begin{array}{c}
\scalebox{0.8}{
\begin{tikzpicture}[scale=.55,line width=1.35pt]
\foreach \i in {1,...,3} 
{ \path (\i,1.75) coordinate (T\i);
 \path (\i,0) coordinate (B\i);}
\draw[rounded corners=.15mm,  fill= gray!25,draw=gray!25,line width=4pt]  (T1) -- (T3) -- (B3) -- (B1) -- cycle;
\draw (T1) .. controls +(.1,-.6) and +(-.1,-.6) .. (T2);
\draw (B1) .. controls +(.1,.6) and +(-.1,.6) .. (B2);
\draw (T2) .. controls +(.1,-.6) and +(-.1,-.6) .. (T3) ;
\draw (B2) .. controls +(.1,.6) and +(-.1,.6) .. (B3) ;
\draw (T3) -- (B3);
\foreach \i in {1,...,3} { \fill (T\i) circle (4.5pt);   \fill (B\i) circle (4.5pt);} 
\end{tikzpicture} }
\end{array}, 
\begin{array}{c}
\scalebox{0.8}{
\begin{tikzpicture}[scale=.55,line width=1.35pt]
\foreach \i in {1,...,3} 
{ \path (\i,1.75) coordinate (T\i);
 \path (\i,0) coordinate (B\i);}
\draw[rounded corners=.15mm,  fill= gray!25,draw=gray!25,line width=4pt]  (T1) -- (T3) -- (B3) -- (B1) -- cycle;
\draw (T1) .. controls +(.1,-.6) and +(-.1,-.6) .. (T2);
\draw (B1) .. controls +(.1,.6) and +(-.1,.6) .. (B2);
\draw (T2) -- (B2);
\foreach \i in {1,...,3} { \fill (T\i) circle (4.5pt);   \fill (B\i) circle (4.5pt);} 
\end{tikzpicture} }
\end{array},
\begin{array}{c}
\scalebox{0.8}{
\begin{tikzpicture}[scale=.55,line width=1.35pt]
\foreach \i in {1,...,3} 
{ \path (\i,1.75) coordinate (T\i);
 \path (\i,0) coordinate (B\i);}
\draw[rounded corners=.15mm, fill= gray!25,draw=gray!25,line width=4pt]  (T1) -- (T3) -- (B3) -- (B1) -- cycle;
\draw (T1) .. controls +(.1,-.6) and +(-.1,-.6) .. (T2);
\draw (B1) .. controls +(.1,.6) and +(-.1,.6) .. (B2);
\draw (T3) -- (B3);
\foreach \i in {1,...,3} { \fill (T\i) circle (4.5pt);   \fill (B\i) circle (4.5pt);} 
\end{tikzpicture} }
\end{array} \right\}. 
\end{equation*}
The first diagram corresponds to the divisor $\kappa$ while the others correspond to the divisor $\nu$. This coefficient appears in the factorization of the character table for $\Pb_3(n)$ in Example~\ref{ex:chartables} (a). 
\end{examp}

Combining Proposition~\ref{prop:coeff} with~\eqref{eqn:charformulageneral1}  gives our main result.

\begin{thm} \label{thm:chars} 
If $\lambda$ is a partition of $n$ such that $|\lambda^\ast| = m$, and $\kappa$ is an integer partition such that $|\kappa| \leq k$, then
\begin{equation*}
\chi_{\Pb_k(n)}^\lambda(\gamma_\kappa \ot \e^{\ot s}) = \sum_{\substack{\mu \vdash m}} \sum_{\nu | \kappa} \prod_{i} \sum_{t} \stirling{\m_i(\nu)}{t}\binom{t}{\m_i(\mu)}i^{\m_i(\nu)-t} \,\chi_{\Sb_m}^{\lambda^\ast}(\gamma_\mu),
\end{equation*}
where the first sum is over partitions $\mu$ of $m$, the second sum is over divisors $\nu$ of $\kappa$, and $\m_i(\nu)$ is the number of parts of $\nu$ equal to $i$. 
\end{thm}

\begin{rems} \label{rems:specialcases2} \setcounter{equation}{\value{equation}-1}
\hspace{0.5cm}
\begin{enumerate}[label=\rm{(R\arabic*)}]
\item A recursive Murnaghan-Nakayama rule for computing $\chi^\lambda_{\Pb_k(n)}(\gamma_\kappa \ot \e^{\ot s})$ is given in \cite{Ha}. The closed formula in Theorem~\ref{thm:chars} in terms of the symmetric group character $\chi^{\lambda^\ast}_{\Sb_m}$ is new. 

\item When $|\lambda^\ast|=k$, the only divisor of $\kappa$ that contributes to the sum is $\kappa$ itself, and the only partition $\mu \vdash k$ that contributes to the sum is $\mu =\kappa$. Hence 
\begin{equation*}
\chi^{\lambda}_{\Pb_k(n)}(\gamma_\kappa) = \prod_{i} \stirling{\m_i(\kappa)}{\m_i(\kappa)}\binom{\m_i(\kappa)}{\m_i(\kappa)} \chi^{\lambda^*}_{\Sb_k}(\gamma_\kappa) = \chi^{\lambda^*}_{\Sb_k}(\gamma_\kappa).
\end{equation*} 

\item When $|\lambda^\ast| = 0$,  we have  $\mu=\emptyset$ and $\chi^{\emptyset}_{\Sb_0}(\gamma_{\emptyset}) = 1$, so the character formula specializes to 
\begin{equation*} 
\chi^{\lambda}_{\Pb_k(n)}(\gamma_\kappa \ot \e^{\ot s}) = \sum_{\nu | \kappa} \prod_{i}\sum_{t \geq 0} \stirling{\m_i(\nu)}{t}i^{\m_i(\nu)-t}.
\end{equation*}
This is a new formula for this  character value, which is studied in \cite[Thm.~9]{FaHa} and used there to  prove a ``second orthogonality relation" for the characters of $\Pb_k(n)$ and compute the joint mixed moments of the number of fixed points of $\sigma^i$ for $\sigma \in \Sb_n$. 
\end{enumerate}
\end{rems}

\subsection{Subalgebras}

We now count the number of symmetric $m$-diagrams in $\Wm$ that satisfy the conditions of Lemma~\ref{lemma:iff}, where $\A_k$ is one of the subalgebras of $\Pb_k(n)$. We first consider the non-planar algebras, giving new character formulas for the Brauer and rook-Brauer algebras, and the known character formula obtained in \cite[Prop.~3.5]{So} for the rook monoid algebra. 

\begin{thm} If $\A_k$ is one of the non-planar subalgebras of $\Pb_k(n)$ and $\lambda \in \Lambda^{\A_k}_n$ with $|\lambda^\ast| = m$, then 
\begin{equation*}
\chi^\lambda_{\A_k}(\gamma_\kappa \ot \e^{\ot s})  = \sum_{\mu \vdash m} \F_{\!\A_k}^{\mu,\kappa}\,\chi^{\lambda^\ast}_{\Sb_m}(\gamma_\mu),
\end{equation*}
where $\kappa$ is a partition such that $|\kappa| + 2s = k$ for the Brauer algebra and $|\kappa| + s = k$ for the others.
The coefficients $\F_{\!\A_k}^{\mu,\kappa}$ for the Brauer, rook-Brauer, and rook monoid algebras, respectively, are given by,
\begin{enumerate}[label=\rm{(\alph*)}]
\item \hspace{.01cm} $\F_{\B_k(n)}^{\mu,\kappa} \hspace{.05cm} =\hspace{.075cm} \displaystyle{\prod_{i} \binom{\m_i(\kappa)}{\m_i(\mu)} 
\sum_t \binom{\ds_i(\kappa,\mu)}{2t} (2t-1)!!\ i^{t}\left(\frac{1+(-1)^i}{2}\right)^{\ds_i(\kappa,\mu)-2t},}$

\item $\F_{\RB_k(n)}^{\mu,\kappa} = \displaystyle{\prod_{i} \binom{\m_i(\kappa)}{\m_i(\mu)}\sum_t \binom{\ds_i(\kappa,\mu)}{2t} (2t-1)!!\ i^{t}\left(\frac{3+(-1)^i}{2}\right)^{\ds_i(\kappa,\mu)-2t},}$

\item $\F_{\R_k}^{\mu,\kappa} = \displaystyle{\prod_{i} \binom{\m_i(\kappa)}{\m_i(\mu)},}$

\end{enumerate}
where $\ds_i(\kappa,\mu) = \m_i(\kappa)-\m_i(\mu)$ and for {\rm (a)} we adopt the convention $0^0=1$ as in \cite{Knu}.
\end{thm}

\begin{proof}
We count the number of symmetric $m$-diagrams in $\calF^{\,\mu}_{\!\calA_k}(\kappa)$ for each algebra. For each of the proper subalgebras of $\Pb_k(n)$, the propagating blocks of a symmetric $m$-diagram are identity edges, so the propagating $\kappa$-blocks are of the form $y_{\kappa_i,\kappa_i}$. Consider first the Brauer algebra. There are $\m_i(\kappa)$ blocks which can become propagating blocks of type $i$. If $\m_i(\kappa) < \m_i(\mu)$, there are not enough blocks of type $i$ to propagate and the coefficient is zero. If $\m_i(\kappa) \geq \m_i(\mu)$, then we choose $\m_i(\mu)$ of these to propagate. There are $\ds_i(\kappa,\mu) = \m_i(\kappa)-\m_i(\mu)$ blocks of type $i$ remaining. The non-propagating blocks in Brauer diagrams have size two, and if $i$ is even, we can choose $2t$ of the remaining blocks to pair up in $(2t-1)!!$ ways, where a given pair can be connected in $i^t$ ways. The remaining $\ds_i(\kappa,\mu)-2t$ blocks are not paired up, and are made into blocks of type $i/2$ by pairing up vertices within each block. If $i$ is odd and $\ds_i(\kappa,\mu)$ is even we pair all $\ds_i(\kappa,\mu)$ blocks together, which can happen in $(\ds_i(\kappa,\mu)-1)!!\ i^{\ds_i(\kappa,\mu)/2}$ ways. If both $i$ and $\ds_i(\kappa,\mu)$ are odd, there are zero ways of pairing the non-propagating vertices. Taking the product over all values of $i$ gives the result. 

For the rook-Brauer algebra it is possible to have non-propagating blocks of size one. If $i$ is even, each of the $\ds_i(\kappa,\mu)-2t$ blocks designated as non-propagating can either consist of singletons or pairs, so there are $2^{\ds_i(\kappa,\mu)-2t}$ configurations for the non-propagating blocks. If $i$ is odd, all of the non-propagating blocks must be singletons, so there is only one choice. 

Finally, for the rook monoid algebra, all of the non-propagating blocks are singletons, so for each $i$ we simply choose $\m_i(\mu)$ blocks of type $i$ to propagate. 
\end{proof}

The characters of the planar subalgebras are determined by their values on the identity diagram $\mathbf{1}_r$, for $r \leq k$ (see~\eqref{eqn:charformulageneral1}). It follows that the set of fixed points  equals the set of symmetric $m$-diagrams. This gives the known character formulas obtained in \cite[Sec.~2]{HR2} for the Temperley-Lieb algebra, in \cite[Sec.~4.3]{BH-motzkin} for the Motzkin algebra, and in \cite[Sec.~5]{FHH} for the planar rook monoid algebra. 

\begin{thm} If $\A_k$ is one of the planar subalgebras of $\Pb_k(n)$ and $\lambda \in \Lambda^{\A_k}_n$ with $|\lambda^\ast| = m$, then 
\begin{equation*}
\chi^\lambda_{\A_k}({{\bf 1}_r} \ot \e^{\ot s})  =  \F_{\!\A_k}^{m,r},
\end{equation*}
where $r + 2s = k$ for the Temperley-Lieb algebra and $r + s = k$ for the others. The coefficients  $\F_{\!\A_k}^{m,r}$ for the Temperley-Lieb, Motzkin, and planar rook monoid algebras, respectively, are given by,
\begin{enumerate}[label=\rm{(\alph*)}]
\item $\F_{\TL_k(n)}^{m,r} =  \displaystyle{ \binom{r}{\frac{r-m}{2}} - \binom{r}{\frac{r-m}{2}-1}},$

\item $\F_{\Motz_k(n)}^{m,r} = \displaystyle{\sum_t \binom{r}{m+2t}\left( \binom{m+2t}{t} - \binom{m+2t}{t-1}\right)},$

\item $\F_{\PR_k}^{m,r} =  \displaystyle{\binom{r}{m}}.$

\end{enumerate}

\end{thm}

\begin{proof} The proof is by counting symmetric $m$-diagrams in the planar algebras, which  is done  in \cite[Sec.~5.5--5.7]{HReeks}.
\end{proof}

\subsection{Character tables}

When viewed as a matrix, the character table of $\A_k$, denoted $\Xi_{\A_k}$, can be expressed as the product of a direct sum of character tables $\Xi_{\Sb_m}$ for $0 \leq m \leq k$ and the matrix $\F_{\!\A_k}$, whose $\mu,\kappa$ entry is $\F_{\!\A_k}^{\mu,\kappa}$. It is clear from the definitions above that, in all cases, $\F_{\!\A_k}$ is unitriangular (with respect to lexicographic order on partitions) with entries in $\ZZ_{\geq 0}$ and has determinant equal to one. As a result, the absolute value of the determinant of the character table $\Xi_{\A_k}$ is equal to the absolute value of the product of determinants of symmetric group character tables $\Xi_{\Sb_m}$. In \cite{Ja} and \cite{ScSi} it is shown that the absolute value of the determinant of $\Xi_{\Sb_m}$ is equal to the product of all parts of all partitions of $m$:
$
|\det \Xi_{\Sb_m}| = \prod_{\mu\vdash m} \prod_{i} i^{\m_i(\mu)}.
$
This leads to the following result. 

\begin{prop}
Let $\A_k$ be any of the diagram algebras above, and let $\Xi_{\A_k}$ denote the character table of $\A_k$ viewed as a matrix with integer entries. Then 
\begin{equation*}
|\det \Xi_{\A_k}| = \begin{cases}
\displaystyle \prod_{\lambda}\prod_{i} i^{\m_i(\lambda^\ast)} & \text{if $\A_k$ is non-planar, }  \\
\quad 1 & \text{if $\A_k$ is planar, } 
\end{cases}
\end{equation*}
where the product is over partitions $\lambda \in \Lambda^{\A_k}_n$. 
\end{prop}

We conclude the section by providing examples of character tables for the non-planar algebras. 

\begin{examp} \label{ex:chartables}
In the following examples, the rows of $\Xi_{\A_k}$ are indexed by the irreducible $\A_k$-modules, which are labelled by partitions $\lambda \in \Lambda^{\A_k}_n$, and the columns are indexed by conjugacy class analogs, which are labelled by partitions of $0, \ldots, k$. Both are arranged in lexicographic order. For example, the rows of $\Xi_{\Pb_3(n)}$ are indexed by 
$\{ [n], [n-1,1], [n-2,2] [n-2,1,1],[n-3,3],[n-3,2,1],[n-3,1,1,1]\}$ and the columns are indexed by $\{ \emptyset, [1], [2], [1,1],  [3],[2,1],[1,1,1]\}$.

\begin{enumerate}[label=(\alph*)]
\item \emph{The partition algebra, $\Pb_3(n)$.  Note that the entry $\F^{[1],[2,1]}_{\Pb_3(n)}=4$ is computed in Example~\ref{ex:coeff}:}
\begin{equation*}
\underbrace{\begin{pmatrix}
  1 & 1 & 2 & 2 & 2 & 3 & 5 \\
\cdot    & 1 & 1 & 3 & 1 & 4 & 10 \\
\cdot   & \cdot & 1  & 1 & 0 & 2 & 6 \\
\cdot   & \cdot & -1 & 1 & 0 & 0 & 6 \\
\cdot  & \cdot & \cdot & \cdot & 1 & 1 & 1 \\
\cdot   &\cdot  & \cdot & \cdot & -1 & 0 & 2 \\
\cdot  & \cdot &  \cdot&  \cdot& 1 & -1 & 1
\end{pmatrix}}_{\Xi_{\Pb_3(n)}} = 
\underbrace{\begin{pmatrix}
 1 & \cdot &\cdot  & \cdot & \cdot &\cdot  & \cdot \\
\cdot  & 1 & \cdot &\cdot  &\cdot  & \cdot &  \cdot\\
\cdot  &\cdot  & 1  & 1 & \cdot &  \cdot& \cdot \\
\cdot & \cdot & -1 & 1 & \cdot &\cdot  &  \cdot\\
\cdot & \cdot & \cdot & \cdot & 1 & 1 & 1 \\
\cdot & \cdot & \cdot & \cdot & -1 & 0 & 2 \\
\cdot  & \cdot& \cdot & \cdot & 1 & -1 & 1
\end{pmatrix}}_{\Xi_{\Sb_0}\oplus \Xi_{\Sb_1}\oplus \Xi_{\Sb_2}\oplus \Xi_{\Sb_3} }
\underbrace{\begin{pmatrix}
1 & 1 & 2 & 2 & 2 & 3 & 5 \\
 \cdot   & 1 & 1 & 3 & 1 & \mathbf{4} & 10 \\
 \cdot &\cdot    & 1 & 0 & 0 & 1 & 0 \\
 \cdot &\cdot  &\cdot  & 1 & 0 & 1 & 6 \\
 \cdot &\cdot  &\cdot  &\cdot  & 1 & 0 & 0 \\
 \cdot &\cdot  &\cdot  &\cdot  &\cdot  & 1 & 0 \\
 \cdot &\cdot  &\cdot  & \cdot &\cdot  &\cdot  & 1
\end{pmatrix}}_{\F_{\Pb_3(n)}}
\end{equation*}

\item \emph{The rook-Brauer algebra, $\RB_3(n)$:}
\begin{equation*}
\underbrace{\begin{pmatrix}
  1 & 1 & 2 & 2 & 1 & 2 & 4 \\
   \cdot   & 1 & 0 & 2 & 0 & 2 & 6 \\
  \cdot   &  \cdot  & 1  & 1 & 0 & 1 & 3 \\
   \cdot  & \cdot   & -1 & 1 & 0 & -1 & 3 \\
  \cdot  &  \cdot &  \cdot  &  \cdot  & 1 & 1 & 1 \\
  \cdot   &  \cdot  &  \cdot  &  \cdot  & -1 & 0 & 2 \\
  \cdot  &  \cdot  &  \cdot  &  \cdot  & 1 & -1 & 1
\end{pmatrix}}_{\Xi_{\RB_3(n)}} = 
\underbrace{\begin{pmatrix}
 1 & \cdot &\cdot  & \cdot & \cdot &\cdot  & \cdot \\
\cdot  & 1 & \cdot &\cdot  &\cdot  & \cdot &  \cdot\\
\cdot  &\cdot  & 1  & 1 & \cdot &  \cdot& \cdot \\
\cdot & \cdot & -1 & 1 & \cdot &\cdot  &  \cdot\\
\cdot & \cdot & \cdot & \cdot & 1 & 1 & 1 \\
\cdot & \cdot & \cdot & \cdot & -1 & 0 & 2 \\
\cdot  & \cdot& \cdot & \cdot & 1 & -1 & 1
\end{pmatrix}}_{\Xi_{\Sb_0}\oplus \Xi_{\Sb_1}\oplus \Xi_{\Sb_2}\oplus \Xi_{\Sb_3} }
\underbrace{\begin{pmatrix}
1 & 1 & 2 & 2 & 1 & 2 & 4 \\
   \cdot   & 1 & 0 & 2 & 0 & 2 & 6 \\
 \cdot   &   \cdot   & 1 & 0 & 0 & 1 & 0 \\
 \cdot   &  \cdot  &  \cdot  & 1 & 0 & 0 & 3 \\
  \cdot  &   \cdot &  \cdot  &  \cdot  & 1 & 0 & 0 \\
  \cdot  &  \cdot  &  \cdot  &  \cdot  & \cdot   & 1 & 0 \\
  \cdot  &  \cdot  & \cdot   &  \cdot  &  \cdot  &  \cdot  & 1
\end{pmatrix}}_{\F_{\RB_3(n)}}
\end{equation*}

\item \emph{The rook monoid algebra, $\R_3$:}
\begin{equation*}
\underbrace{\begin{pmatrix}
 1 & 1 & 1 & 1 & 1 & 1 & 1 \\
  \cdot  & 1 & 0 & 2 & 0 & 1 & 3 \\
 \cdot   &  \cdot  & 1 & 1 & 0 & 1 & 3 \\
  \cdot  &  \cdot  & -1 & 1 & 0 & -1 & 3 \\
 \cdot  & \cdot   &  \cdot  &  \cdot  & 1 & 1 & 1 \\
  \cdot  & \cdot   &  \cdot  & \cdot  & -1 & 0 & 2 \\
  \cdot  & \cdot   & \cdot   &  \cdot  & 1 & -1 & 1
\end{pmatrix}}_{\Xi_{\R_3(n)}} = 
\underbrace{\begin{pmatrix}
 1 & \cdot &\cdot  & \cdot & \cdot &\cdot  & \cdot \\
\cdot  & 1 & \cdot &\cdot  &\cdot  & \cdot &  \cdot\\
\cdot  &\cdot  & 1  & 1 & \cdot &  \cdot& \cdot \\
\cdot & \cdot & -1 & 1 & \cdot &\cdot  &  \cdot\\
\cdot & \cdot & \cdot & \cdot & 1 & 1 & 1 \\
\cdot & \cdot & \cdot & \cdot & -1 & 0 & 2 \\
\cdot  & \cdot& \cdot & \cdot & 1 & -1 & 1
\end{pmatrix}}_{\Xi_{\Sb_0}\oplus \Xi_{\Sb_1}\oplus \Xi_{\Sb_2}\oplus \Xi_{\Sb_3} }
\underbrace{\begin{pmatrix}
 1 & 1 & 1 & 1 & 1 & 1 & 1 \\
 \cdot   & 1 & 0 & 2 & 0 & 1 & 3 \\
 \cdot   &  \cdot  & 1 & 0 & 0 & 1 & 0 \\
 \cdot   &  \cdot  & \cdot   & 1 & 0 & 0 & 3 \\
 \cdot   &  \cdot  & \cdot  &  \cdot  & 1 & 0 & 0 \\
 \cdot  & \cdot   &  \cdot  & \cdot   &  \cdot  & 1 & 0 \\
 \cdot   &  \cdot & \cdot   & \cdot   &  \cdot  &  \cdot  & 1
\end{pmatrix}}_{\F_{\R_3}}
\end{equation*}

\item \emph{The Brauer algebra, $\B_4(n)$:}
\setlength{\arraycolsep}{3pt}
\begin{equation*}
\underbrace{\begin{pmatrix}
1 & 1 & 1 & 1 & 0 & 3 & 1 & 3 \\
  \cdot   & 1 & 1 & 0 & 0 & 2 & 2 & 6 \\
 \cdot   & -1   & 1 & 0 & 0 & -2 & 0 & 6 \\
 \cdot  & \cdot  & \cdot  & 1 & 1 & 1 & 1 & 1 \\
 \cdot  & \cdot  & \cdot  & -1 & 0 & -1 & 1 & 3 \\
 \cdot  & \cdot  & \cdot  & 0 & -1 & 2 & 0 & 2 \\
 \cdot  & \cdot  & \cdot  & 0 & 0 & -1 & -1 & 3 \\
 \cdot  & \cdot  & \cdot  & -1 & 1 & 1 & -1 & 1
\end{pmatrix}}_{\Xi_{\B_4(n)}} = 
\underbrace{\begin{pmatrix}
 1 & \cdot   & \cdot   & \cdot   & \cdot   & \cdot   & \cdot   \\
  \cdot  & 1 & 1 & \cdot   & \cdot   & \cdot   & \cdot   \\
  \cdot  &  -1 & 1  & \cdot   & \cdot   & \cdot   & \cdot   \\
 \cdot  & \cdot   & \cdot  & 1 & 1 & 1 & 1 & 1 \\
 \cdot  & \cdot  & \cdot  & -1 & 0 & -1 & 1 & 3 \\
  \cdot & \cdot  & \cdot  & 0 & -1 & 2 & 0 & 2 \\
 \cdot  & \cdot  & \cdot  & 0 & 0 & -1 & -1 & 3 \\
  \cdot & \cdot  & \cdot  & -1 & 1 & 1 & -1 & 1
\end{pmatrix}}_{\Xi_{\Sb_0}\oplus \Xi_{\Sb_2}\oplus \Xi_{\Sb_4} }
\underbrace{\begin{pmatrix}
1 & 1 & 1 & 1 & 0 & 3 & 1 & 3 \\
  \cdot    & 1 & 0 & 0 & 0 & 2 & 1& 0 \\
  \cdot  &  \cdot    & 1 & 0 & 0 & 0 & 1 & 6 \\
 \cdot   &  \cdot  & \cdot   & 1 & 0 & 0 & 0 & 0  \\
  \cdot  &  \cdot  &  \cdot  &  \cdot  & 1 & 0 & 0 & 0\\
 \cdot   &  \cdot  & \cdot   & \cdot   &  \cdot  & 1 & 0 & 0 \\
  \cdot  &  \cdot  & \cdot  &  \cdot  &   \cdot &  \cdot  & 1 & 0 \\
 \cdot  &  \cdot   &  \cdot  & \cdot   & \cdot   &   \cdot & \cdot   & 1
\end{pmatrix}}_{\F_{\B_4(n)}}
\end{equation*}

\end{enumerate}

\end{examp}

\bibliographystyle{math}

\bibliography{PartAlgReps}

\end{document}